\documentclass{amsproc}
\usepackage{euscript}
\usepackage{cases}
\usepackage{mathrsfs}
\usepackage{bbm}
\usepackage{amssymb}
\usepackage{amsfonts,amsmath,amsxtra,mathdots,mathabx}
\usepackage{color}
\usepackage{hyperref}
\usepackage{tikz}
\usepackage{appendix,upgreek}

\allowdisplaybreaks

\DeclareFontFamily{U}{matha}{\hyphenchar\font45}
\DeclareFontShape{U}{matha}{m}{n}{
	<5> <6> <7> <8> <9> <10> gen * matha
	<10.95> matha10 <12> <14.4> <17.28> <20.74> <24.88> matha12
}{}
\DeclareSymbolFont{matha}{U}{matha}{m}{n}

\DeclareMathSymbol{\Lt}{3}{matha}{"CE}
\DeclareMathSymbol{\Gt}{3}{matha}{"CF}

\DeclareSymbolFont{mathc}{OML}{txmi}{m}{it}
\DeclareMathSymbol{\varuu}{\mathord}{mathc}{117}
\DeclareMathSymbol{\varvv}{\mathord}{mathc}{118}
\DeclareMathSymbol{\varww}{\mathord}{mathc}{119}

\def\SB{\text{\raisebox{- 2 \depth}{\scalebox{1.1}{$ \text{\usefont{U}{BOONDOX-calo}{m}{n}B}   $}}}}

\def\SD{\text{\raisebox{- 2 \depth}{\scalebox{1.1}{$ \text{\usefont{U}{BOONDOX-calo}{m}{n}D} \hspace{0.5pt} $}}}}

\def\SE{\text{\raisebox{- 2 \depth}{\scalebox{1.1}{$ \text{\usefont{U}{BOONDOX-calo}{m}{n}E} \hspace{0.5pt} $}}}}

\def\SB{\text{\raisebox{- 2 \depth}{\scalebox{1.1}{$ \text{\usefont{U}{BOONDOX-calo}{m}{n}B} \hspace{0.5pt} $}}}}

\def\SM{\text{\raisebox{- 2 \depth}{\scalebox{1.1}{$ \text{\usefont{U}{BOONDOX-calo}{m}{n}M} \hspace{0.5pt} $}}}}

\def\SO{\text{\raisebox{- 2 \depth}{\scalebox{1.1}{$ \text{\usefont{U}{BOONDOX-calo}{m}{n}O} \hspace{0.5pt} $}}}}

\def\SC{\text{\raisebox{- 2 \depth}{\scalebox{1.1}{$ \text{\usefont{U}{BOONDOX-calo}{m}{n}C}$}}}}

\def\valpha{\text{\scalebox{0.86}[1.02]{$\alpha$}}}   
\def\vepsilon{\upvarepsilon}
\def\vnu{\text{{\scalebox{0.9}[1]{$\nu$}}}} 
\def\vkappa{\text{{\scalebox{0.86}[1.1]{$\kappa$}}}}

\def\vQ  {\text{\scalebox{0.9}[1]{$Q$}}}

\newcommand{\BH}{{\mathbf {H}}}
 
\newcommand{\BR}{{\mathbf {R}}} 
\newcommand{\BZ}{{\mathbf {Z}}}

\newcommand{\RC}{{\mathrm {C}}}

\newcommand{\GL}{{\mathrm {GL}}}

\newcommand{\ra}{\rightarrow} 
\def\sumx{\sideset{}{^\star}\sum}

\def\nd{\mathrm{d}}

\def\lp {\left (}
\def\rp {\right )}

\def\shskip{\hspace{0.5pt}}

\newcommand{\delete}[1]{}

\theoremstyle{plain}

\newtheorem{thm}{Theorem} \newtheorem{cor}[thm]{Corollary}
\newtheorem{coro}{Corollary}[section]
\newtheorem{lem}{Lemma}[section]
\newtheorem{theorem}{Theorem}[section] 
\newtheorem*{conj}{Conjecture}

\theoremstyle{remark} 
\newtheorem{remark}{Remark}[section] 
\newtheorem{defn}{Definition}[section]

\numberwithin{equation}{section}

\begin{document}
	

	\title[Non-vanishing of Rankin--Selberg $L$-functions]{{On the Effective Non-vanishing of Rankin--Selberg $L$-functions at Special Points}}

	\begin{abstract}
	Let $\vQ (z)$ be a holomorphic Hecke cusp newform of square-free level and $u_j (z)$ traverse an orthonormal basis of Hecke--Maass cusp forms of full level. Let $1/4 + t_j^2$ be the Laplace eigenvalue of $u_j (z)$. In this paper, we prove that there is a constant $ \gamma (\vQ) $ expressed as a certain Euler product associated to $\vQ$ such that at least $ \gamma (\vQ) / 11 $ of the Rankin--Selberg special $L$-values $L (1/2+it_j, \vQ \otimes u_j)$  for $ t_j \leqslant T$ do not vanish as $T \rightarrow \infty$. Further, we show that the non-vanishing proportion is at least $\gamma (\vQ) \cdot (4\mu-3) / (4\mu+7) $ on the short interval $ |t_j - T| \leqslant T^{\mu} $ for any $3/4 <  \mu < 1$. 
	\end{abstract}
	
	\author{Zhi Qi}
	\address{School of Mathematical Sciences\\ Zhejiang University\\Hangzhou, 310027\\China}
	\email{zhi.qi@zju.edu.cn}
	
	\thanks{The author was supported by National Key R\&D Program of China No. 2022YFA1005300.}

	\subjclass[2020]{11M41, 11F72}
	\keywords{Rankin--Selberg $L$-functions, Kuznetsov formula, Vorono\"i formula.}
	
	\maketitle
	
	{\small \tableofcontents}

	\section{Introduction}
	

	In 1985, it was shown in the seminal work of Phillips and Sarnak \cite{Phillips-Sarnak} that a Maass cusp form  $u_j (z)$ for the congruence subgroup $\Gamma_0 (q) \subset \mathrm{SL} (2,  \BZ) $ may be annihilated under the quasi-conformal deformation in Teichm\"uller space generated by the quadratic differential $\vQ (z) \nd z^2$ for a holomorphic cusp form $\vQ (z)$ of weight $4$ for $ \Gamma_0 (q) $, provided that the Rankin--Selberg $L$-function $L (s, \vQ \otimes u_j)$ does not vanish at the special point $s = 1/2+it_j$, where $\lambda_j = 1/4+t_j^2 $ is the Laplace eigenvalue of $u_j (z)$. The Phillips--Sarnak theory reveals the astonishing truth that there are very few Maass cusp forms, possibly none at all,  for generic co-finite $\Gamma \subset \mathrm{SL} (2,  \BZ)$, which is against the Weyl law conjectured in Selberg's 1954 G\"ottingen lectures \cite{Selberg-Weyl}.

	Shortly after Phillips and Sarnak, in the setting that 
	\begin{itemize}
		\item [(i)\,]
 $u_j (z)$ traverses an orthonormal basis of  Hecke--Maass cusp forms for $\Gamma_0 (1)$,
	\item [(ii)] $\vQ (z)$ is a holomorphic Hecke newform for $\Gamma_0(q)$ of  prime level $q$   and weight $2 k $ (not necessarily $2k=4$), 
	\end{itemize}
	Deshouillers and Iwaniec \cite{DI-Nonvanishing} (see also \cite{DIPS-Maass}) used the (Kloosterman) Kuznetsov trace formula to establish the first statistical result:
	\begin{align}
	\text{\small \bf \#} \big\{ j :  0 < t_j \leqslant T, \, L (1/2+it_j , \vQ \otimes u_j) \neq 0
		\big\} \Gt_{\vepsilon} T^{1-\vepsilon}. 
	\end{align}

In 1993,  Luo \cite{Luo-Non-Vanishing} proved an asymptotic formula for the first spectral moment of $L (1/2+it_j , \vQ \otimes u_j)$ and thereby obtained the improvement:
\begin{align}\label{1eq: Luo, 1}
 	\text{\small \bf \#} \big\{ j :  0 < t_j \leqslant T, \, L (1/2+it_j , \vQ \otimes u_j) \neq 0
	 \big\} \Gt_{\vepsilon} T^{2-\vepsilon}. 
\end{align}

In 2001, Luo \cite{Luo-Weyl} (see also his corrigendum in \cite[Appendix]{Luo-2nd-Moment}) used the mollification method \`a la Selberg to attain a positive proportion of non-vanishing $L (1/2+it_j , \vQ \otimes u_j)$: 
\begin{align} \label{1eq: Luo}
 \text{\small \bf \#} \big\{ j :  0 < t_j \leqslant T, \, L (1/2+it_j , \vQ \otimes u_j) \neq 0
	\big\} \Gt  T^{2},  
\end{align} 
and proved that, under certain standard eigenvalue multiplicity assumptions, the Weyl law  is indeed false for generic  $\Gamma $ from the Phillips--Sarnak deformation. 

In the setting as above, but  with $q$ square-free (slightly more general),  we   refine Luo's result \eqref{1eq: Luo} and prove an effective non-vanishing proportion as follows. 


\begin{thm}\label{thm: non-vanishing}
As $T \ra \infty$,   we have 
	\begin{align} \label{1eq: main, short}
	 \frac	{\text{\rm \small \bf \#}    \big\{ j :  0 < t_j \leqslant T, \, L (1/2+it_j , \vQ \otimes u_j)  \neq 0
		\big\} }  {\text{\rm \small \bf \#} \big\{ j :  0 < t_j \leqslant T 
			\big\}}  \geqslant   \gamma (\vQ)  \bigg( \frac 1 {11} - \vepsilon \bigg)    ,  
	\end{align} 
for any $\vepsilon > 0$, where $\gamma (\vQ)$ is the Euler product\shskip {\rm:}
\begin{align}\label{1eq: beta(Q)}
	\begin{aligned}
		 \prod_{p \nmid q}  \bigg(1 - \frac {a(p)^2} {p(p+1) } \bigg)^{-1}  \bigg(1 + \frac {a(p)^2} {p+1 } \bigg)  \bigg( 1 - \frac {a(p)^2 - 2} {p} + \frac 1 {p^2}  \bigg)  \bigg( 1 - \frac {1} { p } \bigg)^2  \cdot  \prod_{p | q} \bigg( 1 - \frac 1 {p^2 } \bigg) ,
	\end{aligned}
\end{align} 
with $a (p)$ the $p$-th Hecke eigenvalue of $\vQ$. 
\end{thm}

Note that for every $p$ the Euler factor of $\gamma (\vQ)$ as in \eqref{1eq: beta(Q)} is $1 + O (1/p^2)$ (due to the Deligne bound). However, it requires a numerical evaluation of  $\gamma (\vQ)$ to see how large (or small) is the non-vanishing proportion. 

Further, we can establish as well an effective non-vanishing result on short intervals $ T - T^{\mu} \leqslant t_j \leqslant T + T^{\mu} $ for $3/4 < \mu < 1$.  

\begin{thm}\label{thm: non-vanishing, short}
Let $ 3/4  < \mu < 1$. As $T \ra \infty$,   we have 
	\begin{align} \label{1eq: main}
	\frac {	\text{\rm \small \bf \#} \big\{ j :  T - T^{\mu} \leqslant t_j \leqslant T + T^{\mu}, \, L (1/2+it_j , \vQ \otimes u_j) \neq 0
		\big\}} {\text{\rm \small \bf \#} \big\{ j :  T - T^{\mu} \leqslant t_j \leqslant T + T^{\mu}
		\big\} } \geqslant \gamma (\vQ)  \bigg( \frac {4\mu-3} {4\mu+7} - \vepsilon \bigg)    ,  
	\end{align} 
	for any $\vepsilon > 0$. 
\end{thm}

Similar to \eqref{1eq: Luo, 1}, we have a weaker  result for smaller $1/3  < \mu \leqslant 3/4$.  

\begin{thm}\label{thm: non-vanishing, short, 2}
	Let $ 1/3  < \mu \leqslant 3/4$. Then for $T$ large,   we have 
	\begin{align} \label{1eq: main, 2}
		  {	\text{\rm \small \bf \#} \big\{ j :  T - T^{\mu} \leqslant t_j \leqslant T + T^{\mu}, \, L (1/2+it_j , \vQ \otimes u_j) \neq 0
			\big\}} \Gt_{\vepsilon} T^{1+\mu - \vepsilon}  ,  
	\end{align} 
	for any $\vepsilon > 0$. 
\end{thm}

	Note that by the Weyl law for $\Gamma_0 (1) $ (see for example \cite[(11.5)]{Iw-Spectral}) we have
	\begin{align*}
		 \text{\rm \small \bf \#} \big\{ j :  0 < t_j \leqslant T 
		 \big\} = \frac 1 {12} T^2 + O (T \log T), 
	\end{align*}
and hence 
\begin{align*}
	\text{\rm \small \bf \#} \big\{ j :  T - T^{\mu} \leqslant t_j \leqslant T+ T^{\mu}
	\big\} = \frac 1 {3}  T^{1+\mu  }  + O (T \log T). 
\end{align*}

\subsection*{Notation}

By $X \Lt Y$ or $X = O (Y)$ we mean that $|X| \leqslant c Y$  for some constant $c  > 0$, and by $X \asymp Y$ we mean that $X \Lt Y$ and $Y \Lt X$. We write $X \Lt_{\valpha, \beta, ...} Y $ or $  X = O_{\valpha, \beta, ...} (Y) $ if the implied constant $c$ depends on $\valpha$, $\beta$, ....  

The notation $x \sim X$ stands for  $ X \leqslant  x \leqslant 2 X $.  


Throughout the paper,  $\vepsilon  $ is arbitrarily small and its value  may differ from one occurrence to another.

\delete{Moreover, Luo \cite{Luo-2nd-Moment}  proved the following asymptotic formula:
\begin{align}
	\sum_{j=1}^{\infty} \omega_j \left|L (s_j, \vQ_2  \otimes u_j) \right|^2 \exp (-t_j / T) = c_0 T^2 \log T + c_1 T + O_{\vQ_2, \vepsilon} (T^{11/6+\vepsilon}),  
\end{align}
where $c_0$ and $c_1$ are constants depending on $\vQ_2$ only.}

\section{Theorems on the Twisted Spectral Moments} 

Proof of the effective non-vanishing results in Theorems \ref{thm: non-vanishing} and \ref{thm: non-vanishing, short} will rely on the asymptotic formulae for the twisted first and second moments of $ L (1/2+it_j , \vQ \otimes u_j) $.  

\subsection{Definitions} 

Let $\{u_j (z)\}_{j=1}^{\infty} $ be an orthonormal basis of (even or odd) Hecke--Maass cusp forms on the modular surface $\mathrm{SL}  (2, \BZ) \backslash \mathbf{H} $.   Let $\lambda_j = s_j (1-s_j)$  
be the Laplace eigenvalue of $u_j (z)$,  
with $s_j = 1/2+ i t_j$ ($t_j > 0$). 
The Fourier expansion of $u_j (z)$ reads:
\begin{align*}
	u_j (x+iy) =   \sqrt{y} \sum_{n \neq 0}  \rho_j (n) K_{i t_j} (2\pi |n| y) e (n x), 
\end{align*}
where as usual $K_{\vnu} (x)$ is the $K$-Bessel function  and $e (x) = \exp (2\pi i x)$. Let $\lambda_j (n)$  be the $n$-th Hecke eigenvalue of $ u_j (z) $. It is well known that $\lambda_j (n)$ is real and $ \rho_j (  n) =   \lambda_j (n) \rho_j (  1)  $ for any $n \geqslant 1$. Define the harmonic weight 
\begin{align*}
	\omega_j = \frac {|\rho_j (1)|^2} {\cosh \pi t_j} .    
\end{align*}
	Let $\vQ (z)$ be a primitive holomorphic newform for $\Gamma_0 (q) \backslash \BH$, with Fourier expansion
\begin{align*} 
	\vQ (z) = \sum_{n=1}^{\infty} a(n) n^{(2k-1)/2} e (n z), \qquad a(1) = 1.   
\end{align*}

Let $T, \varPi$ be large parameters such that $T^{\vepsilon} \leqslant \varPi \leqslant T^{1-\vepsilon}$.  Define the smoothly weighted twisted moments on the cuspidal spectrum: 
\begin{align}\label{2eq: moments M1}
	\SC_{1}  (m  ) = \sum_{j = 1}^{\infty} \omega_j \lambda_j (m) m^{-it_j}  L (s_j,  \vQ \otimes u_j )    \exp \left(  - \frac {(t_j - T)^2}  {\varPi^2} \right) , 
\end{align}
\begin{align}\label{2eq: moments M2}
	\begin{aligned}
		\SC_{2}  (m_1, m_2) = \mathrm{Re}  \sum_{j = 1}^{\infty}   \omega_j   \frac {\lambda_j ( m_1  ) \lambda_j ( m_2  )  } {(m_1/m_2)^{it} }  |L (s_j,  \vQ \otimes u_j )|^2     \exp \left(  - \frac {(t_j - T)^2}  {\varPi^2} \right)  . 
	\end{aligned}
\end{align}  

\subsection{Main Theorems} 

\begin{thm}\label{thm: C1(m)} 
Let  $T^{\vepsilon} \leqslant \varPi \leqslant T^{1-\vepsilon}$. We have
	\begin{align}\label{2eq: C1(m)}
		\SC_{1}  (m  ) = \frac {2} {\pi\sqrt{\pi}} \varPi T \cdot {\delta ({m, 1})}  + O_{\vQ, \vepsilon} \big( \sqrt[3]{m^2 (T+\varPi^2)} T^{1+\vepsilon} \big), 
	\end{align}
where $\delta (m ,1)$ is the Kronecker $\delta$ that detects $m = 1$. 
\end{thm}

\begin{thm}\label{thm: C2(m)} 
Let $T^{1/2} \leqslant \varPi \leqslant T^{1-\vepsilon}$. Assume   $m_1$, $m_2$ square-free and $(m_1 m_2, q) = 1$.  Set \begin{align}
	m = (m_1, m_2), \qquad r = \frac {m_1m_2}  {m^2}. 
\end{align}  
 Then 
	\begin{align}\label{2eq: C2(m)}
		\begin{aligned}
			\SC_{2}  (m_1, m_2 ) =\frac {4 a(r)} {\pi \sqrt{\pi r} } \varPi T \bigg(  \bigg( & \gamma_1 \log \frac {T} {\sqrt{r}}  + \gamma_{0}'  \bigg) \SB (m)  + \gamma_1 \SB ' (m)   \bigg) & \\
			& +   O_{\vQ, \vepsilon} \bigg(  T^{\vepsilon} \bigg(  \frac {  \varPi^3  } {T \sqrt{ r } } +  {\frac { T^2 \sqrt{T}  } { \, {\varPi} } } {\textstyle \sqrt{m_1^3+m_2^3}}    \bigg) \bigg)    , 
		\end{aligned}
	\end{align}
where 
the constants $\gamma_0' = \gamma_0' (\vQ) $, $\gamma_1 = \gamma_1 (\vQ) $, and   $\SB (m)  $, $\SB ' (m)$ are defined explicitly in Definition \ref{defn: constants, 1} and Lemma \ref{lem: main term}. 
\end{thm}


\delete{
There are two major differences between this work and the previous works of Deshouillers, Iwaniec, and Luo \cite{DI-Nonvanishing,Luo-Non-Vanishing,Luo-Weyl}:
\begin{itemize}
	\item [(1)] the spectral Kuznetsov formula (\cite[Theorem 1]{Kuznetsov}) is used instead of the Kloosterman Kuznetsov formula (\cite[Theorem 2]{Kuznetsov}) with careful analysis of  the related Bessel integral, 

\item  [(2)] the Vorono\"i summation formula is applied together with the Wilton bound. 
\end{itemize}}

\subsection{Moments with no Twist} 

The main term in \eqref{2eq: C1(m)} survives only for $m=1$, and it dominates the error term   when $T^{1/3+\vepsilon} \leqslant \varPi \leqslant T^{1-\vepsilon}$.  

\begin{cor}\label{cor: 1st moment}
Let  $T^{1/3+\vepsilon} \leqslant \varPi \leqslant T^{1-\vepsilon}$. 	We have 
	\begin{align}\label{2eq: C1(1)}
		\SC_{1} (1)  = \frac {2} {\pi\sqrt{\pi}} \varPi T  + O_{\vQ, \vepsilon} \big(\sqrt[3]{T + \varPi^2} T^{1+\vepsilon} \big). 
	\end{align} 
\end{cor} 

The main term in \eqref{2eq: C2(m)} simplifies  if $(m_1, m_2) = (1, 1)$   (or $(p, 1)$ for $p$ prime as in Corollary \ref{cor: C2(p)} below) since $\SB (1) = 1$ and $\SB' (1) = 0$.  Note that in this case the error term is inferior to the main term only for $T^{3/4+\vepsilon} \leqslant \varPi \leqslant T^{1-\vepsilon}$. 

\begin{cor}\label{cor: 2nd moment}
Let $T^{3/4+\vepsilon} \leqslant \varPi \leqslant T^{1-\vepsilon}$.	We have 
	\begin{align}\label{2eq: C2(1)}
		\SC_{2}  (1,1) = \frac {4   } {\pi \sqrt{\pi  }} \varPi T \big(\gamma_1 \log   {T}   + \gamma_0'   \big) + O_{\vQ, \vepsilon} \bigg(T^{\vepsilon} \bigg(\frac {\varPi^3} {T} + \frac {T^{2}\sqrt{T}} {\varPi}  \bigg)\bigg). 
	\end{align}
\end{cor} 

For  $ T^{\vepsilon} \leqslant \varPi \leqslant T^{3/4+\vepsilon} $, however, we have the mean Lindel\"of bound as a simple consequence of the large sieve inequality of Luo in \cite{Luo-LS}. 

\begin{thm} \label{thm: Lindelof}
	Let $T^{\vepsilon} \leqslant \varPi \leqslant T^{3/4+\vepsilon}$.	We have 
	\begin{align}
		\SC_{2}  (1,1) = O_{\vQ, \vepsilon} (\varPi T^{1+\vepsilon}). 
	\end{align} 
\end{thm} 

It is natural to propose the following conjecture. 

\begin{conj}
	For any $ T^{\vepsilon} \leqslant \varPi \leqslant T^{1-\vepsilon} $, we have
	\begin{align*}
	\SC_{1} (1)  = \frac {2} {\pi\sqrt{\pi}} \varPi T  + o (\varPi T), \qquad 	 \SC_{2}  (1,1) = \frac {4   } {\pi \sqrt{\pi  }} \varPi T \big(\gamma_1 \log   {T}   + \gamma_0'   \big) + o (\varPi T).  
	\end{align*}
\end{conj}

\subsection{Proof of Theorem \ref{thm: non-vanishing, short, 2}}  Theorem \ref{thm: non-vanishing, short, 2} is an easy consequence of Corollary \ref{cor: 1st moment} and Theorem \ref{thm: Lindelof}, along with 
the bounds of Iwaniec, Hoffstein, and Lockhart \cite{Iwaniec-L(1),HL-L(1)}:
\begin{align}\label{3eq: omegaj}
	t_j^{-\vepsilon} \Lt \omega_j \Lt t_j^{\vepsilon}. 
\end{align}
For $1/3 < \mu < 1$ let $ \varPi = T^{\mu}/ \log T $.   By Corollary \ref{cor: 1st moment},  
\begin{align*}
\varPi^2 T^2  & \Lt \Bigg(   \sum_{ |t_j - T| \leqslant \varPi \log T } \omega_j | L (s_j,  \vQ \otimes u_j )| \Bigg)^2 .
\end{align*}
By Cauchy and  Theorem \ref{thm: Lindelof},  the right-hand side is bounded by 
\begin{align*}
	 \Bigg(\mathop{\sum_{ |t_j - T| \leqslant \varPi \log T }}_{L (s_j,  \vQ \shskip \otimes u_j ) \neq 0} \omega_j  \Bigg)  \Bigg(\mathop{\sum_{ |t_j - T| \leqslant \varPi \log T }}  \omega_j |L (s_j,  \vQ \otimes u_j )|^2  \Bigg)  \Lt \varPi  T^{1+\vepsilon} \mathop{\sum_{ |t_j - T| \leqslant \varPi \log T }}_{L (s_j,  \vQ \shskip \otimes u_j ) \neq 0} \omega_j. 
\end{align*}
Thus \eqref{1eq: main, 2} follows immediately due to the upper bound in \eqref{3eq: omegaj}. 

\subsection{Aside: the Determination Problem} 


\begin{cor}\label{cor: C2(p)}
Let $T^{3/4+\vepsilon} \leqslant \varPi \leqslant T^{1-\vepsilon}$.	Let $p$ be a prime with $p \nmid q$. We have 
	\begin{align}\label{2eq: C2(p)}
	\SC_{2}  (p, 1 ) = \frac {4 a (p) } {\pi \sqrt{\pi p}} \varPi T \bigg(\gamma_1 \log \frac {T} {\sqrt{p}} + \gamma_0'   \bigg) + O_{\vQ, p, \vepsilon } \bigg(T^{\vepsilon} \bigg( \frac {\varPi^3} {T} + \frac {T^{2}\sqrt{T}} {\varPi}  \bigg)\bigg) .
	\end{align}
\end{cor}

However, it is curious that the asymptotic for $ \SC_{1}  (p)  $  in \eqref{2eq: C1(m)} yields no information on $\vQ$ as it has no main term at all.

\begin{thm}\label{thm: determination}
	Let $\vQ$, $\vQ'$ be holomorphic Hecke newforms, respectively, of  square-free level $q$, $q'$  and even weight $2 k $, $2k'$. If there is a constant $c > 0$ such that 
	\begin{align}
		|L (s_j , \vQ \otimes u_j)| = c \,  | L (s_j , \vQ' \otimes u_j) |
	\end{align}
for all Hecke--Maass cusp forms $u_j$, then we have $ \vQ = \vQ'$. 
\end{thm}

\begin{proof}
	 Let us denote $a' (p)$, $\gamma_1'$, and $\gamma_{0}''$ for the $p$-th Hecke eigenvalue and the constants  associated to $Q'$ as in \eqref{2eq: C2(1)} and \eqref{2eq: C2(p)}. It follows from \eqref{2eq: C2(1)} that $\gamma_1 = c \gamma_1'$ and $ \gamma_{0}' = c \gamma_{0}''$, and hence from  \eqref{2eq: C2(p)} that $a (p) = a' (p)$ for any $p \nmid q q'$.  Consequently, we have $\vQ = \vQ'$ by the strong multiplicity one theorem. 
\end{proof}

Since the pioneer work of Luo and Ramakrishnan \cite{Luo-Ram-GL(2)xGL(1)}, many authors have studied the determination problem (usually) by  Rankin--Selberg central $L$-values at $s =1/2$ in the $\GL (2) \times \GL(1)$, $\GL (2) \times \GL (2) $, $\GL (3) \times \GL(1)$, or $\GL (3) \times \GL(2)$ settings. See for example \cite{CD-GL(3)xGL(1),GHS-GL(2)xGL(2),Liu-GL(3)xGL(2),Luo-Ram-GL(2)xGL(1),Luo-GL(2)xGL(2),Munshi-GL(2)xGL(1)-Effective}.  Note that  their results are all based on the asymptotic formulae for the twisted {\it first} moments. 

\subsection{Comparison with Luo's Results} 
As Luo \cite{Luo-Weyl} treated the mollified moments directly,  so a comparison of strength of our (twisted) results with his  is not quite transparent, especially for the second moment. However, his method enabled him to establish  the following asymptotic formulae in \cite{Luo-Non-Vanishing,Luo-2nd-Moment}:
\begin{align}\label{2eq: Luo 1st}
	\sum_{j = 1}^{\infty} \omega_j   L (s_j,  \vQ \otimes u_j )  \exp \left(  -  t_j / T \right) & = \frac {2 } {\pi^2} T^{2}  + O (T^{7/4} \log^9 T),  \\
	\label{2eq: Luo 2nd}
	\sum_{j = 1}^{\infty} \omega_j    |L (s_j,  \vQ \otimes u_j )|^2  \exp \left(  -  t_j / T \right) & = a_1 T^2 \log T + a_0 T^2 + O (T^{ 11/6+\vepsilon } ),
\end{align}
where $a_1 = 4 \gamma_1 / \pi^2$, while the method in this paper would yield 
the following asymptotics for the unsmoothed spectral moments. 
\delete{The $2$ is from partial summation: 
\begin{align*}
	\int_0^{\infty} t^2 \exp (-t ) \nd t = 2 . 
\end{align*}}

\begin{thm}\label{cor: unsmoothed}
	We have 
	\begin{align}\label{2eq: 1st moment}
		\sum_{t_j \leqslant T} \omega_j L (s_j,  \vQ \otimes u_j ) = \frac {1} {\pi^2} T^{2} + O (T^{5/3+\vepsilon}), 
	\end{align} 
\begin{align}\label{2eq: 2nd moment}
	\sum_{t_j \leqslant T} \omega_j |L (s_j,  \vQ \otimes u_j )|^2 = \frac {2\gamma_1} {\pi^2}  T^2 \log T + \frac {2\gamma_0' - \gamma_1 } {\pi^2 } T^2  + O (T^{7/4+\vepsilon}). 
\end{align}
\end{thm}

The situation might be illustrated by the similar example of the fourth moment of the Riemann $\zeta$ function: Heath-Brown \cite{Heath-Brown-4th} compared his asymptotic formula 
\begin{align*}
	\int_0^{T} |\zeta (1/2+it)|^4 \nd t = a_4 T \log^4 T + a_3 T \log^3 T + \cdots + a_0 T + O(T^{7/8+\vepsilon}), 
\end{align*}
 with Atkinson's 
 \begin{align*}
 	\int_0^{\infty} |\zeta (1/2+it)|^4 \exp (-t/T) \nd t = a_4 T \log^4 T + a_3' T \log^3 T + \cdots + a_0' T + O(T^{8/9+\vepsilon}), 
 \end{align*}
where  the change  $\delta = 1 / T $ has been made to \cite{Atkinson-4th}, and   he remarked that Atkinson's result yields no information about 
\begin{align*}
	\int_0^{T} |\zeta (1/2+it)|^4 \nd t 
\end{align*}
beyond 
\begin{align*}
	\int_0^{T} |\zeta (1/2+it)|^4 \nd t \sim \frac 1 {2\pi^2}  T \log^4 T .
\end{align*}
Note that Heath-Brown's $ O(T^{7/8+\vepsilon}) $ was improved into $O (T^{2/3} \log^8 T )$ by Ivi\'c and Motohashi \cite{IM-4th-Moment,Motohashi-Riemann} (a slightly weaker bound $O (T^{2/3+\vepsilon})$ was obtained in \cite{Zavorotnyi-4th}), while Atkinson's $O(T^{8/9+\vepsilon})$ was  improved drastically into $O (T^{1/2}   )$ by Ivi\'c \cite{Ivic-4th}, which is best possible as indicated in \cite{Ivic-4th-2}.

	\section{Preliminaries}

\delete{	

	Let $e (x) = \exp (2\pi i x)$. 	For integers $  m,  n , \rho$ and   $c  \geqslant 1$,      define 
	\begin{align}\label{1eq: defn Kloosterman}
		S   (m, n ; c ) = \sumx_{   \valpha      (\mathrm{mod} \, c) } e \bigg(   \frac {  \valpha     m +   \widebar{\valpha    } n} {c} \bigg) ,
	\end{align} 
	\begin{align} \label{2eq: defn V}
		V_{\rho} (m, n; c) =	\mathop{\sum_{\valpha     (\mathrm{mod}\, c)}}_{ (\valpha     (\rho-\valpha    ), c) = 1 }  e \bigg(   \frac {  \widebar{\valpha    } m +  \overline{\rho - \valpha    } n } {c} \bigg),
	\end{align}
	where the $\star$ indicates the condition $(\valpha    , c) = 1$ and $ \widebar{\valpha    }$ is given by $\valpha     \widebar{\valpha    }  \equiv 1 (\mathrm{mod} \, c)$. The definition of $V_{\rho} (m, n; c)$ is essentially from Iwaniec--Li  \cite[(2.17)]{Iwaniec-Li-Ortho}. Note that the Kloosterman sum $S (m,n;c)$ is real valued. We have Weil's  bound:
	\begin{align}
		\label{2eq: Weil}
		S (m, n; c)   \Lt \tau (c) \sqrt{ (m, n, c) } \sqrt{c}, 
	\end{align} 
	where as usual $\tau (c) = \tau_0 (c)$ is the number of divisors of $c$.  Moreover, we have Luo's identity in \cite[\S 3]{Luo-LS} (see also \cite[(2.16)]{Iwaniec-Li-Ortho}): 
		\begin{align}\label{2eq: S = V}
			S (m, n; c)  e\Big(\frac {m+n} {c} \Big) = \sum_{r \rho   = c } V_{\rho} (m, n; r). 
		\end{align} 
}
	
	\subsection{Basics of Bessel functions}

The Bessel function $J_{\vnu}(x)$ 
arises in both the Kuznetsov trace formula and the Vorono\"i summation formula.  

For $\mathrm{Re}  (\vnu)  > -   1 / 2$, we have Poisson's integral representation
as in \cite[3.3 (5)]{Watson}: 
\begin{align}\label{3eq: Poisson, J}
	J_{\vnu} (x) = \frac {  ( x/2  )^{\vnu}} {\sqrt{\pi} \Gamma   (\vnu +     1 / 2  )  } \int_0^{\frac 1 2 \pi}   \cos (x \cos \theta ) \sin^{2 \vnu} \theta \,  {\nd}   \theta. 
\end{align}  

For $|\mathrm{Re} (\nu)| < 1$, we have the Mehler--Sonine integral as in \cite[6.21 (12)]{Watson}: 
\begin{align}\label{3eq: integral Bessel}
	J_{  \vnu} (x) = \frac 2 {\pi} \int_0^{\infty} \sin (x \cosh r - \vnu \pi/2) \cosh ( \vnu r) \nd r. 
\end{align}

Moreover, we have the expression (see  \cite[\S 16.12, 16.3, 17.5]{Whittaker-Watson} or \cite[\S 7.2]{Watson}): 
 \begin{equation} \label{2eq: Bessel, Whittaker}
 	J_{\vnu}(x) = \frac {1} {\sqrt {2 \pi x}} \big( \exp ({i x}) W_{\vnu,\scriptscriptstyle +}(x) + \exp ({- i x}) W_{ \vnu, \scriptscriptstyle -} (x) \big), 
 \end{equation}
 with 
 \begin{equation}
 	\label{2eq: bounds for Whittaker functions}
 	x^{n} W_{\vnu, \scriptscriptstyle \pm}^{( n )} (x) \Lt_{  \vnu, n } 1, \qquad x \Gt 1, 
 \end{equation}
while it follows from \eqref{3eq: Poisson, J} that 
\begin{align}\label{3eq: bounds for Bessel, x<1}
	 x^{n - \vnu} J_{\vnu}^{(n)} (x) \Lt_{  \vnu, n }  1  , \qquad x \Lt 1 .
\end{align}

\subsection{Kloosterman and Ramanujan Sums}

Let $e (x) = \exp (2\pi i x)$. 	For integers $  m,  n $ and   $c  \geqslant 1$,      define  the Kloosterman sum
\begin{align}\label{3eq: Kloosterman}
	S   (m, n ; c ) = \sumx_{   \valpha      (\mathrm{mod} \, c) } e \bigg(   \frac {  \valpha     m +   \widebar{\valpha    } n} {c} \bigg) ,
\end{align} 
where the $\star$ indicates the condition $(\valpha    , c) = 1$ and $ \widebar{\valpha    }$ is given by $\valpha     \widebar{\valpha    }  \equiv 1 (\mathrm{mod} \, c)$. 
We have the Weil bound: 
\begin{align}
	\label{3eq: Weil}
	S (m, n; c)   \Lt \tau (c) \sqrt{ (m, n, c) } \sqrt{c} ,  
\end{align} 
where as usual $\tau (c)  $ is the number of divisors of $c$. For the Ramanujan sum $S (m, 0; c)$ we have
\begin{align}\label{3eq: Ramanujan}
	S (m, 0; c) = \sum_{d| (c, m)} d \mu \Big(   \frac {c} {d}  \Big).  
\end{align}

	\subsection{Maass Cusp Forms and Kuznetsov Trace Formula} 
	
As before, let $ \{u_j (z)\}_{j=1}^{\infty}$ be an orthonormal basis of Hecke--Maass cusp forms for the cuspidal space $L_c^2 (\Gamma_0 (1) \backslash \BH)$. For each $u_j (z)  $ with Laplacian eigenvalue $\lambda_j = 1 / 4 + t_j^2$ ($t_j > 0$), it has Fourier expansion of the form
	\begin{align*}
		u_j(z)=   \sqrt{ y}  \sum_{n\neq 0}\rho_j(n) K_{it_j}(2\pi|n|y)e(nx) .
	\end{align*}
As usual, write $s_j = 1/2+ i t_j$ so that $\lambda_j = s_j (1-s_j)$. 	Let $\lambda_j (n)$ ($n \geq 1$) be its Hecke eigenvalues. It is well known that $ \lambda_j (n)$ are all real.  We may assume  $ u_j (z)$ is even or odd in the sense that $ u_j (- \widebar{z}) = \pm u_j  (z)$.     
Then   $\rho_j (\pm n) =   \allowbreak \rho_j (\pm 1)  \lambda_j (n) $, while $\rho_j (-1) = \pm  \rho_j (1)$. Moreover, we have the Hecke relation
\begin{align}\label{3eq: Hecke}
	\lambda_j (m) \lambda_j (n) = \sum_{d | (m, n)} \lambda_j (mn/d^2).  
\end{align}
Define the divisor function (as the Fourier coefficients of the Eisenstein series) 
\begin{align}\label{3eq: divisor function}
	\tau_{\vnu} (n) = \tau_{- \vnu} (n) = \sum_{ a b = n} (a/b)^{  \vnu}   . 
\end{align}
Similar to \eqref{3eq: Hecke}, we have
\begin{align}
	\label{3eq: Hecke, tau}
	\tau_{\vnu} (m)  \tau_{\vnu} (n) = \sum_{d | (m, n)}  \tau_{\vnu} (mn/d^2). 
\end{align}
 
Now we state the spectral Kuznetsov trace formula as in \cite[Theorem 1]{Kuznetsov}. 

	\begin{lem}\label{lem: Kuznetsov}
	Let  $h (t)$ be an even function satisfying the conditions{\hspace{0.5pt}\rm:}
	\begin{enumerate} 
		\item[{\rm (i)\,}] $h (t)$ is holomorphic in  $|\operatorname{Im}(t)|\leq {1}/{2}+\vepsilon$,
		\item[{\rm (ii)}] $h(t)\Lt (|t|+1)^{-2-\vepsilon}$ in the above strip. 
	\end{enumerate}
	Then for $m, n \geq  1$ 	we have the    identity{\hspace{0.5pt}\rm:} 
	\begin{equation}\label{2eq: Kuznetsov}
		\begin{split}
			\sum_{j = 1}^{\infty}  \omega_j h(t_j)   \lambda_j(m)   \lambda_j(n)  + \frac{1}{\pi} & \int_{-\infty}^{\infty} \omega(t) h(t)  \tau_{it}(m) \tau_{it} (n) \nd t\\
			&= \delta ({m, n}) \cdot H +   \sum_{c= 1}^{\infty} \frac{S(m, n;c)}{c} H\bigg(\frac{4\pi\sqrt{m n}}{c}\bigg),
		\end{split}
	\end{equation}
	where  $\delta ({m, n})$ is the Kronecker $\delta$-symbol, 
	\begin{equation}\label{3eq: omega}
		\omega_j=\frac{ |\rho_j(1)|^2}{\cosh(\pi t_j)}, \qquad \omega (t) = \frac {1} {|\zeta(1+2it)|^2},
	\end{equation}
	\begin{align}\label{3eq: integrals} 
			  H = \frac{1}{\pi^2} \int_{-\infty}^{\infty} h(t)\tanh(\pi t) t \nd t , \qquad  
			  H (x)= \frac {2i} {\pi}   \int_{-\infty}^{\infty} h (t ) J_{2it} (x) \frac {t \nd t} {\cosh (\pi t) }. 
		\end{align}  
\end{lem}

\subsection{Poisson Summation Formula}

\begin{lem}\label{lem: Poisson}
 Let $\phi ( n | c)$ be an arithmetic function of period $c$. Then for $F (x) \in C_c^{\infty} \allowbreak (-\infty, \infty)$ we have 
	\begin{align}\label{4eq: Poisson Cor}
		\sum_{n =-\infty}^{\infty}  \phi (n| c)  F (n) = \frac 1 {c} \sum_{n =-\infty}^{\infty} \hat {\phi }  (n|c) \hat{F} (n/   c), 
	\end{align}
	where $\hat{F} (y)$ is   the   Fourier transform of $ F (x) $ defined by  
	\begin{align}\label{4eq: Fourier}
		\hat {F} (y) = \int_{-\infty}^{\infty} F (x) e (-xy) \nd x , 
	\end{align}
and
	\begin{align}
		\hat {\phi }  (n|c) =  \sum_{\valpha (\mathrm{mod}\, c) }  \phi (\valpha|c) e \Big(     \frac {   \valpha {n}  } {c} \Big) .  
	\end{align}
\end{lem}

	\subsection{Holomorphic Cusp Forms and Vorono\"i Summation Formula} 
	
	Let $\vQ (z)$ be a holomorphic newform for $\Gamma_0 (q) \backslash \BH$.  Suppose that $ \vQ (z) $ has trivial nebentypus  and Fourier expansion
	\begin{align*} 
		\vQ (z) = \sum_{n=1}^{\infty} a(n) n^{(2k-1)/2} e (n z),  
	\end{align*}
Hecke-normalized (that is, primitive) so that $ a (1) = 1$. We have the Hecke relation 
\begin{align}\label{3eq: Hecke, Q}
	a (mn) = \sum_{d | (m, n)} \mu (d) \varepsilon_q  (d)  a(m/d) a(n/d),  
\end{align}
where $\varepsilon_q $ is the trivial character modulo $q$. 
Moreover, we have the celebrated Deligne bound  (\cite{Deligne-I}):
\begin{align}\label{3eq: Deligne}
	 |a(n)| \leqslant \tau (n) . 
\end{align} 
 Subsequently, let $q$ be square-free. For $d| q$, we have \begin{align}\label{3eq: a(d)}
 	 a (d)^2 = \frac 1 {d}. 
 \end{align}
According to \cite{Atkin-Lehner-1}, for  $d  | q $ the Atkin--Lehner eigenvalue $\eta_{\vQ} (d)$ is equal to $ \mu (d)  a (d) \sqrt{d}  = \pm 1$.

	Now we state in our setting the Vorono\"i summation formula as in \cite[Theorem A.4]{KMV}. 

 \begin{lem}\label{lem: Voronoi}
Assume that the level $q$ is square-free. Let $\valpha, c $ be integers with $c  \geqslant 1$ and $(\valpha, c) = 1$.  Set $q_c = q / (c, q)$.  Then for  $F (x) \in C_c^{\infty} (0, \infty)$ we have 
 	 \begin{align}\label{2eq: Voronoi, GL(2)} 
 	 	\sum_{n=1}^{\infty}  a  (n)    e\Big(\frac{{\valpha} n}{c}\Big) F ( {n}  )   =     \frac {\mu (q_c) a (q_c) }  {c   }     \sum_{n=1}^{\infty} a (n) e \left(  - \frac{ \overline{\valpha q_c}  n}{c} \right) \check{F}  \bigg(    \frac{  n }{c^2 q_c}   \bigg) ,
 	 \end{align} 
  where $\check{F} (y) $ is the Hankel transform of $F (x)$ defined by 
  \begin{align}
  	\check{F} (y) = 2\pi i^{2k} \int_0^{\infty} F(x) J_{2k-1} (4\pi \sqrt{xy}) \nd x. 
  \end{align}
 \end{lem}

Moreover, we shall need the Wilton bound as in \cite[Theorem 3.1]{Iwaniec-Topics}.

\begin{lem}\label{lem: Wilton}
	For any real $\gamma$ and $N \geqslant 1$ we have
	\begin{align}
		\sum_{n \leqslant N} a (n) e (\gamma n) \Lt_{\vQ} \sqrt{N} \log 2 N, 
	\end{align}
where the implied constant depends only on the form $\vQ$ {\rm(}not on $\gamma${\rm)}. 
\end{lem}

\subsection{Hecke $L$-function {$L (s, {\protect\vQ})$}}   Define 
\begin{align}\label{3eq: L(s,Q)}
	L (s, \vQ) = \sum_{n    =1}^{\infty}\frac{ a (n)  }{n^s} ,
\end{align}
for $\mathrm{Re} (s) > 1$. It is known that $L (s, \vQ)$ has analytic continuation to the whole complex plane and satisfies the functional equation
\begin{align}\label{3eq: Lambda (Q)}
	\Lambda (s, \vQ) = \epsilon_{\vQ}   \Lambda (1-s, \vQ),   
\end{align} 
with
\begin{align}\label{3eq: FE, Q}
	\Lambda (s, \vQ) = \lp   {\sqrt{q}} / {2\pi } \rp^{s} \Gamma (s+k-1/2) L (s, \vQ) , \qquad \epsilon_{\vQ} = i^{2k} \mu (q) a(q) \sqrt{q} = \pm 1.
\end{align} 
The next lemma will be used to estimate the Eisenstein contributions. 

\begin{lem}\label{lem: L(Q)}
	Let $ 1 \leqslant H \leqslant T/3 $. We have
	\begin{align}\label{3eq: L(Q) 1}
		\int_{T- H}^{T+ H}  | L (1/2 +it , \vQ) | \nd t \Lt_{\vQ} {\textstyle \sqrt{     ( H^2 + H T^{2/3}  )   \log T }} ,   
	\end{align} 
	\begin{align}\label{3eq: L(Q) 2}
		\int_{T- H}^{T+ H}  | L (1/2 +it , \vQ) |^2 \nd t \Lt_{\vQ}   {  ( H + T^{2/3}  )} \log T. 
	\end{align} 
\end{lem}

\begin{proof}
By Cauchy--Schwarz, \eqref{3eq: L(Q) 1} follows directly from \eqref{3eq: L(Q) 2}, while the latter is deducible easily from the asymptotic formula of Good \cite{Good-Weyl}: 
\begin{align*}
	 \int_{0}^{T}  | L (1/2 +it , \vQ) |^2 \nd t = c_1   T \log T + c_0 T + O_{\vQ} \big(T^{2/3} \log^{2/3} T \big),  
\end{align*}
where $ c_0  $ and $c_1$ are constants depending on the form $\vQ$ (for instance, $c_1 $ is twice the residue at $s = 1$ of  the Dirichlet series ($\zeta_q (2s)$ excluded) in \eqref{3eq: L(s, QQ)} below).  
\end{proof}

\subsection{Rankin--Selberg $L$-function {$L (s, {\protect\vQ} \otimes {\protect\vQ})$}} Define
\begin{align}\label{3eq: L(s, QQ)}
	L (s, \vQ \otimes \vQ ) = \zeta_q (2s) \sum_{n=1}^{\infty} \frac {a (n)^2 } {n^s},  
\end{align}
 for $\mathrm{Re} (s) > 1$, with 
 \begin{align}\label{3eq: zeta(s)}
 	\zeta_q (s) = \sum_{n    =1}^{\infty} \frac {\varepsilon_q (n)} {n^s} = \prod_{p \shskip \nmid q} \lp 1 - \frac 1 {p^s} \rp^{-1}; 
 \end{align}
it is known that $L (s, \vQ \otimes \vQ )$ has analytic continuation on the complex plane except for a simple pole at $s = 1$. Moreover, we have the Euler product (see \cite[\S 4]{KMV}): 
\begin{align}\label{3eq: L(QQ), Euler}
	L (s, \vQ \otimes \vQ ) = \prod_{p \nmid q} \bigg(1 - \frac {a(p)^2 - 2} {p^s} + \frac 1 {p^{2s}}   \bigg)^{-1} \lp 1 - \frac 1 {p^s} \rp^{-2} \cdot  \prod_{p | q} \bigg( 1 - \frac 1 {p^{1+s}}\bigg)^{-1} . 
\end{align}
\begin{defn} \label{defn: constants, 1}
	Define the constants $\gamma_{0} = \gamma_0(\vQ)$ and $\gamma_{1} = \gamma_{1} (\vQ)$ by the Laurent expansion at $s=1$: 
\begin{align}\label{3eq: Laurent}
	L (s, \vQ \otimes \vQ ) = \frac {\gamma_{1}  } {s-1} + \gamma_0   + O(|s-1|) ; 
\end{align}
more explicitly, 
\begin{align}	\label{3eq: gamma1}
\gamma_1 = \prod_{p \nmid q} \lp 1 - \frac {a(p)^2 - 2} {p} + \frac 1 {p^2} \rp^{-1} \lp 1 - \frac 1 {p} \rp^{-1} \cdot \prod_{p | q} \bigg( 1 + \frac 1 {p } \bigg)^{-1}. 
\end{align}
\end{defn}
\delete{By the well-known Rankin--Selberg method, 
\begin{align}
	\sum_{n \leqslant X}    {\varepsilon_q (n) a(n)^2} { } = \frac {\gamma_1   } {\zeta (2)} X + O (X^{3/5+\vepsilon});  
\end{align}
here we have extracted the Euler product over $p | q$ (it is the product of $ 1/ (1 - 1/ p^{1+s}) $ in view of \eqref{3eq: Hecke, Q} and  \eqref{3eq: a(d)}) out of the Dirichlet series in \eqref{3eq: zeta(s)}. }

\subsection{Rankin--Selberg $L$-function {$L (s, {\protect\vQ} \otimes u_j)$}}  
Define $L (s, \vQ \otimes u_j) $ by 
\begin{align}\label{3eq: RS L(s)}
	L (s, \vQ \otimes u_j) =  \zeta_q (2s)  \sum_{n    =1}^{\infty}\frac{ a (n) \lambda_j (n    )}{n^s},  
\end{align}
for $\mathrm{Re} (s) > 1$, and 
 by analytic continuation on the complex plane. 
Define
\begin{align}\label{3eq: RS}
	\Lambda ( s , \vQ \otimes u_j) =    ( { {q}} / {4\pi^2 }  )^{s} \Gamma  ( s+k-s_j ) \Gamma  ( s+k-\overline{s}_j ) L(s, \vQ \otimes u_j). 
\end{align}
Then the    functional equation for $ L(s,\vQ \otimes u_j) $ reads
\begin{align}\label{3eq: RS, FE}
\Lambda ( s , \vQ \otimes u_j) =  \Lambda ( 1- s , \vQ \otimes u_j) . 
\end{align} 

Moreover, in parallel to \eqref{3eq: RS L(s)}, if we replace $\lambda_j (n)$ by  $  \tau_{ it} (n) $ as defined in \eqref{3eq: divisor function}, then we have the Ramanujan identity
\begin{align}\label{3eq: RS, Eis}
	 L (s +it , \vQ) L (s - it, \vQ ) = \zeta_q (2s ) \sum_{n    =1}^{\infty}\frac{ a (n) \tau_{it} (n    )}{n^{s }} .
\end{align} 
Similar to   \eqref{3eq: RS, FE}, 
it follows from \eqref{3eq: Lambda (Q)} and \eqref{3eq: FE, Q} that
\begin{align}\label{3eq: RS, FE, Eis}
\Lambda (s+it, \vQ) \Lambda (s-it, \vQ) =  \Lambda (1-s+it, \vQ) \Lambda (1-s-it, \vQ) . 
\end{align}

By expanding $\zeta_q (2s)$  (see \eqref{3eq: zeta(s)}) in \eqref{3eq: RS L(s)} and \eqref{3eq: RS, Eis}, we rewrite 
\begin{align}\label{3eq: RS, 2}
	L (s, \vQ \otimes u_j) =    \sum_{n    =1}^{\infty}\frac{ \valpha_j (n    )}{n^s},  \qquad 
	L (s+it, \vQ) L (s-it, \vQ) =  \sum_{n    =1}^{\infty}\frac{ \beta_{it} (n    )}{n^s},
\end{align}
with 
\begin{align}\label{3eq: RS coeff}
	\valpha_j (n) = 
	\sum_{h^2m = n} \varepsilon_q (h) a (m) \lambda_j (m), \qquad  \beta_{it} (n) = 
	\sum_{h^2m = n} \varepsilon_q (h) a (m) \tau_{it}  (m).
\end{align}

\subsection{Approximate Functional Equations} 
\label{sec: AFE}

As it will occur quite often, let us write
\begin{align}
	\theta = \frac {4\pi^2} {q}. 
\end{align}
Similar to \cite{Luo-Non-Vanishing}, we introduce a parameter $X  $  to be specified later. By applying \cite[Theorem 5.3]{IK} to $L ( s  , \vQ \otimes u_j)$ at $s = s_j$ and to  $ L (s+it_j , \vQ \otimes u_j) L (s-it_j, \vQ \otimes u_j) $ at $s = 1/2$, we deduce from \eqref{3eq: RS}, \eqref{3eq: RS, FE}, and \eqref{3eq: RS, 2} the approximate functional equations:  
\begin{align}\label{3eq: AFE} 
	L ( s_j  , \vQ \otimes u_j) =     \sum_{n=1}^{\infty}   \frac{ \valpha_j (n)  }{n^{1/2 + i t_j}} V_{1}  ( \theta n/X  ; t_j ) +     \upepsilon (t_j) \sum_{n=1}^{\infty}   \frac{ \valpha_j (n)  }{n^{1/2 - i t_j}} V_{1}  ( \theta n X  ; - t_j ), 
\end{align}
\begin{align}\label{3eq: AFE, 2} 
	\begin{aligned}
		|L ( s_j  , \vQ \otimes u_j)|^2 = 2 \sum_{n_1=1}^{\infty} \sum_{n_2=1}^{\infty}  \frac{ \valpha_j (n_1) \valpha_j (n_2)  }{n_1^{1/2 + i t_j} n_2^{1/2 - i t_j}} V_{2} ( \theta^2 n_1 n_2   ; t_j ), 
	\end{aligned}
\end{align}
with
\begin{align}\label{3eq: V(y;t)}
	V_{\upsigma} (y; t) =  \frac 1 {2\pi i}  \int_{(3)} \upgamma_{\upsigma}  (v, t)   y^{-v}  \frac {\nd v} {v} , \qquad \upsigma = 1, 2, 
\end{align}
where
\begin{align}\label{3eq: defn G+-}
	\upgamma_{1} (v, t) =   \exp (v^2) \frac {\Gamma  ( k + v )  \Gamma  ( k+v + 2i t ) } {\Gamma (k)   \Gamma (k+2it) } , 
\end{align}
\begin{align}\label{3eq: defn G2}
	\upgamma_{2} (v, t) =  \exp (v^2) \frac {\Gamma  ( k + v )^2 \Gamma  ( k+v + 2i t ) \Gamma  ( k+v - 2i t )} {\Gamma (k)^2 \Gamma (k + 2it) \Gamma (k - 2it)} ,
\end{align}  
\begin{align}\label{3eq: defn e(t)} 
	\upepsilon (t) = \theta^{2 it } \frac {\Gamma (k-2it)} {\Gamma (k+2it)}. 
\end{align}
In parallel,  
\begin{align}\label{3eq: AFE, E(t)} 
	L (1/2 \hspace{-1pt} + \hspace{-1pt} 2it , \vQ) L (1/2, \vQ ) \hspace{-1pt} =  \hspace{-2pt} \sum_{n=1}^{\infty} \hspace{-2pt}  \frac{ \beta_{  it} (n    )}{n^{1/2 + i t }} V_{1 }  ( \theta n/X \hspace{-1pt} ; t   )  \hspace{-1pt} + \hspace{-1pt}   \upepsilon (t) \hspace{-2pt} \sum_{n=1}^{\infty}  \hspace{-2pt} \frac{ \beta_{  it} (n    )}{n^{1/2 - i t }} V_{1}  ( \theta n X \hspace{-1pt} ; -t    ),  
\end{align}
\begin{align}\label{3eq: AFE, E(t), 2} 
	|L (1/2 +2it , \vQ) L (1/2, \vQ )|^2 = 2  \sum_{n_1=1}^{\infty} \sum_{n_2=1}^{\infty}   \frac{   \beta_{  it} (n _1   ) \beta_{  it} (n _2   )}{n_1^{1/2 + i t} n_2^{1/2 - i t}} V_{2} ( \theta^2 n_1 n_2   ; t ) .  
\end{align}

\delete{After the applications of Kuznetsov formula, for the diagonal terms,  the expression of $ V_{\upsigma} (y; t)  $ in \eqref{3eq: V(y;t)} will be suitable, but for the off-diagonal terms, we need to further expand  $\zeta_{\upsigma} (1+2 v, 2t)$  via 
\begin{align*} 
	\zeta_{\pm} (s, t) = \sum_{(h, q) = 1}  \frac {1 } {h^{s \pm  it}}, \qquad \zeta_{2} (s, t)  
	= \mathop{\sum \sum}_{(h_1h_2, q)=1} \frac { (h_1/h_2)^{it} } {(h_1 h_2)^s} , 
\end{align*}
for $\mathrm{Re}(s) > 1$, 
so that
\begin{align}\label{3eq: V+-}
	V_{\pm} (y; t) = \sum_{(h, q)=1} \frac {1 } {h^{1 \pm 2it}} \breve{V}_{\pm} (h^2 y; t),  
\end{align}
\begin{align}\label{3eq: V2}
	V_2 (y; t) = \mathop{\sum \sum}_{(h_1h_2, q)=1}  \frac { (h_1/h_2)^{2it} } {h_1 h_2 } \breve{V}_{2} \big((h_1 h_2)^2   y; t \big),
\end{align}
with 
\begin{align}\label{3eq: V'}
	\breve{V}_{\upsigma}  (y; t) = \frac 1 {2\pi i}  \int_{(3)} \upgamma_{\upsigma}  (v, t)   y^{-v}  \frac {\nd v} {v} .
\end{align}
\delete{Note that in the notation of \S \ref{sec: bi-varialbe}, we may rewrite \eqref{3eq: V2} as 
\begin{align}\label{3eq: V2, 2}
	V_2 (y; t) = \mathop{ \sum}_{(\| \boldsymbol{h} \|, q)=1}  \frac { \mathrm{c}_{t} (\boldsymbol{h}^2) } { \| \boldsymbol{h} \| } \breve{V}_{2} \big(\| \boldsymbol{h} \|^2   y; t \big),
\end{align}} 

}

\begin{lem}\label{lem: AFE}
	Let $U \Gt 1$ and $\vepsilon > 0$.  Then for $\upsigma = 1, 2$, we have
	\begin{align}\label{3eq: V(y;t), 1}
		 {V}_{\upsigma} (y; t) \Lt_{A, k } \bigg(1 + \frac {y} {|t|^{\upsigma}+1} \bigg)^{-A}, 
	\end{align} 
for any $A \geqslant 0$,  
	\begin{align}\label{3eq: V(y;t), 2}
		 {V}_{\upsigma} (y; t) =  \frac 1 {2\pi i}  \int_{\vepsilon - i U}^{\vepsilon+i U}  \upgamma_{\upsigma} (v, t)   y^{-v}  \frac {\nd v} {v} + O_{\vepsilon, k  } \bigg( \frac {(|t|+1)^{\vepsilon}  } {y^{\vepsilon} \exp (U^2/2)}  \bigg) . 
	\end{align} 
\end{lem} 

\begin{proof}
For \eqref{3eq: V(y;t), 1} and \eqref{3eq: V(y;t), 2}, we refer the reader to	  \cite[Proposition 5.4]{IK} and \cite[Lemma 1]{Blomer}. 
\end{proof}

\begin{lem}\label{lem: G(v,t)}
	  Let $\mathrm{Re}(v) > 0$ be fixed. Then for  $|\mathrm{Re} (it)| \leqslant k/2$, we have
	 \begin{align}\label{3eq: bound for G, 1}
	 	  \upgamma_{\upsigma} (v, t)  \Lt_{ \mathrm{Re} (v) , k }   (|t|+1)^{\upsigma \mathrm{Re} (v)} , 
	 \end{align} 
	 and, for $t$ real, we have 
	 \begin{align}\label{6eq: bound for G, 2}
	 \begin{aligned}
	 		\frac {\partial^{n } \upgamma_{\upsigma } (v, t) } {\partial t^n  } \Lt_{n,  \mathrm{Re} (v) ,   k }   (|t|+1)^{   {\upsigma} \mathrm{Re} (v) }  \lp \frac { \log (|t|+2) } {|t|+1 } \rp^{n}   ,  
	 \end{aligned}
	 \end{align} 
	 for any $n \geqslant 0$. 
\end{lem}

\begin{proof}
	The bounds in \eqref{3eq: bound for G, 1} and \eqref{6eq: bound for G, 2} follow  readily from Stirling's formulae for  $\log \Gamma (s)$ and its derivatives (recorded in Appendix \ref{sec: Stirling} for the reader's convenience).  Note that in the definition of $\upgamma_{\upsigma} (v, t)$ as in \eqref{3eq: defn G+-} and \eqref{3eq: defn G2} the factor $ \exp ({v^2})$ is negligibly small unless $ |\mathrm{Im} (v) | \leqslant 
	\log (|t|+2) $. 
\end{proof}

\begin{lem}\label{lem: e(t)} 
For $|\mathrm{Im} (t)| \leqslant (k-1)/2$, we have 
	\begin{align}\label{3eq: e(t), bound}
		\upepsilon (t) \Lt_{  k, q } (|t|+1)^{4 \mathrm{Im}(t)} , 
	\end{align}
and, for $t$ real and $|t|$ large, we have
\begin{align}
	\upepsilon (t) = (\sqrt{\theta} e/2|t|)^{4i t} \updelta (t)  , 
\end{align}
such that 
\begin{align}\label{3eq: d(t), bound}
t^{n}	\updelta^{(n)} (t) \Lt_{n, k} 1 ,
\end{align}
for any $n \geqslant 0$. 
\end{lem}

\begin{proof}
	The bounds in \eqref{3eq: e(t), bound} and \eqref{3eq: d(t), bound} follow from Stirling's formulae in \eqref{app1: Stirling, 1}--\eqref{app1: Stirling, 2.2}. 
\end{proof}

	\subsection{Bi-variable Notation} \label{sec: bi-varialbe}

Subsequently, particularly in our study of the second twisted moment in \S \ref{sec: 2nd moment},  we shall use bold $\boldsymbol{m}$, $\boldsymbol{n}$, ... for the  (positive integral) pairs  $(m_1, m_2)$, $(n_1, n_2)$, .... Define the norm $ \|\boldsymbol{n}\| = n_1 n_2 $, the fraction $\langle \boldsymbol{n} \rangle = n_1/ n_2$,  the dual $\widetilde{\boldsymbol{n}} = (n_2, n_1)$, the reduction $\boldsymbol{n}^{\star} = \boldsymbol{n} / \mathrm{gcd}(\boldsymbol{n}) $, and the product $ \boldsymbol{m n} = (m_1 n_1, m_2 n_2)  $. Let $\delta (\boldsymbol{n})$ be the Kronecker $\delta$ that detects $n_1 = n_2$.  

Define 
\begin{align}\label{3eq: ct(n)}
	\mathrm{c}_t ( \boldsymbol{n}) = \mathrm{Re} \big((n_1/n_2)^{it} \big) =  \cos (t \log (n_1/n_2)). 
\end{align}
For an arithmetic function $f (n)$, define 
\begin{align}\label{3eq: f(n), bi-var}
	f (\boldsymbol{n}) = f (n_1) f(n_2). 
\end{align} 

For instance,  \eqref{2eq: moments M2} now reads 
\begin{align}\label{3eq: moments M2}
	\begin{aligned}
		\SC_{2}  ( \boldsymbol{m} ) = \sum_{j = 1}^{\infty} \omega_j   \lambda_j ( \boldsymbol{m}  ) \mathrm{c}_{t_j} (\boldsymbol{m})   |L (s_j,  \vQ \otimes u_j )|^2     \exp \left(  - (t_j - T)^2/\varPi^2 \right).
	\end{aligned}
\end{align}  

Occasionally, we shall use $\boldsymbol{x} $
 for the positive real pair $(x_1 , x_2)$, and for a smooth function $\varww (x)$ write 
 \begin{align}\label{3eq: w(x)}
 	\varww (\boldsymbol{x}) = \varww ( {x}_1 )  \varww ({x}_2) . 
 \end{align}

For $\boldsymbol{n} = (n_1, n_2)$ and $\boldsymbol{N} = (N_1, N_2)$, we shall write $ \boldsymbol{n} \sim \boldsymbol{N} $ for `$ n_1 \sim N_1 $, $n_2 \sim N_2 $', and similarly $ \boldsymbol{n} \Lt \boldsymbol{N} $ for `$ n_1 \Lt N_1 $, $n_2 \Lt N_2 $'.  
	
	\section{Analysis for the Bessel Integrals}\label{sec: Bessel}

	Let $T, \varPi$ be large parameters such that $ T^{\vepsilon} \leqslant \varPi \leqslant T^{1-\vepsilon} $. Subsequently, by `negligibly small' we mean $ O_A ( T^{-A} )$ for arbitrarily large but fixed $A \geqslant 0$.   
	
	In this section, we study the Bessel transform of test functions in the form
	\begin{align}\label{4eq: defn h}
		h (t; y) = \upgamma (t) y^{-2it} \upphi (t) + \upgamma (- t) y^{2it} \upphi (- t), 
	\end{align} 
\begin{align}\label{4eq: defn h-}
	h_{\scriptscriptstyle -} (t; y) = \upgamma (- t) \upepsilon (t)  y^{-2it} \upphi (t) + \upgamma (  t) \upepsilon (-t) y^{2it} \upphi (- t), 
\end{align} 
	for $\upgamma   (t) $ and $\upepsilon (t)$ respectively in the function spaces $ \mathscr{T}_{\vkappa} $ and $\mathscr{F}_{\valpha    } (a)$ defined below in Definitions \ref{defn: space T(r)} and \ref{defn: space F(r)} and 
	\begin{align}\label{4eq: defn phi}
		\upphi (t) = \exp \bigg(  \hspace{-2pt}    -   \frac {(t - T)^2 } {\varPi^2 }    \bigg) . 
	\end{align}
In the case that $\upgamma (t)$ is even, we also need to consider
\begin{align}
	h_2 (t; y) = h (t; y) + h (t; 1/y) = 2 \upgamma (t) \cos (2\log y) (\upphi (t) + \upphi (-t)) . 
\end{align}
Let $H (x, y)$ and $H_{-} (x, y)$ respectively denote the Bessel transform of $h (t; y)$ and $h_{\scriptscriptstyle -} (t; y)$ as in \eqref{3eq: integrals}, namely 
\begin{align}\label{4eq: H(x,y)}
	H (x, y) = \frac{2 i } {\pi} \int_{-\infty}^{\infty}  h (t; y) J_{2it} (x) \frac {t \nd t} {\cosh (\pi t) }. 
\end{align}
Define $ H_2 (x, y)$ in the same way, but 
\begin{align}\label{4eq: H2=H}
	H_2 (x, y) = H(x, y) + H (x, 1/y), 
\end{align} 
so the results below for $H(x, y) $ will be valid for  $H_2 (x, y) $ as well.

		\begin{defn}\label{defn: space T(r)}
		For fixed real $\vkappa $, 
		let $ \mathscr{T}_{\vkappa } $ denote the space of holomorphic functions $\upgamma   (t) $ on the strip $ |\mathrm{Im} (t)| \leqslant 1/2 + \vepsilon$ such that 
		\begin{align}\label{4eq: bound for w, 1}
			 \upgamma (t)    \Lt (|t|+1)^{\vkappa  }  , 
		\end{align}
and that for  $t$ real its derivatives
			\begin{align}\label{4eq: bound for u, 2}
				\upgamma^{(n)}   (t)    \Lt_{n } (|t|+1)^{ \vkappa }  \lp \frac { \log (|t|+2) } { |t|+1 } \rp^{n}   , 
			\end{align} 
			for any $n \geqslant 0$. 
		
	\end{defn}

\begin{defn}\label{defn: space F(r)}
	For fixed real $a,  \valpha     > 0$. 
	let $ \mathscr{F}_{ \valpha     } (a)$ denote the space of holomorphic functions $\upepsilon   (t) $ on the strip $ |\mathrm{Im} ( t)| \leqslant 1/2 + \vepsilon$ such that 
	\begin{align}\label{4eq: bound for e, 1}
		\upepsilon (t)    \Lt (|t|+1)^{ 2 \valpha        \mathrm{Im} (  t)  }  , 
	\end{align}
	and that for  $t$ real and $|t|$ large 
	\begin{align}\label{4eq: bound for e, 2}
	\upepsilon (t)  = (a e / |t|)^{2 \valpha    i t} \updelta (t) , \qquad t^n \updelta^{(n)}   (t)    \Lt_{n } 1  , 
	\end{align} 
	for any $n \geqslant 0$. 
	
\end{defn}
	 
Later, in practice, $ \upgamma (t)$ will be those $ \upgamma_{\upsigma} (v, t) $ for  $ v \in [\vepsilon -i \log T, \vepsilon + i \log T] $  in view of Lemma \ref{lem: AFE}; indeed,  one may effectively be restricted to $ ||t| - T| \leqslant \varPi \log T $ since $\upphi (\pm t)$  is exponentially small if otherwise, so the error term in \eqref{3eq: V(y;t), 2} is negligibly small if one chooses $U = \log T$. It follows from Lemma \ref{lem: G(v,t)} that $ \upgamma_{\upsigma} (v, t) $  lies in  $ \mathscr{T}_{ \upsigma \vepsilon }  $. Moreover,  the $\upepsilon (t)$ given in \eqref{3eq: defn e(t)} lies in the space $   \mathscr{F}_{2}  (\sqrt{\theta}  /2)$ by Lemma \ref{lem: e(t)}. 

Now let us analyze the Bessel integrals $H(x, y)$ and $H_{-} (x, y)$.  Some preliminary analysis will be similar to those in  Li and Young's works  \cite{XLi2011,Young-Cubic}.


\begin{lem}\label{lem: x<1}
Set $u = xy +x/y$. Then for $u \Lt 1  $,  we have  $H (x, y)   =  \allowbreak O   ( \varPi T^{  \vkappa   }  u   )$. 
\end{lem}


\begin{proof}
This lemma follows easily from shifting the integral contour in \eqref{4eq: H(x,y)} up to $\mathrm{Re} (it) = 1/2 + \vepsilon$. More explicitly, 
in view of \eqref{4eq: defn h}, \eqref{4eq: defn phi},   \eqref{4eq: H(x,y)}, and \eqref{4eq: bound for w, 1}, the  residue  at $ it = 1/2 $ 
	is bounded by
	\begin{align*}
		   \exp  ( - T^2/\varPi^2  )  (y+1/y)  |J_{1} (x) | , 
	\end{align*}
	while the resulting integral is bounded by 
	\begin{align*}
		   (y+1/y)^{1+2\vepsilon} 	\int_{-\infty}^{\infty} (|t|+1)^{1+\vkappa } \exp \bigg(  \hspace{-2pt} -   \frac {(t - T)^2 } {\varPi^2 }    \bigg)  \frac {|J_{1 + 2\vepsilon +2it} (x) |} { |\sin  (\pi (\vepsilon  + it)) |  }  \nd t . 
	\end{align*}
	By the Poisson integral \eqref{3eq: Poisson, J},  
	\begin{align*}
		J_{ 1} ( x)  \Lt  x , \qquad 
		\frac {J_{ 1 + 2\vepsilon +2it } (x)  } { \sin  (\pi (\vepsilon  + it))   }   \Lt    \lp \frac { x  } { |t|+1 } \rp^{1+2\vepsilon}  . 
	\end{align*}  
	Thus 
	\begin{align*}
		H (x, y) \Lt \exp  (  - T^2/\varPi^2  )  u +   {\varPi T^{  \vkappa  - 2\vepsilon } u^{1+2\vepsilon} }   \Lt   {\varPi T^{  \vkappa  } u} ,
	\end{align*}
	provided that $u \Lt 1 $.  
\end{proof}

By the same proof with \eqref{4eq: bound for w, 1} and \eqref{4eq: bound for e, 1} combined, we have a similar bound for $H_-(x, y)$.

\begin{lem}\label{lem: x<1, -}
	Set $$u_{\scriptscriptstyle -} = xy T^{\valpha} + \frac {x}  { y T^{\valpha} } .$$ Then for $u_{\scriptscriptstyle -} \Lt 1  $,  we have    $ H_- (x, y) = O  (\varPi T^{ \vkappa    } u_{\scriptscriptstyle -} )   $.  
\end{lem}

Now we establish the integral representations of $H (x, y)$ and $H_{-} (x, y)$ which will be most crucial in our analysis. 

\begin{lem}\label{lem: Bessel} 
Let $A \geqslant 0$. Then for $x \Gt T^{-A}$ we have   \begin{equation}\label{4eq: H(x,y) = I(v,w)}
	\begin{split} 
		H (x, y) = \varPi T^{1+\vkappa}  \int_{-\varPi^{\vepsilon}/\varPi}^{\varPi^{\vepsilon}/ \varPi} 
		g (  \varPi r) \exp (  2i  T r) \cos  ( f (r; v, w) )   \nd r  + O (T^{-A})  ,   
	\end{split}
\end{equation}  
where 
\begin{align}\label{4eq: f(r,v,w)}
f (r; v, w) =  v \exp  (r) +   w \exp (- r) , \qquad 	v = \frac {xy} {2}, \qquad w = \frac {x/y} 2, 
\end{align}
and  $g (r)$ is a Schwartz function {\rm(}namely, $ g^{(n)} (r) = O ( (|r|+1)^{-m})$ for any $m, n \geqslant 0${\rm)}. 
\end{lem}

\begin{proof}
By the Mehler--Sonine integral \eqref{3eq: integral Bessel}, we have  
	\begin{align*}
		\frac{J_{2it} (x) - J_{-2it} (x)} {\cosh (\pi t)}
		=   \frac {2} {\pi i} \tanh (  \pi t) \int_{-\infty}^{\infty} \cos (x \cosh r) \exp (2 i t r) \nd r. 
	\end{align*}
For $x \Gt T^{-A}$, it follows by partial integration that only a negligibly small error will be lost if the integral
above is truncated at $|r| = T^{\vepsilon}$.  Thus, in view of \eqref{4eq: defn h}, \eqref{4eq: defn phi}, and \eqref{4eq: H(x,y)}, up to a negligibly small error, the Bessel integral $H (x, y)$  equals
\begin{align*}
 	\frac{2 } {\pi^2}  \int_{-T^{\vepsilon}}^{T^{\vepsilon}} \cos (x \cosh r)   \int_{ - \infty }^{ \infty }  t \upgamma (t) \tanh (\pi t) \upphi (t)  \exp ( 2 i t (    r - \log y ))  \nd t \, {  \nd r} . 
\end{align*}
Note that $\upphi (t)$ as defined in \eqref{4eq: defn phi} is exponentially small unless $|t - T | \leqslant \varPi \log T$, so $\tanh (\pi t)$ is removable as $ \tanh (\pi t) = 1 + O (\exp (-\pi T)) $ on this range. By the change of variables  $r \ra r + \log y$ and $t \ra T + \varPi t$, the expression above  turns  into 
\begin{align*}
\varPi T^{1+\vkappa}   \int_{-T^{\vepsilon}}^{T^{\vepsilon}} \exp (   2 i T r ) \cos (x \cosh (r + \log y))     \int_{ - \infty }^{ \infty }  \upbeta   (t) \exp ( 2 i \varPi t r )  \nd t \, {  \nd r}  ,
\end{align*}
where 
\begin{align*}
\upbeta  (t) =	\frac{2 (1 + \varPi t/ T ) \upgamma (T+\varPi t) } {\pi^2 T^{\vkappa}}  \exp (-t^2)	 . 
\end{align*}
By \eqref{4eq: bound for u, 2}, we infer that
\begin{align*}
	\upbeta^{(n)} (t) \Lt_{n} \exp (-t^2/2) , 
\end{align*} 
for any $n  \geqslant 0$, where the implied constants do {\it not} depend on $T$ or $\varPi$. Now the inner integral is the Fourier transform $ \hat{\upbeta}  ( - \varPi r/ \pi) $, so the $r$-integral may be truncated effectively at $|r| = \varPi^{\vepsilon} / \varPi$. Let us write $ g (r) = \hat{\upbeta} ( -  r/ \pi) $, then \eqref{4eq: H(x,y) = I(v,w)} follows from 
\begin{align*}
	x \cosh (r + \log y) =   v \exp (r) +   w \exp (-r) . 
\end{align*}
\end{proof}

\begin{lem}\label{lem: Bessel, -} 
Let $A \geqslant 0$. Then for $x \Gt T^{-A}$ we have   \begin{equation}\label{4eq: H-(x,y)}
		\begin{split} 
			H_{\scriptscriptstyle -} (x, y) = \varPi T^{1+\vkappa}   \int_{  -\rho \varPi^{\vepsilon}  /\varPi}^{  \rho\varPi^{\vepsilon}  / \varPi}   g_{\scriptscriptstyle -} (  \varPi r) \exp (  2i  T r) \cos  (   f (r; v_{\scriptscriptstyle -}, w_{\scriptscriptstyle -} ) )  \nd r  
			+ O \big(T^{-A} \big)    ,   
		\end{split}
	\end{equation}  
	where 
	\begin{align}\label{4eq: rho}
		\rho =   1 + \frac {\varPi^2} {T}, 
	\end{align}
the function $f (r; v , w )$ is defined as in {\rm\eqref{4eq: f(r,v,w)}},
	\begin{align}\label{4eq: v-, w-}
		v_{\scriptscriptstyle -} =  \frac {x y T^{\valpha}} {2 a^{\valpha}} , \qquad w_{\scriptscriptstyle -} = \frac {a^{\valpha} x/y } {2  T^{\valpha} } , 
	\end{align}
\begin{align}\label{4eq: g-(r)}
	g_{\scriptscriptstyle -} (r) =   \int_{ - \log T }^{  \log T }  \upbeta_{\scriptscriptstyle -}   (t) \exp ( 2 i (t r - \valpha \lambda (t) ) )  \nd t ,
\end{align}
\begin{align}
	\lambda (t) =     (T+\varPi t) \bigg( \log \bigg(1+ \frac {\varPi t} T\bigg) - 1 \bigg),
\end{align}
and  $ \upbeta_{\scriptscriptstyle -}  (t)  $ is  a certain weight of bounds  
\begin{align}\label{4eq: beta-(t)}
	\upbeta_{\scriptscriptstyle -}^{(n)} (t) \Lt_{n} \exp (-t^2/2) , 
\end{align} 
for any $n  \geqslant 0$. Moreover,  the integral $g_{\scriptscriptstyle -} (r)$ is negligibly small unless $|r   | \leqslant \rho\varPi^{\vepsilon}$ {\rm(}as indicated in the integral domain of {\rm\eqref{4eq: H-(x,y)}}{\rm)}, in which case 
	\begin{align}\label{4eq: bound for g-(r)}
		 g_{\scriptscriptstyle -}^{(n)} (r) \Lt \frac 1 {\sqrt{\rho}},
	\end{align} 
for any $ n \geqslant 0 $. 
\end{lem}

\begin{proof}
The proof is similar to that of Lemma \ref{lem: x<1, -}, but we need to incorporate the factor $ (a e / |t|)^{2 \valpha    i t} $ from $\upepsilon (t)$ (see \eqref{4eq: bound for e, 1}). For  simplicity, write $$ y_{\scriptscriptstyle -} =  \frac {y T^{\valpha}} {a^{\valpha}} . $$
By the line of arguments above,  we obtain 
$$\varPi T^{1+\vkappa}   \int_{-T^{\vepsilon}}^{T^{\vepsilon}} g_{\scriptscriptstyle -} (  \varPi r) \exp (   2 i T r ) \cos (x \cosh (r + \log y_{\scriptscriptstyle -} ))      {  \nd r}  , $$
with 
\begin{align*}
	\upbeta_{\scriptscriptstyle -}  (t) =	\frac{2 (1 + \varPi t / T  ) \upgamma \updelta (T+\varPi t) } {\pi^2 T^{\vkappa}}  \exp (-t^2)	,
\end{align*}
\begin{align*}
	x \cosh (r + \log y_{\scriptscriptstyle -}) =   v_{\scriptscriptstyle -} \exp (r) +   w_{\scriptscriptstyle -} \exp (-r) . 
\end{align*}
 Note that the change $r \ra r + \log y$ needs to be replaced by $r \ra r + \log y_{\scriptscriptstyle -}$.   
 Moreover,  the bounds for $\upbeta_{\scriptscriptstyle -}  (t)$ in \eqref{4eq: beta-(t)} follow  readily from
  \eqref{4eq: bound for u, 2}  and \eqref{4eq: bound for e, 2}. 
 
 It is left to prove the statements for the integral $g_{\scriptscriptstyle -} (r) $ as defined in  \eqref{4eq: g-(r)}. 
To this end, we apply stationary phase as in Lemmas   \ref{lem: staionary phase} and \ref{lem: 2nd derivative}. Now  the phase function has derivative 
\begin{align*}
	r - \valpha \lambda' (t) = r - \valpha   \varPi    \log  \bigg(1+ \frac {\varPi t} T\bigg)    = r  + O \bigg(\log T \cdot \frac {\varPi^2} {T} \bigg), 
\end{align*}
while $$  \lambda '' (t) =  \frac {  \varPi^2  } {T+\varPi t},  $$
and hence $ \lambda^{(j)} (t) = O_{n} (T (\varPi/T)^j) $ for any $j \geqslant 2$. By an application of Lemma \ref{lem: staionary phase} with $P = 1$, $Q = T/\varPi$,  $Z = T$, and $ R = \rho \varPi^{\vepsilon}$, we infer that  the integral $ g_{\scriptscriptstyle -} (r)$ is negligibly small if $ |r  | > \rho \varPi^{\vepsilon} $. As for the bound in \eqref{4eq: bound for g-(r)}, we estimate the integral (after differentiation) trivially if $ \varPi  < \sqrt{T} $ and by the second derivative test in Lemma \ref{lem: 2nd derivative} if otherwise. Actually, to be strict, in the applications of Lemmas \ref{lem: staionary phase} and \ref{lem: 2nd derivative} we need to use a smooth cut-off function of the form $\varvv (  t / \log T)$. 
\end{proof}

Next,  we analyze the integrals   in \eqref{4eq: H(x,y) = I(v,w)} and \eqref{4eq: H-(x,y)} by stationary phase  (Lemmas   \ref{lem: staionary phase}, \ref{lem: 2nd derivative}, and    \ref{lem: 2nd derivative test, dim 2}).   

\begin{lem}\label{lem: analysis of I}
	Let   $T^{\vepsilon} \leqslant \varPi \leqslant T^{1-\vepsilon}$.  
	Then $H (x, y) = O  (T^{-A})$ if $v, w \Lt T$. 
\end{lem}


\begin{proof}
	Consider the phase  $$
	2  T r \pm f (r; v, w) = 2  T r \pm ( v \exp (r)  + w \exp (-r) ) ,  $$ which arises with the cosine expanded  in \eqref{4eq: H(x,y) = I(v,w)}. 
	Note that it has derivative 
	\begin{align*}
	2 T   \pm f' (r; v, w) =	2 T   \pm (v \exp (r) - w \exp (-r) ) .
	\end{align*}
	On the range $|r| \leqslant \varPi^{\vepsilon}/ \varPi$, we  have $ |2 T   \pm f' (r; v, w)| \Gt T$ for $v, w \Lt T$ and $ f^{(j)} (r; v, w) \Lt v + w  $ for any $j \geqslant 2$. By applying Lemma \ref{lem: staionary phase}, with $P = 1/ \varPi$, $Q = 1$,  $Z = v + w$, and $ R = T$,  we infer that      $H (x, y) $ is negligibly small. 
\end{proof}

Lemma \ref{lem: analysis of I} may be refined in the case when $v $ and $w$ are far apart.    

\begin{lem}\label{lem: analysis of I, 2}
	Let   $T^{\vepsilon} \leqslant \varPi \leqslant T^{1-\vepsilon}$.  
	If $ v < w / T^{\vepsilon}$ or $  w < v / T^{\vepsilon}$, then $H (x, y) = O  (T^{-A})$ unless  $  \sqrt{2} T < w < 2 \sqrt{2} T$ or $\sqrt{2} T < v < 2 \sqrt{2} T$, respectively. 
\end{lem}

As for $H_{-} (x, y)$, in practice $w_{\scriptscriptstyle -} $ will be (very) small  due to the $T^{\valpha}$ (as $\valpha = 2$) in \eqref{4eq: v-, w-}, and in this case we prove the next lemma.  

\begin{lem}\label{lem: analysis of I, -}
	Let   $T^{\vepsilon} \leqslant \varPi \leqslant T^{1-\vepsilon}$. Assume $\valpha > 1$. Then for $ w_{\scriptscriptstyle -} \Lt \rho T /\varPi^{1-\vepsilon}   $, we have $H_{-} (x, y) = O  (T^{-A})$ unless  $ |v_{\scriptscriptstyle -} - 2 T| \Lt \rho T /\varPi^{1-\vepsilon}  $, in which case  
	\begin{align}\label{4eq: bound for H-}
		H_{-} (x, y) = O \bigg(\frac {\varPi T^{1/2+ \vkappa} } {\sqrt{\rho}} \bigg) . 
	\end{align}  
\end{lem}

\begin{proof}
	 As in the proof of Lemma \ref{lem: analysis of I},   consider the phase   $$  2  T r \pm f (r; v_{\scriptscriptstyle -}, w_{\scriptscriptstyle -}) = 2  T r \pm ( v_{\scriptscriptstyle -} \exp (r)  + w_{\scriptscriptstyle -} \exp (-r) ) . $$ 
	 Note that in order for the integral not to be negligibly small, we must have  sign $-$ and $ v_{\scriptscriptstyle -} \asymp T $, since $   |2 T   \pm f' (r; v_{\scriptscriptstyle -}, w_{\scriptscriptstyle -})| \Gt T$ if  otherwise. However, when   Lemma  \ref{lem: staionary phase} is applied, the lower bound $T$ is actually wasteful in our setting. Observe that 
	 $$ 2 T - f' (r; v_{\scriptscriptstyle -}, w_{\scriptscriptstyle -}) = 2 T - v_{\scriptscriptstyle -}  + O (\rho T / \varPi^{1-\vepsilon} ) ,$$
	 on the range $|r| \leqslant \rho \varPi^{\vepsilon}/ \varPi $, and hence  $ |2 T   - f' (r; v, w)| \Gt \rho T / \varPi^{1-\vepsilon}$ in the case that $ |v_{\scriptscriptstyle -} - 2 T| \Gt \rho T /\varPi^{1-\vepsilon}   $. Then an application of  Lemma  \ref{lem: staionary phase}, with  $P = 1/ \varPi$, $Q = 1$,  $Z = T$, and $ R = \rho T /\varPi^{1-\vepsilon}$ yields the first statement of this lemma. Note here that $R > \varPi^{\vepsilon} \sqrt{T}$ and $R > \varPi^{1+\vepsilon}$ due to $\rho = 1+\varPi^2 / T$ (see \eqref{4eq: rho}).    
	 
	 It is left to prove the bound in \eqref{4eq: bound for H-}. To this end, we use the 1- or 2-dimensional second derivative test in Lemma \ref{lem: 2nd derivative} or \ref{lem: 2nd derivative test, dim 2} according as whether  $ \varPi < \sqrt{T} $ or not. In the former case, we have $  \rho \asymp 1$ and $ v_{\scriptscriptstyle -}  \asymp T $, so Lemma \ref{lem: 2nd derivative} yields directly $H_{-} (x, y) = O (\varPi T^{1/2+ \vkappa}) $.  As for the latter case with $\rho \asymp \varPi^2/T$, after inserting \eqref{4eq: g-(r)} into \eqref{4eq: H-(x,y)}, we obtain a double integral of phase function
	 \begin{align*}
  f (t, r; v_{\scriptscriptstyle -}, w_{\scriptscriptstyle -}) = 2 \varPi t r - 2 \valpha \lambda (t) + 2 Tr -  f(r; v_{\scriptscriptstyle -}, w_{\scriptscriptstyle -}) .
	 \end{align*}
 Then
 \begin{align*}
 		& \partial^2 f (t, r; v_{\scriptscriptstyle -}, w_{\scriptscriptstyle -}) / \partial t \partial r = 2 \varPi , \\
  	\partial^2 f (t, r; v_{\scriptscriptstyle -}, w_{\scriptscriptstyle -}) / \partial t^2 & = - 2 \valpha \lambda'' (t) = -2 \valpha \frac {\varPi^2} {T} \bigg(1 + O \bigg(\log T \cdot \frac {\varPi} {T} \bigg)\bigg), \\
 	 	\partial^2 f (t, r; v_{\scriptscriptstyle -}, w_{\scriptscriptstyle -}) / \partial r^2 & = - v_{\scriptscriptstyle -} \exp (r) - w_{\scriptscriptstyle -} \exp (-r) = - 2 T \bigg( 1 + O \bigg(\varPi^{\vepsilon}  \frac { \varPi } {T} \bigg)\bigg). 
 \end{align*} 
By the assumption $\valpha > 1$, we infer that $$ |\det f'' (t, r; v_{\scriptscriptstyle -}, w_{\scriptscriptstyle -})| = 4 (\valpha - 1 + o(1) ) \varPi^2  \Gt \varPi^2.  $$
Thus   Lemma \ref{lem: 2nd derivative test, dim 2} yields 
\begin{align*}
	H_{-} (x, y) \Lt \frac {\varPi T^{1+ \vkappa}} {\sqrt{\varPi^2}} = T^{1+\vkappa}, 
\end{align*}
as desired. 
\end{proof}

By Lemmas \ref{lem: x<1}, \ref{lem: Bessel}, and \ref{lem: analysis of I}, we conclude that the Bessel integral $H (x, y)$ is negligibly small if $ xy+x/y \Lt T $.

\begin{coro}\label{cor: u<T}
Set $u = xy +x/y$ as before. Then for  $ u \Lt T $, we have 
	\begin{align}
		H (x, y) \Lt \varPi T^{\vkappa} \min \left\{ u,  \frac 1 {T^A}  \right\} ,
	\end{align}
for any $A \geqslant 0$. 
\end{coro}

Similarly, by Lemmas \ref{lem: x<1, -}, \ref{lem: Bessel, -}, and \ref{lem: analysis of I, -}, we conclude that the Bessel integral $H_{-} (x, y)$ is negligibly small if $ w_{\scriptscriptstyle -} \Lt \rho T /\varPi^{1-\vepsilon} \Lt |v_{\scriptscriptstyle -} - 2 T|    $. 

\begin{coro}\label{cor: xy, -}
Let $\valpha > 1$. Let $v_{\scriptscriptstyle -} $ and $w_{\scriptscriptstyle -}$ be defined in {\rm\eqref{4eq: v-, w-}}.   Assume $ w_{\scriptscriptstyle -} \Lt \rho T /\varPi^{1-\vepsilon}   $. Then for $ |v_{\scriptscriptstyle -} - 2 T| \Gt \rho T /\varPi^{1-\vepsilon}  $, we have
	\begin{align}
		H_{-} (x, y) \Lt \varPi T^{\vkappa} \min \left\{ v_{\scriptscriptstyle -} + w_{\scriptscriptstyle -} ,  \frac 1 {T^A}  \right\} ,
	\end{align}
	for any $A \geqslant 0$. 
\end{coro}

\section{Setup}

For large parameters $T, \varPi$ with $T^{\vepsilon} \leqslant \varPi \leqslant T^{1-\vepsilon}$, let us define 
\begin{align}\label{5eq: phi}
	\upphi (t) = \exp \left( - \frac {(t-T)^2}  {\varPi^2} \right),
\end{align}
and consider  
\begin{align}\label{5eq: C1(m)}
	{\SC}^{\natural }_{1}  (m) = \sum_{j = 1}^{\infty} \omega_j \lambda_j (m)   \big ( m^{-i t_j} L (s_j,  \vQ \otimes u_j )   \upphi (t_j) + m^{i t_j} L (\overline{s}_j,  \vQ \otimes u_j )   \upphi (- t_j) \big) ,
\end{align}
\begin{align}\label{5eq: C2(m)}
	{\SC}^{\natural }_{2}  (\boldsymbol{m} ) = \sum_{j = 1}^{\infty}  \omega_j \lambda_j (\boldsymbol{m} ) \mathrm{c}_{ t_j} (\boldsymbol{m})   | L (s_j,  \vQ \otimes u_j )|^2  (  \upphi (t_j) +    \upphi (- t_j)  ) . 
\end{align}
Note that 
${\SC}^{\natural }_{1}  (m ) $ or $ {\SC}^{\natural }_{2}  (\boldsymbol{m} )  $ respectively differs from  ${\SC}_{1}  (m  ) $ or $ {\SC}_{2}  (\boldsymbol{m} ) $ as in \eqref{2eq: moments M1} or \eqref{2eq: moments M2} (see also \eqref{3eq: moments M2}) merely by  an exponentially small error term.  In parallel, the Eisenstein contributions read 
\begin{align}
\SE_{1}  (m ) = \frac {2}  {\pi}   \cdot  L (1/2, \vQ )  \int_{-\infty}^{\infty} \omega(t) \tau_{  it} (m) m^{-it}  L (1/2 +2it , \vQ)  \upphi (t)  \nd t ,
\end{align} 
\begin{align}
\begin{aligned}
\SE_{2}  (\boldsymbol{m} )	=  \frac {2}  {\pi}   \cdot   L (1/2, \vQ )^{2}   \int_{-\infty}^{\infty} \omega(t) \tau_{  it} (\boldsymbol{m})   \mathrm{c}_{ t } (\boldsymbol{m})  |L (1/2 +2it , \vQ)|^{2}   \upphi (t)  \nd t,
\end{aligned}
\end{align}
where, according to \eqref{3eq: f(n), bi-var},  
$
	\tau_{it} (\boldsymbol{m}) =  \tau_{it} (m_1 ) \tau_{it} (m_2 ). 
$
 
\delete{which may be rewritten as 
	\begin{align*}
		\frac {1}  {\pi}      \int_{-\infty}^{\infty} \omega(t) \tau_{  it} (m) L (1/2, \vQ )^{\eta} \big(  L (1/2 +2it , \vQ)^{\eta}   \upphi (t) + L (1/2 -2it , \vQ)^{\eta}   \upphi (-t)  \big)  \nd t  .
\end{align*}}

\begin{lem}\label{lem: Eis E1,2}
We have   
\begin{align}\label{5eq: E1(m)}
			\SE_{1} (m) = O_{\vQ} \big( \tau (m) \cdot {\textstyle  \sqrt{\varPi^2 + \varPi T^{2/3} } }  \log^{3} T \big)  , 
	\end{align}
\begin{align}\label{5eq: E2(m)}
	 \SE_{2} (\boldsymbol{m} )  = O_{  \vQ}   \big(  \tau  ( \boldsymbol{m} )  \cdot  (\varPi + T^{2/3}) \log^3 T  \big) . 
\end{align}
\end{lem}

\begin{proof}
Since $ \omega (t) = O (  \log^2 (|t|+1) )$  by \cite[(3.11.10)]{Titchmarsh-Riemann}, it is clear that \eqref{5eq: E1(m)} and \eqref{5eq: E2(m)} follow from \eqref{3eq: L(Q) 1} and \eqref{3eq: L(Q) 2} in Lemma \ref{lem: L(Q)} respectively, if we choose $H = \varPi \log T$.   
\end{proof}


\section{The Twisted First Moment} 

In this section, 
we study the  twisted first moment ${\SC}^{\natural }_{1}  ( {m})$ as defined in \eqref{5eq: C1(m)}.  

\subsection{Application of Kuznetsov Formula} \label{sec: apply Kuz}

By the approximate functional equations \eqref{3eq: AFE} and \eqref{3eq: AFE, E(t)}, along with \eqref{3eq: RS coeff},  we may write
\begin{align*}  
		{\SC}^{\natural }_{1}  (m) + {\SE}_{1}  (m)  = \sum_{\pm} 
		\sum_{h=1}^{\infty} \sum_{n=1}^{\infty} \frac{\varepsilon_q (h) a (n)} {h \sqrt{n}} & \Bigg\{   \sum_{j = 1}^{\infty} \omega_j   {  \lambda_j (m    )  \lambda_j (n    )}    {V}_{\pm} (m, h^2 n ; t_j)  \\
		& + \frac 1 {\pi} \int_{-\infty}^{\infty} \omega(t) \tau_{  it} (m)  \tau_{  it} (n)   {V}_{\pm} (m, h^2n; t) \nd t \Bigg\} ,  
\end{align*} 
where 
\begin{align}\label{6eq: V+(n;t)}
	{V}_{+}  (m, n; t) = V_{1}    (  \theta n/X ; t  )  (m n)^{ - i t}  {\upphi (t)}  + V_{1}    (  \theta n / X ; - t  )   (mn)^{  i t} {\upphi (- t)} ,
\end{align} 
\begin{align}\label{6eq: V-(n;t)}
	{V}_{-}  (m, n; t) = V_{1}    (  \theta n X ; - t  )  ( n / m  )^{i t} \upepsilon (t)  {\upphi (t)}  + V_{1}    ( \theta  n X ;  t  )  (  n / m  )^{- i t} \upepsilon (t) {\upphi (- t)} . 
\end{align} 
We make the mild assumption that 
\begin{align*}
	T^{\vepsilon} < X <    T^{1-\vepsilon}. 
\end{align*}
It follows from the Kuznetsov formula in Lemma \ref{lem: Kuznetsov} that
\begin{align}\label{6eq: C+E=D+O}
	{\SC}^{\natural }_{1}  ( {m}) + {\SE}_{1}  (m) = \SD_1^+ (m) + \SD_1^- (m)  + \SO_1^+ (m) + \SO_1^- (m), 
\end{align}
with the diagonal 
\begin{align}\label{6eq: D+-(m)}
	\SD_1^{\pm} (m) = \frac{a (m)} {\sqrt{m}}  
	\sum_{h=1}^{\infty}  \frac {\varepsilon_q (h) {H}_{\pm} (m, h)} {h} , 
\end{align} 
where
\begin{align}\label{6eq: H+-(m)}
	 {H}_{\pm} (m, h)  =  \frac 1 {\pi^2 }  \int_{-\infty}^{\infty}   {V}_{\pm}  (m, h^2 m; t) \tanh(\pi t) t \nd t ,
\end{align}
and the off-diagonal 
\begin{align}\label{6eq: O(m)}
	\SO_1^{\pm} (m) =  
	\sum_{h=1}^{\infty} \frac {\varepsilon_q (h)} {h}  \sum_{n=1}^{\infty} \frac{a (n)} {\sqrt{n}} \sum_{c=1}^{\infty} \frac {S(m,n; c)} {c}   {H}_{\pm} \bigg(\frac {4\pi \sqrt{mn}} {c} ; m, h^2 n \bigg), 
\end{align}
where 
\begin{align}\label{6eq: Bessel, 1}
	  {H}_{\pm} (x; m, n) = \frac {2i} {\pi}   \int_{-\infty}^{\infty}   {V}_{\pm}  (m, n; t) J_{2it} (x)  \frac {t \nd t} {\cosh (\pi t) } . 
\end{align} 

\subsection{Asymptotics for the Diagonal Sums}

In view of \eqref{6eq: V+(n;t)}, \eqref{6eq: V-(n;t)}, and \eqref{6eq: H+-(m)}, we have 
 \begin{align*}
	{H}_{+} (m, h)  = \frac {2} {\pi^2}  \int_{-\infty}^{\infty}   V_{1}    (  \theta h^2 m / X ; t  ) (h m)^{ - 2 i t}  {\upphi (t)}  \tanh(\pi t) t \nd t, 
\end{align*}
\begin{align*}
	{H}_{-} (m, h)  = \frac {2} {\pi^2}  \int_{-\infty}^{\infty}   V_{1}    (  \theta h^2 m X ; - t  )  h^{ 2 i t}  \upepsilon (t) {\upphi (t)}  \tanh(\pi t) t \nd t .
\end{align*} 
Keep in mind that  $\upphi (t)$ as in \eqref{5eq: phi} is exponentially small unless $|t - T | \leqslant \varPi \log T$. Thus, up to an exponentially small error, $\tanh(\pi t) $ is removable as $\tanh(\pi t) = 1 + O (\exp (-\pi T))$ on this range.   Also, let us truncate  the $h$-sum in \eqref{6eq: D+-(m)} at $ h < T $ say, as otherwise $ V_{1}    (  \theta h^2 m /X  ; t  ) $ and $ V_{1}    (  \theta h^2 m  X  ; t  ) $ 
are negligibly small due to \eqref{3eq: V(y;t), 1} in Lemma \ref{lem: AFE}. 

\begin{lem}\label{lem: H(m,h)}
Let $ h <  T$.  Then, except for $ H_+(1, 1) $, we have 
	\begin{align}
		 H_{\pm} (m, h) = O     (    {\varPi T^{\vepsilon}}    ). 
	\end{align}
\end{lem}
\begin{proof} By inserting the integral  in \eqref{3eq: V(y;t), 2}, with $U = \log T$, we arrive at
\begin{align*}
	 {H}_{+} (m, h)  = \frac {1} {\pi^3 i}  \int_{\vepsilon-i\log T}^{\vepsilon+i\log T}  \int_{T-\varPi \log T}^{T+\varPi \log T}      \frac { \upgamma_{1} (v, t)    {\upphi (t)} t } { (\theta/ X)^{v} h^{2v+2it} m^{v+2it}   }    \frac {  \nd t \nd v} { v} + O (T^{-A}), 
\end{align*}
\begin{align*}
	{H}_{-} (m, h)  = \frac {1} {\pi^3 i}  \int_{\vepsilon-i\log T}^{\vepsilon+i\log T}  \int_{T-\varPi \log T}^{T+\varPi \log T}      \frac { \upgamma_{1} (v, - t) \upepsilon (t)   {\upphi (t)} t } { (\theta m X)^{v} h^{2v-2it}   }    \frac {  \nd t \nd v} { v} + O (T^{-A}) .
\end{align*} 
For $t , v$ on the integral domains, by applying Stirling's formula \eqref{app1: Stirling, 1} to \eqref{3eq: defn G+-}, we have 
\begin{align*} 
	\upgamma_{1} (v, t) = \updelta_{\scriptscriptstyle +} (v) \big( t^v + O  (    {T^{\vepsilon} } / T  ) \big) ,  \quad 
	 \upgamma_{1} (v, - t) \upepsilon (t) = \updelta_{\scriptscriptstyle -} (v)  (  \sqrt{\theta} {e} / {2t}  )^{4it} \big( t^v + O  (    {T^{\vepsilon} } / T  ) \big),
\end{align*} 
for 
\begin{align*}
\updelta_{\scriptscriptstyle \pm} (v) =	 \exp (v^2 + v \log (\pm 2i )) \cdot  {\Gamma  ( k + v )  } / {\Gamma (k)  } . 
\end{align*}
Consequently,
\begin{align*} 
	 {H}_{+} (m, h)  = \frac {1} {\pi^3 i}  \int_{\vepsilon-i\log T}^{\vepsilon+i\log T} \updelta_{\scriptscriptstyle +} (v) \int_{T-\varPi \log T}^{T+\varPi \log T}      \frac {    {\upphi (t)} t^{1+v} } { (\theta/X)^{v} h^{2v+2it} m^{v+2it} }    \frac {  \nd t \nd v} { v} + O  (  {\varPi T^{\vepsilon}}    ), 
\end{align*}
\begin{align*}
	{H}_{-} (m, h)  = \frac {1} {\pi^3 i}  \int_{\vepsilon-i\log T}^{\vepsilon+i\log T} \updelta_{\scriptscriptstyle -} (v) \int_{T-\varPi \log T}^{T+\varPi \log T}      \frac { (e / 2 t)^{ 4 i t} {\upphi (t)} t^{1+v} } { (mX)^{v}  \theta^{v-2it} h^{2v-2it} }    \frac {  \nd t \nd v} { v} + O  (  {\varPi T^{\vepsilon}}    ). 
\end{align*} 
Next, consider the phase functions in the inner integrals  
\begin{align*}
	f_{\scriptscriptstyle +} (t) = - 2 t \log (hm), \qquad f_{\scriptscriptstyle -} (t) = 4 t \log (\sqrt{\theta h} e / 2 t), 
\end{align*} 
whose derivatives read
\begin{align*}
	f_{\scriptscriptstyle +}' (t) = -2 \log (hm), \qquad f'_{\scriptscriptstyle -} (t) = 4 \log (\sqrt{\theta h} / 2 t)  . 
\end{align*}
Clearly, $ |f_{\scriptscriptstyle +}' (t) | \Gt 1$ if $h m > 1$ and $ |f_{\scriptscriptstyle -}' (t) | \Gt \log T $ as $\sqrt{h} < \sqrt{T}$. Moreover, we have $ f_{\scriptscriptstyle +}^{(j)} (x) = 0$ and $ f_{\scriptscriptstyle -}^{(j)} (x) \Lt_j 1/T^{j-1}$ for $j \geqslant 2$.   By applying Lemma \ref{lem: staionary phase}, with  $P = \varPi$, $Q = T$,  $Z = T$, and $ R = 1$ or $\log T$,  we infer that the integrals  are negligibly small, except for the case of $H_{+} (1, 1)$.  
\end{proof}

\begin{lem}
	We have 
	\begin{align}
		 H_+(1, 1) = \frac {2} {\pi \sqrt{\pi}}  \varPi T + O (\varPi / X). 
	\end{align}
\end{lem}

\begin{proof}
By shifting the integral contour in \eqref{3eq: V(y;t)} down to $\mathrm{Re} (v) = -1$, we have
\begin{align*}
	V_{1} (\theta/ X; t) = 1 + O \bigg( \frac 1 { (|t|+1) X } \bigg), 
\end{align*} 
by the Stirling formula. Therefore 
\begin{align*}
	H_+(1, 1) = \frac {2} {\pi^2}  \int_{-\infty}^{\infty}     {\upphi (t)}  t \nd t + O (\varPi/ X) , 
\end{align*} whereas 
	\begin{align*}
		   \int_{-\infty}^{\infty}   {\upphi (t)} t \nd t =    \varPi T \int_{-\infty}^{\infty}   \exp (-t^2)  \nd t =  {  \sqrt{\pi}}  \varPi T . 
	\end{align*}
\end{proof}

\begin{lem}\label{lem: H+-(m)}
	We have 
\begin{align}\label{6eq: D(m)}
	\SD_1^+ (m) + \SD_1^- (m) =  \frac {2} {\pi\sqrt{\pi}} \varPi T \cdot {\delta ({m, 1})} + O  \bigg(   \frac {\varPi T^{\vepsilon}  } {\sqrt{m}}   \bigg) . 
\end{align}
\end{lem}

\delete{\begin{remark}
	Actually, the integral in     may be evaluated explicitly by {\rm\cite[3.896 4, 3.952 1]{G-R}} and there will arise the exponential factor $\exp (-\varPi^2 \log^2 m)$.  
\end{remark}}

\subsection{Treatment of the Off-diagonal Sums}  \label{sec: off, 1st}

 Let us  consider  the off-diagonal sums $	\SO_1^{\pm} (m)$ given by \eqref{6eq: O(m)} and \eqref{6eq: Bessel, 1} (as well \eqref{6eq: V+(n;t)} and \eqref{6eq: V-(n;t)}).  
 
\vspace{5pt} 

\subsubsection{Initial Reductions}  \label{sec: reduction}

By \eqref{3eq: V(y;t), 1} in Lemma \ref{lem: AFE} along with  a smooth dyadic partition,  at the cost of a negligible error,  we may confine the $h$- and $n$-sums in \eqref{6eq: O(m)}  on the range $ h^2 n \sim N $ for dyadic $ N \Lt T^{1+\vepsilon} X^{\pm 1}$ (namely, $N$ is of the form $2^{j/2}$).  
Next, we insert the expression  of $ V_{1}     (  y ; t  ) $ as in \eqref{3eq: V(y;t), 2} in Lemma \ref{lem: AFE} into \eqref{6eq: V+(n;t)}, \eqref{6eq: V-(n;t)}, and \eqref{6eq: Bessel, 1}. After this, there arises the Bessel transform 
\begin{align*}
	H_{\pm} (x, y; v) = \frac{2 i } {\pi} \int_{-\infty}^{\infty} h_{\pm} (t; y; v) J_{2it} (x)  \frac {t \nd t} {\cosh (\pi t) } , 
\end{align*} 
in which
\begin{align*} 
	h_{\scriptscriptstyle +} (t; y; v) = \upgamma_1 (v, t) y^{-2it} \upphi (t) + \upgamma_1 (v, - t) y^{2it} \upphi (- t), 
\end{align*} 
\begin{align*} 
	h_{\scriptscriptstyle -} (t; y; v) = \upgamma_1 (v, - t) \upepsilon (t)  y^{-2it} \upphi (t) + \upgamma_1 (v,  t) \upepsilon (-t) y^{2it} \upphi (- t), 
\end{align*}  
for 
\begin{align*}\label{6eq: x, y}
	 x = \frac {4\pi \sqrt{mn}} {c}, \qquad y = \left\{ 
	 \begin{aligned}
	 &	h \sqrt{m n} , & & \text{ if } +, \\
	 & \sqrt{m} /   h \sqrt{n }, & & \text{ if } -, 
	 \end{aligned} \right. 
\end{align*}
(see \eqref{6eq: V+(n;t)} and \eqref{6eq: V-(n;t)}), and for $ v \in [\vepsilon -i \log T, \vepsilon + i \log T] $. Since the variable $v$ is not essential in our analysis, so it will usually be omitted in the sequel.  The Bessel integral $H_{\pm} (x, y ) $ has been studied thoroughly in \S \ref{sec: Bessel}; in particular, by Corollaries \ref{cor: u<T} and \ref{cor: xy, -}, along with the Weil bound \eqref{3eq: Weil}, we may further restrict the $c$-sum  in \eqref{6eq: O(m)}   to  $c  \Lt  C_{+}$ or $|c - C_{0} | \Lt C_- $ for 
\begin{equation} \label{7eq: C+-}
	C_{+} = \frac {m N} {h T}, \qquad C_0 = \frac {4\pi m T} {h \theta}, \quad  C_{-} = \frac {m  } {h } \varPi^{\vepsilon}  \bigg( \varPi + \frac {T} {\varPi}    \bigg) ;
\end{equation}  
in the notation of \S \ref{sec: Bessel}, we have $\rho T /\varPi = \varPi + T/\varPi$ and  we check here that
\begin{align*}
	w_{\scriptscriptstyle -} = \frac {\pi \theta hn} {2 T^2} \Lt \frac {T^{\vepsilon}} {T X}  
\end{align*}
 is indeed very small. 
We conclude that, up to  a negligible error, the off-diagonal sum $ \SO_1^{\pm} (m)  $ has bound 
\begin{align}\label{7eq: O1(m), 1}
	 \SO_1^{\pm} (m) \Lt  \frac {T^{\vepsilon}} {\sqrt{N}}  \max_{N \Lt T^{1+\vepsilon} X^{\pm 1} }   \sum_{h \Lt  \sqrt{N} }    \int_{\vepsilon-i\log T}^{\vepsilon+i\log T}   \big|\SO_1^{\pm} (m,  h; N;  v) \big|   \frac {\nd v} {v} , 
\end{align}
where, if we omit $v$ from the notation, 
\begin{align}
	\SO_1^{+} (m,   h; N) = \sum_{\, c \shskip \Lt C_{+} }     \sum_{n \sim N/h^2}   {a (n)}    \frac   {S(m,n; c)} c    \varww \bigg(  \frac {h^2n} {N} \bigg) {H}_{+}  \bigg(\frac {4\pi \sqrt{mn}} {c} ; h \sqrt{m n}  \bigg), 
\end{align}
\begin{align}
	\SO_1^{-} (m,   h; N) = \sum_{ |c-C_0| \Lt C_{-} }    \sum_{n \sim N/h^2}  \hspace{-3pt}  {a (n)}     \frac  {S(m,n; c)}   c    \varww \bigg(  \frac {h^2n} {N} \bigg) {H}_{-}  \bigg(\frac {4\pi \sqrt{mn}} {c} ; \frac {\sqrt{m} } {  h \sqrt{n} } \bigg), 
\end{align}
with
\begin{align*}
	\varww (x) = \frac {\varvv (x)} {x^{1/2+v}}, 
\end{align*}
for a certain fixed smooth weight $\varvv  \in C_c^{\infty}[1,2]$. Note that  $	\varww (x)$ is $\log T$-inert in the sense of Definition \ref{defn: inert}; indeed, 
we have uniformly 
\begin{align*}
	\varww^{(j) } (x) \Lt_{j, \vepsilon} \log^{j} T. 
\end{align*}

\subsubsection{Application of Lemma {\ref{lem: Bessel}}}  \label{sec: O+, 1}
By opening the Kloosterman sum $S (m, n;c)$ as in \eqref{3eq: Kloosterman} and inserting  the formula of the Bessel integral $H_{+} (x, y)$ as in  Lemma \ref{lem: Bessel}, we have 
\begin{align}\label{7eq: O1(m), 2}
	 \SO_1^{+} (m,   h; N) \Lt  {\varPi T^{1+\vepsilon} }  	  \sum_{\pm}  \sum_{\, c \shskip \Lt C_{+} }  \frac 1 c \, \sumx_{\valpha (\mathrm{mod}\, c)}   \int_{-\varPi^{\vepsilon}/ \varPi}^{\varPi^{\vepsilon}/ \varPi} \left| \SO_{\pm}^{+}  (r; \valpha ;  m, c, h; N) \right|    \nd r  , 
\end{align}
with  
\begin{align}
\SO_{\pm}^{+} (r; \valpha ; m, c, h; N) = \sum_{n \sim N/h^2}   {a (n)}      e \Big(    \frac {\valpha  n} {c}   \Big)   e \bigg(       \frac {hmn} {c} \exp  (r)  \bigg)   \varww_{\scriptscriptstyle \pm}  \bigg(  \frac {h^2n} {N} \bigg), 
\end{align}  
where  $\varww_{\scriptscriptstyle +} (x) = \varww (x) $ and $\varww_{\scriptscriptstyle -} (x) = \overline{\varww (x) }$. 

\vspace{5pt} 

\subsubsection{Estimation by the Wilton Bound}  \label{sec: O+, 2}
As the weights $\varww_{\pm} (x)$ are  $\log T$-inert, by partial summation, it follows from the Wilton bound in Lemma \ref{lem: Wilton} that 
\begin{align*}
	 \SO_{\pm}^{+}  (r; \valpha ;  m, c, h; N) \Lt \frac {\sqrt{N} \log^2 T}    h,
\end{align*}
hence  
\begin{align*}
	  \SO_1^{+} (m,   h; N) \Lt \frac {C_{+}  \sqrt{N} T^{1+\vepsilon}} {h} , 
\end{align*}
by \eqref{7eq: O1(m), 2},  and we conclude with the next lemma by  \eqref{7eq: C+-} and \eqref{7eq: O1(m), 1}. 

\begin{lem}\label{lem: O1(m)}
	We have 
	\begin{align}\label{6eq: O1(m), final}
		 \SO_1^+ (m)   \Lt  m T^{1+\vepsilon} X. 
	\end{align}
\end{lem}

\subsubsection{Application of Lemma {\ref{lem: analysis of I, -}}}  By applying the Weil bound \eqref{3eq: Weil}, the Deligne bound \eqref{3eq: Deligne}, and the bound  for $H_-(x, y)$ as in \eqref{4eq: bound for H-} in Lemma \ref{lem: analysis of I, -}, we have
\begin{align*}
	 \SO_1^{-} (m,   h; N) & \Lt T^{\vepsilon} \frac {\varPi T } {\sqrt{T} + \varPi }   \sum_{ |c-C_0| \Lt C_{-} }  \frac  { {\tau (c)}}  {\sqrt{c}}    \sum_{n \sim N/h^2}  \hspace{-3pt}  {\tau (n)}  \sqrt{( n, c)}    \\
	 & \Lt T^{\vepsilon} \frac {\varPi T } {\sqrt{T} + \varPi } \frac {C_-  } { \sqrt{C_0}} \frac {N} {\,h^2} , 
\end{align*}
and hence by \eqref{7eq: C+-}, 
\begin{align*}
	\SO_1^{-} (m,   h; N)   \Lt T^{\vepsilon} \frac {\sqrt{m T} N (\sqrt{T} + \varPi)} {h^2\sqrt{h} } .
\end{align*}
 On inserting this into \eqref{7eq: O1(m), 1}, we obtain the following bound for $ \SO_1^{-} (m) $.

\begin{lem}\label{lem: O1(m)-}
	We have 
	\begin{align}\label{6eq: O1(m)-, final}
		\SO_1^- (m)   \Lt  T^{\vepsilon} \frac {\sqrt{m} T (\sqrt{T} + \varPi)} {\sqrt{X}}. 
	\end{align}
\end{lem}


\subsection{Conclusion} 
In view of  \eqref{6eq: C+E=D+O}, 
it follows from Lemmas \ref{lem: Eis E1,2}, \ref{lem: H+-(m)},   \ref{lem: O1(m)}, and \ref{lem: O1(m)-} that 
\begin{align*}
	\SC_1 (m) = \frac {2} {\pi\sqrt{\pi}} \varPi T \cdot {\delta ({m, 1})}  + O \bigg( T^{\vepsilon} \bigg( m T  X + \frac {\sqrt{m} T (\sqrt{T} + \varPi)} {\sqrt{X}} \bigg)\bigg), 
\end{align*}
so, on the choice $X = \sqrt[3]{ (T+\varPi^2) /m  }$ we obtain the asymptotic formula in Theorem \ref{thm: C1(m)}.

\section{The Twisted Second Moment} \label{sec: 2nd moment} 


  In this section, for $m_1, m_2 $ square-free, with $(m_1m_2, q) = 1$, we proceed to establish   asymptotic formula for the twisted second moment ${\SC}^{\natural }_{2}  ( m_1, m_2 )$ as defined in \eqref{5eq: C2(m)}. The bi-variable notation in \S \ref{sec: bi-varialbe} will be used extensively to simplify the exposition. 

Our method does not only use the Wilton bound, but also the Luo identity, the Poisson, and Vorono\"i summation inspired by the works of Luo, Iwaniec, Li, and Young \cite{Luo-Twisted-LS,Iwaniec-Li-Ortho,Young-GL(3)-Special-Points}.

\subsection{Application of Kuznetsov Formula} 

First, in the bi-variable notation, by the approximate functional equations \eqref{3eq: AFE, 2} and \eqref{3eq: AFE, E(t), 2} (here, we need to expand $2 \mathrm{c}_t ( \boldsymbol{m}  ) \mathrm{c}_t (  \boldsymbol{n}) = \mathrm{c}_t ( \boldsymbol{m} \boldsymbol{n}) + \mathrm{c}_t ( \boldsymbol{m} / \boldsymbol{n}) $ (see \eqref{3eq: ct(n)}) and use the symmetry in the   $\boldsymbol{n}$-sums),  along with \eqref{3eq: RS coeff},  we obtain 
\begin{align*}  
		{\SC}^{\natural }_{2}  ( \boldsymbol{m}) +  {\SE}_{2}  (\boldsymbol{m})    =   
		\sum_{\boldsymbol{h}}   \sum_{\boldsymbol{n}}   \frac{\varepsilon_q (\| \boldsymbol{h} \|) a (\boldsymbol{n})} {\| \boldsymbol{h} \| \sqrt{\| \boldsymbol{n} \|} }  \Bigg\{   \sum_{j =1}^{\infty}  \omega_j   {  \lambda_j ( \boldsymbol{m}   )  \lambda_j (\boldsymbol{n}  )}     {V}_{2} ( \boldsymbol{m},  \boldsymbol{h}^2 \boldsymbol{n}   ; t_j) & \\
		     + \frac 1 {\pi} \int_{-\infty}^{\infty}  \omega(t) \tau_{  it} ( \boldsymbol{m} ) \tau_{  it} (\boldsymbol{n})       {V}_{2} (\boldsymbol{m},  \boldsymbol{h}^2 \boldsymbol{n}  ; t)  \nd t & \Bigg\}  ,  
\end{align*} 
where 
\begin{align*}
  {V}_{2} (\boldsymbol{m}, \boldsymbol{n} ; t) =  2  V_2 ( \theta^2  \| \boldsymbol{n} \|  ; t) \mathrm{c}_t ( \boldsymbol{m} \boldsymbol{n}) 
 (  \upphi (t) +    \upphi (- t)  ). 
\end{align*}  
Note that $ {V}_{2} (\boldsymbol{m},  \boldsymbol{n} ; t)$ is even in $t$ and real-valued. 
By the Hecke relation \eqref{3eq: Hecke} and\eqref{3eq: Hecke, tau}, we have  
\begin{align*}  
	{\SC}^{\natural }_{2}  (\boldsymbol{m}) + {\SE}_{2}  (\boldsymbol{m})  =  \sum_{\boldsymbol{d} \boldsymbol{f} = \boldsymbol{m} }   
	\sum_{\boldsymbol{h}}    \sum_{  \boldsymbol{n}}  \frac{ \varepsilon_q (\| \boldsymbol{h} \|) a (\boldsymbol{d n}   )  } {\| \boldsymbol{h} \| \sqrt{  \| \boldsymbol{d n} \|} }  \Bigg\{   \sum_{j=1}^{\infty}  \omega_j    \lambda_j ( \boldsymbol{f n}  )   {V}_2 (\boldsymbol{m}, \boldsymbol{d} \boldsymbol{h}^2 \boldsymbol{  n}  ; t_j) & \\
	 +    \frac 1 {\pi}   \int_{-\infty}^{\infty}   \omega(t)  \tau_{  it}    ( \boldsymbol{f n}  )    {V}_2 (\boldsymbol{m}, \boldsymbol{d} \boldsymbol{h}^2 \boldsymbol{  n}  ; t ) \nd t  &  \Bigg\}   .
\end{align*} 
Now an application of the Kuznetsov formula in Lemma \ref{lem: Kuznetsov} yields
\begin{align}\label{7eq: C+E=D+O}
	 {\SC}^{\natural }_{2}  ( \boldsymbol{m}) + {\SE}_{2}  (\boldsymbol{m}) = \SD_2 (\boldsymbol{m})  + \SO_2 (\boldsymbol{m}), 
\end{align}
with the diagonal 
\begin{align}\label{7eq: D2(m)}
	\SD_2 (\boldsymbol{m}) = \sum_{\boldsymbol{d} \boldsymbol{f} = \boldsymbol{m} } 
	\sum_{\boldsymbol{h}}   \sum_{  \boldsymbol{n}}  \frac{ \varepsilon_q (\| \boldsymbol{h} \|) a (\boldsymbol{d n}   )  }  { \| \boldsymbol{h} \| \sqrt{  \| \boldsymbol{d n} \|} }  \delta ({\boldsymbol{f n} })  {H}_2  (  \boldsymbol{d h}, \boldsymbol{d} \boldsymbol{h}^2 \boldsymbol{ n}), 
\end{align} 
where
\begin{align}\label{7eq: H2(m)}
	  {H}_2 (\boldsymbol{h},   \boldsymbol{  n} ) = \frac 4 {\pi^2 }  \int_{-\infty}^{\infty}  V_2 (   \theta^2  \| \boldsymbol{n} \|  ; t)  \mathrm{c}_{2t} ( \boldsymbol{h} ) 
	    \upphi (t)   \tanh(\pi t) t \nd t ,
\end{align}
and the off-diagonal 
\begin{align}\label{7eq: O(m)}
	\begin{aligned}
		\SO_2 (\boldsymbol{m}) = \sum_{\boldsymbol{d} \boldsymbol{f} = \boldsymbol{m} } 
		\sum_{\boldsymbol{h}}  \sum_{  \boldsymbol{n}}  \frac{\varepsilon_q(\| \boldsymbol{h} \|) a (\boldsymbol{d n}   )  } {\| \boldsymbol{h} \| \sqrt{  \| \boldsymbol{d n} \|} }     \sum_{c  } \frac{S( \boldsymbol{f n}   ; c)} {c}   {H}_2 \bigg(\frac{4\pi\sqrt{\| \boldsymbol{f n} \|  }}{c};  \boldsymbol{m},  \boldsymbol{d}  \boldsymbol{h}^2 \boldsymbol{n} \bigg) ,
	\end{aligned}
\end{align}
where  
\begin{align}\label{7eq: H2(x)}
	  {H}_2  ( x; \boldsymbol{m},   \boldsymbol{n} ) = \frac {4 i} {\pi}   \int_{-\infty}^{\infty}    	 {V}_2 (   \theta^2 \| \boldsymbol{n} \| ;  t)    \mathrm{c}_t (   \boldsymbol{m n})
	 (  \upphi (t) +    \upphi (- t)  ) J_{2it} (x)  \frac {t \nd t} {\cosh (\pi t) } .
\end{align}

\subsection{Asymptotic for the Diagonal Sum}

Set  
\begin{align}\label{7eq: f}
  f = \mathrm{gcd}(\boldsymbol{f}), \qquad    \boldsymbol{f^{\star}} = \boldsymbol{f} / f, 
\end{align}
then $  \delta ({\boldsymbol{f n} }) = 1 $  if and only if there is $n$ such that 
 $	\boldsymbol{n} = n   \widetilde{\boldsymbol{f^{\star} }}  .  $ 
Recall from \S \ref{sec: bi-varialbe} that $ \widetilde{\boldsymbol{f^{\star}}} $ is  dual to $\boldsymbol{f^{\star}}$. Therefore \eqref{7eq: D2(m)} is turned into 
\begin{align}\label{7eq: D2(m), 2.0}
	\SD_2 (\boldsymbol{m}) = \sum_{\boldsymbol{d} \boldsymbol{f} = \boldsymbol{m} } \frac {f} {\sqrt{\| \boldsymbol{m} \|}}  
	\sum_{\boldsymbol{h}}   \sum_{   {n}}  \frac{\varepsilon_q(\| \boldsymbol{h} \|) a (n    \boldsymbol{d } \widetilde{\boldsymbol{f^{\star} }}    )  }  { n \| \boldsymbol{h} \| }     {H}_2  (  \boldsymbol{d h}, n   \boldsymbol{d} \widetilde{ \boldsymbol{f^{\star}} } \boldsymbol{h}^2 ), 
\end{align} 
In view of \eqref{3eq: V(y;t), 1} and \eqref{7eq: H2(m)}, the summation  in \eqref{7eq: D2(m), 2.0} may be restricted effectively to $ n  \| \boldsymbol{h} \| \sqrt{\| \boldsymbol{m} \|} / f  \leqslant T^{1+\vepsilon}$. It follows from \eqref{3eq: V(y;t), 1} and  \eqref{5eq: phi} that $ H_2 (\boldsymbol{h}, \boldsymbol{n}) = O (\varPi T) $, so trivial estimation by \eqref{3eq: Deligne} yields 
\begin{align*}
\SD_2	(\boldsymbol{m}) = O \bigg(  \frac { \varPi T^{1+\vepsilon} } {\sqrt{\|\boldsymbol{m^{\star}}\|} } \bigg), 
\end{align*}
where $\boldsymbol{m^{\star}} = \boldsymbol{m} / \mathrm{gcd} (\boldsymbol{m})$. 
However, our aim here is to derive an asymptotic formula for $\SD_2	(\boldsymbol{m})$. 

\begin{lem}\label{lem: H2(m,n)}
	We have 
	\begin{align}
	 {H}_2 (\boldsymbol{h},   \boldsymbol{  n} )  = O    \bigg(  \frac {\varPi T^{\vepsilon}}  { \| \boldsymbol{n} \|^{\vepsilon}}  \bigg), 
	\end{align}
if 
\begin{align}\label{7eq: h1/h2 - 1}
	 | \langle \boldsymbol{h} \rangle - 1  | >  \frac {\varPi^{\vepsilon}} {\varPi}, \qquad \text{{\rm(}$ \langle \boldsymbol{h} \rangle = h_1/h_2${\rm)}} . 
\end{align}
\end{lem}

\begin{proof}
	 Let us argue as in the proof of Lemma \ref{lem: H(m,h)}. First, by \eqref{3eq: V(y;t), 2}, with $U = \log T$, we have 
	 \begin{align*}
	 {H}_2 (\boldsymbol{h},   \boldsymbol{  n} )   = \frac {2} {\pi^3 i}  \int_{\vepsilon-i\log T}^{\vepsilon+i\log T}  \int_{T-\varPi \log T}^{T+\varPi \log T}      \frac { \upgamma_{2} (v, t)  \mathrm{c}_{2t} ( \boldsymbol{h}  )   {\upphi (t)} t } { \theta^{2 v} \| \boldsymbol{n} \|^v }    \frac {  \nd t \nd v} { v} + O (T^{-A}) . 
	 \end{align*}
 For $t , v$ on the integral domains, by applying Stirling's formula \eqref{app1: Stirling, 1} to \eqref{3eq: defn G2}, we have 
 \begin{align*} 
 	\upgamma_{2} (v, t) = \updelta_{2} (v) \big( t^{2 v} + O  (    {T^{\vepsilon} } / T  ) \big) , 
 \end{align*} 
 for 
 \begin{align*}
 	\updelta_{2} (v) =	 \exp (v^2 + 2 \log 2 \cdot  v ) \cdot  {\Gamma  ( k + v )^2  } / {\Gamma (k)^2  } . 
 \end{align*}
Consequently,
\begin{align*}
	{H}_2 (\boldsymbol{h},   \boldsymbol{  n} )   = \frac {2} {\pi^3 i}  \int_{\vepsilon-i\log T}^{\vepsilon+i\log T} \frac { \updelta_{2} (v)  } { \theta^{2 v} \| \boldsymbol{n} \|^v } \int_{T-\varPi \log T}^{T+\varPi \log T}        {   \mathrm{c}_{2t} ( \boldsymbol{h}  )   {\upphi (t)} t^{1+2v} }   \nd t  \frac {   \nd v} { v} + O \bigg(\frac {\varPi T^{\vepsilon}} { \| \boldsymbol{n} \|^{\vepsilon}} \bigg) . 
\end{align*}
The inner integral is a cosine Fourier integral of phase $ 2 t \log \, \langle \boldsymbol{h} \rangle $ (see \eqref{3eq: ct(n)}), so, by repeated partial integration (or by Lemma  \ref{lem: staionary phase}), we infer that it is negligibly small since $ | \log \, \langle \boldsymbol{h} \rangle | \Gt \varPi^{\vepsilon} / \varPi $ by \eqref{7eq: h1/h2 - 1}.   
\end{proof}

In view of Lemma \ref{lem: H2(m,n)}, those terms with $ |\langle \boldsymbol{dh} \rangle - 1| > \varPi^{\vepsilon}/ \varPi $ contribute at most $ O (\varPi T^{\vepsilon}  / \sqrt{\|\boldsymbol{m^{\star}}\| }) $, while the rest of the off-diagonal terms are very close to the diagonal 
and are in total $ O ( T^{1+\vepsilon}  / \sqrt{\|\boldsymbol{m^{\star}}\| }) $.


\vspace{5pt} 

\subsubsection{Treatment of the Diagonal Terms} 
It is left to consider those diagonal terms with $ \delta (\boldsymbol{d h}) = 1 $. Similar to \eqref{7eq: f}, 
we set 
\begin{align}\label{7eq: d}
	d = \mathrm{gcd}(\boldsymbol{d})	 , \qquad   \boldsymbol{d^{\star}} = \boldsymbol{d} / d,  
\end{align}
then $ \delta (\boldsymbol{d h} )   = 1 $  if and only if there is $h$ such that 
$	\boldsymbol{h} = h \widetilde{\boldsymbol{d^{\star}}}  . 
$ 
Therefore \eqref{7eq: D2(m), 2.0} is turned into
\begin{align}\label{7eq: D2(m), 3.0}
	 \SD_{2} (\boldsymbol{m}) =  \SD_{20} (\boldsymbol{m})  +   O \bigg(    \frac {T^{1+\vepsilon} } {\sqrt{\|\boldsymbol{m^{\star}}\|} } \bigg) , 
\end{align}
with
\begin{align}\label{7eq: D2(m), 3.1}
	\SD_{20} (\boldsymbol{m}) =  	\sum_{\boldsymbol{d} \boldsymbol{f} = \boldsymbol{m} } \frac {   f } {\|\boldsymbol{d^{\star}} \| \sqrt{\| \boldsymbol{m} \|}}  
	\sum_{ {h}}   \sum_{   {n}}  \frac{\varepsilon_q(h  ) a (n  \cdot \boldsymbol{d } \widetilde{\boldsymbol{f^{\star} }}    )  }  { h^2 n   }     {H}_2  (  h^2 n  \cdot \boldsymbol{d} \widetilde{ \boldsymbol{d^{\star}} }{}^2  \widetilde{ \boldsymbol{f^{\star}} } ) , 
\end{align} 
where, after the removal of $\tanh (\pi t)$, 
\begin{align}\label{7eq: H2(n)}
	H_2 ( \boldsymbol{n} ) = \frac 4 {\pi^2 }  \int_{-\infty}^{\infty}  V_2 (   \theta^2  \| \boldsymbol{n} \|  ; t) 
	\upphi (t)   t \nd t . 
\end{align}
Now it follows from \eqref{3eq: V(y;t)},   \eqref{7eq: D2(m), 3.1}, and \eqref{7eq: H2(n)} that 
\begin{align}\label{7eq: D20}
	 \SD_{20} (\boldsymbol{m}) = \frac {2} {\pi^3 i} 
	 \sum_{\boldsymbol{d} \boldsymbol{f} = \boldsymbol{m} } 
	 \int \upphi (t)   t  \int_{(3)} \upgamma_{2}  (v, t) \frac { {  D_{\boldsymbol{d} \widetilde{\boldsymbol{f^{\star}}}}} (1+2 v, \vQ)} {\theta^{2v}  (\|\boldsymbol{d^{\star }}{}^2  \boldsymbol{m} \| /  f^2 )^{1/2+v } }   \frac {\nd v \nd t} {v}  	  , 
\end{align}
\begin{align}
  	D_{\boldsymbol{r}} (s, \vQ) = \zeta_{q} (2 s)  \sum_{n=1}^{\infty} \frac {a (r_1 n) a (r_2 n) }  {n^{s}}. 
\end{align} 
Clearly, in our setting,  $\|\boldsymbol{r}\| = \|\boldsymbol{d} \widetilde{\boldsymbol{f^{\star}}}\| $ is  cube-free and $(\|\boldsymbol{r}\|, q) = 1$. 
Note that $ D_{1, 1}  (s, \vQ)$ is exactly the Rankin--Selberg  $L (s, \vQ\otimes\vQ)$. 

For the rest of this sub-section, we shall stop using the bi-variable notation. As such $(\cdot , \cdot)$ will stand for the greatest common divisor.

\delete{As $m_1$ and $m_2$ are square-free, we have
\begin{align}
	(m_1, m_2) = d  { } f (d_1^{\star}, f_2^{\star}) (d_2^{\star}, f_1^{\star}); 
\end{align}
here $(\cdot, \cdot)$ means the greatest common divisor. }

\begin{lem}
	Let $ L (s, \vQ\otimes\vQ)$ be given as in {\rm \eqref{3eq: L(s, QQ)}}. For $  r_1 r_2 $ cube-free,   define \begin{align}\label{7eq: r} 
		d =   (r_1, r_2), \qquad r_1^{\star} = r_1/d, \quad r_2^{\star} = r_2/d, \qquad   r_1^{\star} r_2^{\star} = r      r_{ \square}^2 ,
	\end{align}
	for  $r     $ and $ r_{ \square} $ square-free.   Assume  that $(r_1 r_2, q) = 1$ and $ d  $ is also square-free. 
	Then
	\begin{align}\label{7eq: D(s)}
		D_{r_1, r_2} (s, \vQ) =   E_{d} (s, \vQ) E^1_{r     } (s, \vQ) E_{r_{ \square}}^{2} (s, \vQ) L (s, \vQ\otimes\vQ),  
	\end{align}
	for the finite Euler products defined by
	\begin{align}\label{7eq: E, 1}
		E_d (s, \vQ) = \prod_{p | d}   \frac { a(p)^2 - a (p^2)   / p^{s} +1 / p^{2s}} {1+ 1 / p^{s}}  , 
	\end{align}
	\begin{align}\label{7eq: E, 2}
		E^1_{r     } (s, \vQ) = \prod_{p | r     } \frac {a(p)} {1+ 1 / p^{s} }, \qquad E_{r_{ \square}}^{2} (s, \vQ) = \prod_{p | r_{ \square} }    \frac {a(p^2) - 1 / p^{s} } {1+ 1 / p^{s} }  .
	\end{align}
\end{lem}

\begin{proof}
	For prime $p \nmid q$    and $\mathrm{Re}(s) > 1$, define  
	\begin{align*}
		C_p   (s) = \sum_{\vnu = 0}^{\infty} \frac {a (p^{\vnu+1})^2  } {p^{\vnu s}}, \qquad 	C_p^{\mu}  (s) = \sum_{\vnu = 0}^{\infty} \frac {a(p^{\vnu}) a (p^{\vnu+\mu})} {p^{\nu s}}, \quad \mu = 0, 1, 2 . 
	\end{align*}
	We just need to calculate the quotients $ C_p   (s) /  C_p^{0}  (s)$, $ C_p^1   (s) /  C_p^{0}  (s) $, and $ C_p^2  (s) /  C_p^{0}  (s) $. 
	To this end, we invoke the Hecke relation (see \eqref{3eq: Hecke, Q}):
	\begin{align*}
		a (p^{\vnu+\mu}) = a (p^{\vnu}) a (p^{\mu}) -  a (p^{\vnu-1}) a (p^{\mu-1}), 
	\end{align*} 
	if we set $ a (p^{-1}) = 0 $ for simplicity. Let us write $X = 1/p^s$.  It follows that
	\begin{align*}
		C_p (s) = (a(p)^2 + X ) C^0_p (s) - 2 a(p) X C^1_0 (s) , 
	\end{align*} 
	\begin{align*}
		C^1_p(s) = a(p) C^0_p (s) - X C_p^1 (s),  
	\end{align*}
	\begin{align*}
		C^2_p(s) = a(p^2 ) C^0_p (s) - a(p) X C_p^1 (s),  
	\end{align*} 
	and some direct calculations yield 
	\begin{align*}
		\frac{C_p (s)} {C_p^{0} (s) } =  \frac { a(p)^2 - a (p^2) X +X^2} {1+ X}  , \qquad \frac{C^1_p(s)} {C_p^{0} (s) }  = \frac {a(p)} {1+X} , \qquad \frac {C^2_p(s)} {C_p^{0} (s)  } =  \frac {a(p^2) - X } {1+ X } ,
	\end{align*}
	as desired. Note here that $ a (p)^2 = a (p^2) + 1$. 
\end{proof}

\subsubsection{Conclusion} 
Let us denote
\begin{align}\label{7eq: m, ds}
	m = (m_1, m_2), \qquad d_{\star} = \frac {d_1^{\star} d_2^{\star} } { r_{\square} }  . 
\end{align} Observe that in the notation of \eqref{7eq: f}, \eqref{7eq: d},  \eqref{7eq: r}, and \eqref{7eq: m, ds}, we have 
 \begin{align} \label{7eq: m, rf}
 	  m = d {} f r_{\square}, \quad r_{\square} =  (d_1^{\star}, f_2^{\star})   (d_2^{\star}, f_1^{\star}), \quad r      = m_1^{\star} m_2^{\star} = \frac {m_1 m_2} {  m^2 }, 
 \end{align}
while 
\begin{align}\label{7eq: quotient}
	\frac {f} { d_1^{\star} d_2^{\star} \sqrt{m_1 m_2} } = \frac 1 { \sqrt{r     }} \frac 1 {d_{\star} d  r_{\square}^2 } .  
\end{align}
Thus, in view of \eqref{7eq: D(s)} and \eqref{7eq: quotient}, it is more natural to use $ d $,  $r_{\square}$, and  $d_{\star}$
as the new variables of summation. Actually, the $ d_{\star} $-sum over $d_{\star} | r     $ will split out as  certain divisor functions of $r     $ (see \eqref{7eq: Sigma}), so at the end, the main term will be simplified into double sums over $ d r_{\square} | m $ (see \eqref{7eq: B(m), 1}); these sums may be considered as very complicated divisor functions of $m$ as well. 
More explicitly, we have the following asymptotic formula for $  \SD_{2} (m_1, m_2) $. 
\begin{lem}\label{lem: main term}
Assume $m_1$, $m_2$ square-free with $(m_1 m_2, q) = 1$. Set \begin{align}
	 m = (m_1, m_2), \qquad r      = \frac {m_1m_2}  {m^2}. 
\end{align} 
Let  $\gamma_{0} = \gamma_0(\vQ)$ and $\gamma_{1} = \gamma_1 (\vQ)$ be the constants as in Definition \ref{defn: constants, 1}, and define the constant $\gamma_0' = \gamma_{0}' (\vQ)$ to be 
\begin{align}\label{7eq: gamma0}
	\gamma_0'  =    \gamma_0 + \gamma_1   \frac {\Gamma'(k)} {\Gamma (k)} + \gamma_1 \log    \Big( \frac {q} {2\pi^2}  \Big)   . 
\end{align} 
Then 
	 \begin{align}\label{7eq: D2, final}
	 \begin{aligned}
	 		 \SD_{2} (m_1, m_2) = \frac {4 a(r)} {\pi \sqrt{\pi r} } \varPi T \bigg(  \bigg(\gamma_1 \log \frac {T} {\sqrt{r}}  + \gamma_{0}'  \bigg) \SB (m)  + \gamma_1 \SB ' (m)  \bigg) & \\
	 	   + O \bigg(  T^{\vepsilon} \bigg(  \varPi +  \frac { T  } {\sqrt{ r      } } +   \frac {  \varPi^3  } {T \sqrt{ r      } }  \bigg) \bigg)   & , 
	 \end{aligned}
	 \end{align}
where    
\begin{align}\label{7eq: B(m), 1}
	\SB (m) =   {\sum_{ d r_{\square} | m } } 
	\frac{ E(d) E^{2} (r_{ \square}) } {  d r_{\square}^2 } , 
	\ \ \  
	\SB'  ( m) =    {\sum_{ d r_{\square} | m } } 
	\frac{ E(d) E^{2} (r_{ \square}) } {  d r_{\square}^2 }    \Big(    \breve{E} (d) + \breve{E}^{2} (r_{ \square})
	- \log  (    d r_{\square}^2  )  \Big), 
\end{align}
for 
\begin{align}
	\label{7eq: E(d), 2}	E ({d}) & =     \prod_{p | d} 
\frac { a(p)^2 - a (p^2) /p +1/p^2} {1+ 1/p},  \ \ 
\breve{E} (d)   =   \sum_{p | d} 
\log p \bigg(    \frac { a(p^2) p - 2} {a(p)^2 p^2 - a(p^2) p + 1 } + \frac {1} {p + 1} \bigg) ,
\end{align}
\begin{align}\label{7eq: E2(r)}
E^{2} (r_{ \square})  = \prod_{p | r_{ \square} }    \frac {a(p^2) - 1/p} {1+ 1/p }   ,    \qquad \breve{E}^{2} (r_{ \square})  = \sum_{p | r_{ \square} }  \log p \bigg( \frac {1}  {a(p^2) p - 1  } + \frac {1} {p+1}\bigg) . 
\end{align}

\end{lem}

\begin{proof}
By applying  \eqref{7eq: D(s)} and \eqref{7eq: quotient}  to the expression \eqref{7eq: D20}, the inner  integral reads: 
\begin{align*}
\int_{(3)} \upgamma_{2}  ( v, t) \frac {  E_{d} (1+2 v, \vQ) E^1_{r     } (1+2 v, \vQ) E_{r_{ \square}}^{2} (1+2 v, \vQ) L (1+2 v , \vQ\otimes\vQ)  } {\theta^{2v}  (\sqrt{r     } d_{\star} d r_{\square}^2  )^{1+2v } }   \frac {\nd v  } {v}  . 
\end{align*}
We shift the integral contour from $\mathrm{Re}(v) = 3$ down to $\mathrm{Re} (v) = \vepsilon -1/2$ and collect the residue at the double pole $v = 0$. By the Stirling formula, the resulting integral has bound  $O ((|t|+1)^{\vepsilon-1}  )$ and it yields the error $O (\varPi T^{\vepsilon})$   in \eqref{7eq: D2, final}, while, by \eqref{3eq: Laurent},  \eqref{3eq: defn G2},  \eqref{7eq: E, 1}, and \eqref{7eq: E, 2},  the residue equals
\begin{align*}
\begin{aligned}
	 \frac{E^1 ({r     }) E(d) E^{2} (r_{ \square}) } { \sqrt{r     } d_{\star} d r_{\square}^2 } \bigg\{ \gamma_{0} + \gamma_{1}   \bigg( \frac{\upsigma  (t)} 2   +  \breve{E}^1 ({r     }) + \breve{E} (d) + \breve{E}^{2} (r_{ \square})
	   - \log \big(\theta \sqrt{r     } d_{\star} d r_{\square}^2 \big)   \bigg)    \bigg\} , 
\end{aligned}
\end{align*}
with 
\begin{align}\label{7eq: E1(r)}
	E^1 ({r     })  = a(r     )  \prod_{p | r      } 
\frac {1} {1+ 1/ p }, \qquad \breve{E}^1 ({r     })  =  \sum_{p | r     } 
\frac { \log p } {p+ 1}    ,  
\end{align} 
$ E(d)$, $\breve{E} (d)$, $ E^{2} (r_{ \square}) $,  $\breve{E}^{2} (r_{ \square})$ similarly defined in  \eqref{7eq: E(d), 2} and \eqref{7eq: E2(r)}, and
\begin{align*}
\upsigma  (t) = 2 \psi (k)  + \psi (k+2it) + \psi (k-2it), 
\end{align*} 
where $\psi (s) = \Gamma'(s)/ \Gamma (s)$ is the di-gamma function. 
By Stirling's formula \eqref{app1: Stirling, 2}, 
\begin{align*}
	\upsigma (t) = 2 \psi (k) + 2 \log \sqrt{k^2+4t^2} + O \bigg( \frac 1 {t^2+1} \bigg). 
\end{align*}
Then the $t$-integrals may be evaluated easily by 
\begin{align*}
	   \int_{-\infty}^{\infty}   {\upphi (t)} t \nd t = {  \sqrt{\pi}}  \varPi T , \quad \int_{-\infty}^{\infty}   {\upphi (t)} t \log \sqrt{k^2+4t^2}  \nd t = {  \sqrt{\pi}}  \varPi T \log (2T)   + O \bigg(\frac {\varPi^3} T  \bigg) .
\end{align*}
It is clear that  the error after this process is $O (\varPi^3T^{\vepsilon}/ T {\displaystyle \sqrt{r     }})$. Let us also recollect the error  $O (T^{1+\vepsilon}/ {\displaystyle \sqrt{r     }})$ from \eqref{7eq: D2(m), 3.0}. Thus the error bound in \eqref{7eq: D2, final}  has been established. 

As for the main term, by some rearrangements, we have 
\begin{align*}
	\frac {4 } {\pi \sqrt{\pi} } \varPi T \big(  \log T \cdot  \SD  (r     ) \SB (m)  + \SD '  (r     ) \SB (m)   + \SD  (r      ) \SB ' (m)  \big), 
\end{align*}
where  
\begin{align*} 
		\SD  (r     )   = \gamma_1 \frac {   E^1 ({r     }) \varSigma  ( r      )    } {\sqrt{r     }}   ,  \qquad 	\SB (m) =   {\sum_{ d r_{\square} | m } } 
		\frac{ E(d) E^{2} (r_{ \square}) } {  d r_{\square}^2 } , 
\end{align*}
\begin{align*}
	\SD ' (r      )   =  \frac {   E^1 ({r     }) } {\sqrt{r     }}  \lp  \lp \gamma_0' +   \gamma_1     \breve{E}^1 ({r     }) -  \gamma_1 \log    \sqrt{r     }      \rp \varSigma  ( r      ) - \gamma_{1} \breve{\varSigma}  ( r      ) \rp    , 
\end{align*} 
\begin{align*}
	\SB'  ( m) =    {\sum_{ d r_{\square} | m } } 
	\frac{ E(d) E^{2} (r_{ \square}) } {  d r_{\square}^2 }    \left(    \breve{E} (d) + \breve{E}^{2} (r_{ \square})
	- \log  (    d r_{\square}^2  )  \right),
\end{align*}
for 
\begin{align}\label{7eq: Sigma}
	\varSigma   ( r      ) =	 {\sum_{d_{\star} | r      }  } 
	\frac {1} {d_{\star}}, \qquad  \breve{\varSigma}   ( r      ) =	 {\sum_{d_{\star} | r      }  } 
	\frac {\log d_{\star} } {d_{\star}} . 
\end{align} 
Moreover, in view of \eqref{7eq: E1(r)} and \eqref{7eq: Sigma}, it turns out that 
\begin{align*}
	E^1 ({r     }) \varSigma  ( r      ) = a (r) , \qquad  \breve{E}^1 ({r     }) \varSigma  ( r      ) = \breve{\varSigma}   ( r      ), 
\end{align*} hence
\begin{align*}
	 \SD  (r     )   = \gamma_1 \frac {  a (r)   } {\sqrt{r     }} , \qquad \SD ' (r      )   =  \frac {   a (r) } {\sqrt{r     }}     \lp \gamma_0'  -  \gamma_1 \log    \sqrt{r     }      \rp ,
\end{align*}
and we conclude with the main term in  \eqref{7eq: D2, final}. 
\end{proof}

For later use, we record some simple results on $ \SB (m) $ and $\SB' (m)$. 

\begin{lem}\label{lem: B(m)}
	We have
	\begin{align}\label{7eq: B(m)}
		\SB (m) = \prod_{p | m} \bigg(1 + \frac {a(p)^2} {p+1} \bigg) ,
	\end{align}
\begin{align}\label{7eq: B'(m)}
	\SB' (m) / \SB (m) = \sum_{p | m} B (p) \cdot \log p , 
\end{align}
for $ B (p) = O (1/p) $. 
\end{lem}

\begin{proof}
	 The product in \eqref{7eq: B(m)} follows from the multiplicativity of  $ \SB (m)  $ 
	 and  the identity
	 \begin{align*}
	 1 +	\frac {E (p)} {p} + \frac {E^2 (p)} {p^2} = 1 + \frac {a(p)^2 } {p+1}. 
	 \end{align*}
 As for \eqref{7eq: B'(m)}, the coefficient $B (p)$ has the exact formula: 
 \begin{align*}
   B (p)  =	\frac 1 {\SB (p) } \bigg( \frac {E(p) \breve{E} (p) } {p} +\frac {E^2(p) \breve{E}^2 (p) } {p^2} - \frac {E(p)} {p} - 2 \frac  {E^2 (p)} {p^2} \bigg), 
 \end{align*}
and   $B (p) = O (1/ p)$ is obvious from the Deligne bound \eqref{3eq: Deligne}. 
\end{proof}

\subsection{Treatment of the Off-diagonal Sum}  \label{sec: off}

Now we turn to the off-diagonal sum $	\SO_2 (\boldsymbol{m})$ given by \eqref{7eq: O(m)} and \eqref{7eq: H2(x)}.

\vspace{5pt}

\subsubsection{Initial Reductions} 

Similar to the initial arguments for $ \SO_1^{\pm} (m)$  in \S \ref{sec: reduction}, by applying Lemma   \ref{lem: AFE} and Corollary \ref{cor: u<T} to \eqref{7eq: O(m)} and \eqref{7eq: H2(x)}, we have 
\begin{align}\label{7eq: O2(m)}
	\SO_2 (\boldsymbol{m} ) \Lt T^{\vepsilon}  \max_{\| \boldsymbol{N} \| \Lt T^{2+\vepsilon} }  \frac {1} {\sqrt{\| \boldsymbol{N} \|}}   \sum_{\boldsymbol{d} \boldsymbol{f} = \boldsymbol{m} } \, \sideset{}{^{_{  {\scriptstyle \prime}}}} \sum_{\boldsymbol{h} \Lt  \sqrt{\boldsymbol{N}} }    \int_{\vepsilon-i\log T}^{\vepsilon+i\log T}   \big|\SO_2 (\boldsymbol{d}, \boldsymbol{f},    \boldsymbol{h}; \boldsymbol{N};  v) \big|   \frac {\nd v} {v} , 
\end{align}
where  $\boldsymbol{N} $ are dyadic,  the prime on the $ \boldsymbol{h} $-sum means $(\| \boldsymbol{h} \|, q) = 1$ (also keep in mind that $(\| \boldsymbol{m} \|, q) = 1$),  
and  $ \SO_2 (\boldsymbol{d}, \boldsymbol{f},    \boldsymbol{h}; \boldsymbol{N} )$ is defined to be the sum
\begin{align}
\begin{aligned}
	    \sum_{  c \shskip \Lt C  }     \sum_{ \boldsymbol{n} \sim \boldsymbol{N} / \boldsymbol{d h}^2 }  
a(\boldsymbol{d n})  	 \frac  {S( \boldsymbol{f n}   ; c)}  {c}  \varww \bigg(    \frac {\boldsymbol{d h}^2 \boldsymbol{n}} {\boldsymbol{N}}   \bigg)    {H}_2 \bigg(    \frac{4\pi\sqrt{\| \boldsymbol{f n} \|  }}{c}; \langle \boldsymbol{d h} \rangle  \sqrt{\langle \boldsymbol{f n} \rangle }    \bigg)   , 
\end{aligned}
\end{align} 
in which 
\begin{align}
	C = \frac  { \max  \{    \boldsymbol{  m N }   \} } {T \|\boldsymbol{d h}\| } , 
\end{align}  
the weight $\varww (\boldsymbol{x})$ is of the product form as in \eqref{3eq: w(x)} for the $\log T$-inert  $\varww   \in C_c^{\infty} [1, 2]$ as in \S \ref{sec: reduction}, and  $H_{2} (x, y ) $ is the Bessel transform (see \S \ref{sec: Bessel}): 
\begin{align*}
	H_{2} (x, y ) = \frac{2 i } {\pi} \int_{-\infty}^{\infty} h_{2} (t; y ) J_{2it} (x)  \frac {t \nd t} {\cosh (\pi t) } , 
\end{align*} 
for 
\begin{align*}
	h_{2} (t; y ) = 2 \upgamma_{2} (v, t)  \cos (2t \log y) (\upphi (t) +  \upphi (- t)); 
\end{align*}   
here $v$ has been suppressed again from the notation.

\vspace{5pt} 

\subsubsection{Application of Lemma {\ref{lem: Bessel}}}  \label{sec: after Bessel}
By the formula of the Bessel integral $H  (x, y)$ as in  Lemma \ref{lem: Bessel}, along with \eqref{4eq: H2=H},  it follows that $\SO_2 (\boldsymbol{d}, \boldsymbol{f},    \boldsymbol{h}; \boldsymbol{N} ) $ splits into four similar parts, one of which is of the form  
\begin{align}\label{8eq: O2++}
 \SO^{_{++}}_{2} (\boldsymbol{d}, \boldsymbol{f},    \boldsymbol{h}; \boldsymbol{N} )  = 	
\varPi T^{1+\vepsilon}  \int_{-\varPi^{\vepsilon}/ \varPi}^{\varPi^{\vepsilon}/ \varPi}  g (\varPi r) e ( Tr/\pi) \cdot  
 {\SO}^{_{++}}_{2} (r; \boldsymbol{d}, \boldsymbol{f},    \boldsymbol{h}; \boldsymbol{N} )      \nd r  , 
\end{align}
with  
\begin{align}\label{8eq: O++}
\SO^{_{++}}_{2} (r; \boldsymbol{d}, \boldsymbol{f},    \boldsymbol{h}; \boldsymbol{N} ) =  \sum_{  c \shskip \Lt C  }  \frac 1 c  \sum_{ \boldsymbol{n} } 
	a(\boldsymbol{d n})  	   {T ( \boldsymbol{f n} ; \boldsymbol{d h}   ; c)} e ( \uprho (r; \boldsymbol{f n} , \boldsymbol{d h} ) )  \varww \bigg(    \frac {\boldsymbol{d h}^2 \boldsymbol{n}} {\boldsymbol{N}}   \bigg) , 
\end{align}
where
\begin{align}\label{8eq: defn T}
	T ( \boldsymbol{ n} ; \boldsymbol{ k    }   ; c) = S(n_1, n_2; c ) e \bigg(\frac {k    _1 n_1} {c k    _2  } + \frac {k    _2 n_2} {c k    _1  } \bigg), 
\end{align}
\begin{align}\label{8eq: rho}
	\uprho (r; \boldsymbol{ n} , \boldsymbol{ k    } ; c) = \frac {k    _1 n_1} {c k    _2} \uprho_1 (r) + \frac {k    _2 n_2} {c k    _1} \uprho_2 (r), \qquad \uprho_{1,2} (r) = \exp (\pm r) - 1. 
\end{align}

\vspace{5pt}

\subsubsection{Estimation by the Wilton Bound}  
First, in order to split the $n_1$- and $n_2$-sums in \eqref{8eq: O++}, we open the Kloosterman sum $S(n_1, n_2; c )$ in \eqref{8eq: defn T} (see \eqref{3eq: Kloosterman}).  The $n_1$- and $n_2$-sums are very similar, and  the former  reads 
\begin{align*}
	 \sum_{n_1 \sim N_1 / d_1 h_1^2}   a (d_1 n_1) e \bigg( \frac {\valpha f_1 n_1} {c}   +  \frac {m_1 h_1 n_1} {c d_2 h_2  } \exp (r)  \bigg) \varww \bigg(\frac {d_1 h_1^2 n_1} {N_1} \bigg) .  
\end{align*}   
As $\varww (x)$ is $\log T$-inert, by partial summation,  it follows from  the Wilton bound in Lemma \ref{lem: Wilton}, along with the Hecke relation \eqref{3eq: Hecke, Q} and the Deligne bound \eqref{3eq: Deligne},  that the  $n_1$-sum is bounded by 
\begin{align*}
\log^2 T 	\sum_{\, d | d_1 } |a (d_1 / d) | \sqrt{\frac {N_1} {d d_1 h_1^2} } \Lt T^{\vepsilon} \frac {\sqrt{N_1}} {\sqrt{d_1} h_1} .
\end{align*} 
Note that the bound holds uniformly for $r$, $\valpha$, and $c$.  Similarly, the $n_2$-sum is uniformly $ O \big( T^{\vepsilon}   {\sqrt{N_2}} / {\sqrt{d_2} h_2} \big) $.
Thus, in view of \eqref{8eq: O2++} and \eqref{8eq: O++}, we have 
\begin{align*}
	 \SO_2 (\boldsymbol{d}, \boldsymbol{f},    \boldsymbol{h}; \boldsymbol{N} ) \Lt \varPi T^{1+\vepsilon}  \int_{-\varPi^{\vepsilon}/ \varPi}^{\varPi^{\vepsilon}/ \varPi} \sum_{  c \shskip \Lt C  }  \frac {1 } c  \, \sumx_{\valpha (\mathrm{mod}\, c)}      \frac {\sqrt{\| \boldsymbol{N} \|}} {\sqrt{\|\boldsymbol{d}\|} \| \boldsymbol{h} \|  }  \nd r \Lt  T^{\vepsilon} \frac {T C \sqrt{\| \boldsymbol{N} \|}} {\sqrt{\|\boldsymbol{d}\|} \| \boldsymbol{h} \|  } . 
\end{align*} 

\begin{lem}
	We have
	\begin{align}\label{8eq: bound after Wilton}
		  \SO_2 (\boldsymbol{d}, \boldsymbol{f},    \boldsymbol{h}; \boldsymbol{N} ) \Lt T^{\vepsilon} \frac {  {\| \boldsymbol{N} \|}^{1/2}   \max  \{    \boldsymbol{  m N } \}  } { {\|\boldsymbol{d}\|^{3/2}}  \| \boldsymbol{h} \|^2   } . 
	\end{align}
\end{lem}

\begin{coro}\label{cor: O, 1} 
	The contribution to $\SO_2 (\boldsymbol{m} )$ {\rm(}see {\rm\eqref{7eq: O2(m)}}{\rm)} from those $\SO_2 (\boldsymbol{d}, \boldsymbol{f},    \boldsymbol{h}; \boldsymbol{N} )$ such that $$ \max \big\{\langle  \boldsymbol{  m N } \rangle, \langle  \widetilde{\boldsymbol{  m N }} \rangle \big\} \leqslant {T  / \varPi}$$ has  bound  $O \big( T^{ 3/2+ \vepsilon}  {  \| \boldsymbol{m} \|^{1/2}/ \varPi^{1/2}} \big)$. 
\end{coro}

Henceforth, we shall restrict ourselves to the case 
\begin{align}\label{8eq: condition}
	\langle  \boldsymbol{  m N } \rangle > T/ \varPi ; 
\end{align}
of course, the case $\langle  \widetilde{\boldsymbol{  m N }} \rangle > T/ \varPi$ is similar by symmetry. Therefore, 
 by Lemma \ref{lem: analysis of I, 2} the condition $c \Lt C$ as in \eqref{8eq: O++} may be refined into
$
	c \text{ \small $\asymp$ } C,  
	$ 
and it will be convenient to  truncate the $c$-sum smoothly by  $\varvv (c/ C)$ for a certain fixed $\varvv \in C_c^{\infty} [ \pi/\sqrt{2}, 2\sqrt{2}\pi]$.

Subsequently, we shall avoid the use of bi-variable notation, due to different treatments of the $n_1$- and $n_2$-sums. 
For brevity, let us also remove those bold letters $\boldsymbol{d}$, $\boldsymbol{f}$,    $\boldsymbol{h}$,  $\boldsymbol{N}$ from the notation. 

\vspace{5pt} 

\subsubsection{Transformation of Exponential Sum}

Assume as we may that  $(k_1,  k_2) = 1$ (namely, $k_1$ and $k_2$ are co-prime), since  $T ( n_1, n_2 ; k_1, k_2   ; c)$ as in \eqref{8eq: defn T} depends only on the fraction $k_1/k_2$. 
Let 
\begin{align}\label{8eq: c1, c2}
	c_1 = (c, k    _1^{\infty}), \quad c_2 = (c, k    _2^{\infty}), \qquad c_{\star} = \frac {c} {c_1 c_2} ,
\end{align}
\begin{align}
	k    _{1}^{_{\natural}}  = (k    _1, c^{\infty}), \quad k    _{1}^{\star} = \frac {k    _1} {k    _{1}^{_{\natural}} } , \qquad k    _{2}^{_{\natural}}  = (k    _2, c^{\infty}),  \quad k    _{2}^{\star} = \frac {k    _2} {k    _{2}^{_{\natural}} } .
\end{align}

	In the case that $c \, | k_{\phantom{1}}^{_{\natural}}{}^{\infty}$ and $ k_{\phantom{1}}^{_{\natural}} | c^{\infty}   $ {\rm(}in other words, $c$ and $k_{\phantom{1}}^{_{\natural}}$ have the same set of prime divisors{\rm)},  define 
	\begin{align}\label{7eq: To}
		T_{\mathrm{o}} (n_1, n_2; k_{\phantom{1}}^{_{\natural}};  c  ) = S (n_1, n_2; c) e \bigg( \frac { k_{\phantom{1}}^{_{\natural}}  }  {c  } n_1 + \frac { 1 }  {c k_{\phantom{1}}^{_{\natural}}} n_2 \bigg). 
	\end{align}  

By the Chinese remainder theorem, we split 
\begin{align}\label{8eq: split T}
	\begin{aligned}
		 T ( n_1 & ,  n_2 ;   k_1,  k_2   ; c) =     T ( \overline{c_1 c_2} k_1^{_{\natural}} \overline{k}{}_2^{_{\natural}} n_1, \overline{c_1 c_2} \overline{k}{}_1^{_{\natural}} k_2^{_{\natural}}  n_2; k_1^{\star},  k_2^{\star}; c_{\star} ) \\
		 &    \cdot    T_{\mathrm{o}} (\overline{c_{\star} c_2  } k_1^{\star} \overline{k}_2 n_1, \overline{c_{\star} c_2  } \overline{k}{}_1^{\star} k_2 n_2; k_1^{_{\natural}}; c_1 )    \cdot    T_{\mathrm{o}} (\overline{c_{\star} c_1  } k_2^{\star} \overline{k}_1 n_2, \overline{c_{\star} c_1  } \overline{k}{}_2^{\star} k_1 n_1 ; k_2^{_{\natural}}; c_2 ). 
	\end{aligned}
\end{align}
We have the next two lemmas for the $T$ and $T_{\mathrm{o}}$ exponential sums on the right of \eqref{8eq: split T}. Note that \eqref{8eq: T (n;k;c)}  in the case $k_1=k_2=1$ is an identity of Luo in \cite[\S 3]{Luo-Twisted-LS}. 

\begin{lem}\label{lem: Luo's identity}
	For $ (k_1, k_2) = ( c, k_1 k_2 )  = 1 $, we have
	\begin{align}\label{8eq: T (n;k;c)}
		T ( n_1   ,  n_2 ;   k_1,  k_2   ; c) = e \bigg(\frac {\widebar{c} k_1 n_1}{k_2}  + \frac {\widebar{c} k_2 n_2}{k_1}  \bigg) \hspace{-2pt} \sum_{b w = c} \hspace{-2pt} \mathop{\sum_{\beta (\mathrm{mod}\, b)}}_{(\beta (w \overline{k}_1 k_2 -   \beta), b) = 1} \hspace{-2pt} e \bigg( \frac {\widebar{\beta} n_1 + \overline{w \overline{k}_1 k_2 - \beta } n_2 } {b} \bigg) . 
	\end{align}
\end{lem}  

\begin{proof}
	By opening the Kloosterman sum, we have 
	\begin{align*}
	T ( n_1   ,  n_2 ;   k_1,  k_2   ; c) = 	\sumx_{\valpha (\mathrm{mod}\, c) }  	e \bigg( \bigg(\frac { \valpha  } {c}  + \frac {k_1 } {c k_2}\bigg) n_1   + \bigg(\frac { \widebar{\valpha}  } {c}  + \frac {k_2 } {c k_1}\bigg) n_2   \bigg) .
	\end{align*}
Further, by the Chinese remainder theorem, 
\begin{align*}
	T ( n_1   ,  n_2 ;   k_1,  k_2   ; c) = e \bigg(\frac {\widebar{c} k_1}{k_2} n_1 + \frac {\widebar{c} k_2}{k_1} n_2 \bigg)	\sumx_{\valpha (\mathrm{mod}\, c) }  	e \bigg(  \frac { \valpha  + k_1 \overline{k}_2 } {c}  n_1 + \frac { \widebar{\valpha} + \overline{k}_1 k_2 } {c}   n_2   \bigg) .
\end{align*}
Set  $w = (\valpha  + k_1 \overline{k}_2 , c ) $.  Note that $ w = (\widebar{\valpha} + \overline{k}_1 k_2 , c)  $ as well. 
Write $c = b w $ and $ \valpha  + k_1 \overline{k}_2 = \allowbreak  \widebar{\beta} w $, where $\beta$ ranges over residue classes modulo $b$ such that $(\beta (w - k_1 \overline{k}_2 \beta), b) = 1$. Thus we obtain the identity in \eqref{8eq: T (n;k;c)} because 
\begin{align*}
	\widebar{\valpha} + \overline{k}_1 k_2 \equiv 
	 \overline{ \widebar{\beta} w - {k}_1 \overline{k}_2 } + \overline{k}_1 k_2 \equiv w \cdot \overline{w \overline{k}_1 k_2 - \beta } \, (\mathrm{mod}\, c). 
\end{align*}
\end{proof}

\begin{lem}\label{lem: To}
Let $T_{\mathrm{o}} (n_1, n_2; k_{\phantom{1}}^{_{\natural}};  c  ) $ be defined as in {\rm\eqref{7eq: To}} for  $c \, | k_{\phantom{1}}^{_{\natural}}{}^{\infty}$ and $ k_{\phantom{1}}^{_{\natural}} | c^{\infty}   $. We have
	\begin{align} \label{8eq: To (n;k;c)}
		T_{\mathrm{o}} (n_1, n_2; k_{\phantom{1}}^{_{\natural}};  c  ) = \sumx_{\beta (\mathrm{mod}\, c) }  	e \bigg( \frac {\widebar{\beta} n_1} {c}  + \frac {\overline{1 -    k_{\phantom{1}}^{_{\natural}} \beta } n_2 } {c k_{\phantom{1}}^{_{\natural}} }  \bigg). 
	\end{align}
\end{lem}

\begin{proof}
	By opening the Kloosterman sum, we have 
	\begin{align*}
		T_{\mathrm{o}} (n_1, n_2; k_{\phantom{1}}^{_{\natural}};  c  ) =	\sumx_{\valpha (\mathrm{mod}\, c) }  	e \bigg(  \frac { \valpha + k_{\phantom{1}}^{_{\natural}} } {c}   n_1   + \frac {  k_{\phantom{1}}^{_{\natural}} \widebar{\valpha} + 1 } {c k_{\phantom{1}}^{_{\natural}} } n_2   \bigg) .
	\end{align*}
so the identity in \eqref{8eq: To (n;k;c)} follows easily by the change of variable $\widebar{\beta} = \valpha + k_{\phantom{1}}^{_{\natural}}$. 
\end{proof}


\subsubsection{Preparations for Poisson and Vorono\"i}   

Let 
\begin{align}\label{7eq: k1, k2}
  k_1 = \frac {d_1 h_1} {(d_1 h_1, d_2 h_2)}, \quad k_2 = \frac {d_2 h_2} {(d_1 h_1, d_2 h_2)} .  
\end{align}
Subsequently, we shall only treat  the contribution from those generic terms with $   (c,    k_1 k_2 )    = 1 $ (see Lemma \ref{lem: Luo's identity})   
so as to ease the burden of notation in \eqref{8eq: c1, c2}--\eqref{8eq: split T}. Note that the $T_{\mathrm{o}}$ exponential sum is not hard to handle in view of Lemma \ref{lem: To}. Actually, the reader might as well proceed with the most typical case $\boldsymbol{d}=\boldsymbol{h} = \boldsymbol{1}$, $\boldsymbol{f} = \boldsymbol{m}$. 

Now we  split the $c$-sum in \eqref{8eq: O++} into $b$- and $w$-sums by  Luo's  identity \eqref{8eq: T (n;k;c)} in Lemma \ref{lem: Luo's identity} for the exponential sum in \eqref{8eq: defn T}. 
Our next step is to apply Poisson and Vorono\"i  summation to the $w$- and $n_1$-sums. 


More explicitly, let us  write simply $ \SO^{_{++}}_{2} = \SO_2^{_{++}} (\boldsymbol{d}, \boldsymbol{f},    \boldsymbol{h}; \boldsymbol{N} ) $, then it follows from  \eqref{8eq: O2++}--\eqref{8eq: rho}, and \eqref{8eq: T (n;k;c)} that its generic part is equal to
\begin{align}\label{7eq: O2++}
\SO^{\natural}_{2} =	  T^{ \vepsilon}   \sum_{n_2 \sim N_2'  }   a(d_2n_2) \varww \bigg(  \frac {n_2} {N_2'} \bigg)   \sideset{}{^{{\natural}}}   \sum_{b \shskip \Lt C} \frac {\varPi T} {b}    \int_{-\varPi^{\vepsilon}/ \varPi}^{\varPi^{\vepsilon}/ \varPi}   g (\varPi r) e ( Tr/\pi)    
	{\SO}^{\natural}_{2} (r; b ; n_2 ) \nd r , 
\end{align} where, by the Hecke relation  \eqref{3eq: Hecke, Q}, the partial sum ${\SO}^{\natural}_{2} (r; b)$ (here $n_2$ is not essential (so is $d$ below in \eqref{8eq: O2}) and hence omitted from the notation) that contains $n_1$($= d n$) and $w$ is of the form: 
\begin{align}\label{8eq: O2}
 \SO^{\natural}_{2} (r ; b ) = \sum_{\ d | d_1} \mu (d)  a (d_1/d )  \sum_{ n \sim N } a ( n )  \sideset{}{^{{\natural}}}  \sum_{w \asymp C/b}  \frac { E ( n ;   w) V ( n; w; b) } {w}    \varww  \bigg( r; \frac {n} {N} ;  \frac {b w} {C} \bigg), 
\end{align}  
the {\footnotesize $\natural$} on the $b$- and $w$-sums indicate $(b w, k_1 k_2) = 1$, and 
\begin{align}\label{8eq: parameters}
N_2' = \frac {N_2} {d_2 h_2^2}, \qquad N = \frac {N_1} {d d_1 h_1^2 } , \qquad C = \frac {m_1 N_1} {T d_1 d_2 h_1 h_2},  
\end{align}
\begin{align}\label{7eq: E(n;w)}
	E (n ; w ) =  e \bigg(    \frac {\widebar{b} k_1 d{} f_1 n \overline{w}  }{k_2}  + \frac {\widebar{b} k_2 f_2 n_2 \overline{w} }{k_1}    \bigg),  
\end{align}
\begin{align}\label{7eq: V(n;w;b)}
	V (n; w; b) = \mathop{\sum_{\beta (\mathrm{mod}\, b)}}_{(\beta (w \overline{k}_1 k_2 -   \beta), b) = 1}   e \bigg( \frac {\widebar{\beta} d {} f_1 n + \overline{w \overline{k}_1 k_2 - \beta } f_2 n_2 } {b} \bigg), 
\end{align}
\begin{align}\label{8eq: w (r;x;y)}
	\varww (r; x; y)	=  \varww_{\mathrm{o}} (r; x; y) e 	 (  T x \uprho_1 (r)/ y  ) , \qquad \uprho_1 (r) = \exp (r) -1, 
\end{align}
\begin{align}\label{8eq: w (x;y)}
\varww_{\mathrm{o}} ( r; x; y)	=  \varww (x) \varvv (y) e 	 (  S  \uprho_2 (r)/ y   ), \quad S 
= T \frac {m_2 d_2 h_2^2 n_2} {m_1 N_1 }, \ \ \,  \uprho_2 (r) = \exp (-r) -1 .
\end{align}
Keep in mind that $\varww \in C_c^{\infty} [1,2]$ is $\log T$-inert and $ \varvv \in C_c^{\infty} [ \pi/\sqrt{2}, 2\sqrt{2}\pi]$ is fixed.  
Recall from \eqref{8eq: condition} that
\begin{align}\label{8eq: condition, mN}
	\frac {  m_1 N_1} {  m_2 N_2} > \frac {T} {\varPi}, 
\end{align}
so  for any $|r| \leqslant \varPi^{\vepsilon}/ \varPi$ and $ n_2 \sim N_2' = N_2 / d_2 h_2^2 $, it is easy to verify by   \eqref{8eq: condition, mN} that the exponential factor in \eqref{8eq: w (x;y)} is $T^{\vepsilon}$-inert in both $y$ and $r$ in the sense of Definition \ref{defn: inert}. Similarly,  in the simpler case $|r| \leqslant T^{\vepsilon}/ T$, the  exponential factor in \eqref{8eq: w (r;x;y)} is also $T^{\vepsilon}$-inert and hence will be absorbed into the $T^{\vepsilon}$-inert weight $ \varww_{\mathrm{o}} (r; x; y)$ as well.

\vspace{5pt}

\subsubsection{Application of Poisson Summation} 

By applying to the $w$-sum in \eqref{8eq: O2} the Poisson summation formula of modulus $b k_1 k_2$ as in Lemma \ref{lem: Poisson},  we infer that 
\begin{align}\label{7eq: O2, 2}
 \SO^{\natural}_{2} (r ; b) =     \sum_{\ d | d_1} \frac  {\mu (d)  a (d_1/d )} {b k_1 k_2}	\sum_{w = -\infty}^{\infty} \sum_{ n = 1}^{\infty}  a ( n ) K (n; w ; b)  \hat{\varww} \bigg( r;  \frac {n} {N};  \frac {C w   } {b^2 k_1k_2} \bigg), 
\end{align}
where
\begin{align}\label{7eq: K(n;w;b)}
	K (n; w ; b) = \mathop{\sum_{\valpha (\mathrm{mod}  \, b k_1 k_2)  } }_{(\valpha, k_1 k_2) = 1}   E (n ; \valpha ) V (n; \valpha ; b)  e \bigg(     \frac {   \valpha {w}  } {b k_1 k_2}  \bigg), 
\end{align}
\begin{align}
\hat {\varww} (r; x; v) =	\int \varww   (r;  x  ;  y  ) 
e   (-v y)  \frac { \nd y} {y} . 
\end{align}

For the case $ k_1 = k_2 = 1$ the next lemma is essentially \cite[Lemma 3.1]{Iwaniec-Li-Ortho}. 
\begin{lem} We have 
	\begin{align}\label{7eq: Kloosterman}
	K (n; w ; b) =	S (d {}  f_1 n,    w;   b k_2) S (f_2 n_2,   w; b  k_1)   . 
	\end{align}
\end{lem}

\begin{proof} By \eqref{7eq: E(n;w)}, \eqref{7eq: V(n;w;b)}, and \eqref{7eq: K(n;w;b)}, it follows from the Chinese remainder theorem that
	\begin{align*}
	K (n; w ; b) =	\sumx_{\valpha_1 (\mathrm{mod}  \,   k_2  ) }   e \bigg(    \frac {   \widebar{\valpha}_1 \widebar{b}^2   d{} f_1 n  + \valpha_1 {w}  }{k_2}  \bigg) 
		\sumx_{\valpha_2 (\mathrm{mod}  \,   k_1  ) }     e \bigg(      \frac {  \widebar{\valpha}_2 \widebar{b}^2   f_2 n_2  + \valpha_2 {w}  }{k_1}\bigg)     & \\
		\mathop{\mathop{\sum \sum}_{\valpha_{\mathrm{o}}, \, \beta (\mathrm{mod}\, b)}}_{(\beta (\valpha_{\mathrm{o}}    k_2^2 -    \beta), b) = 1}   e \bigg(    \frac {  \widebar{\beta} d {} f_1 n + \overline{\valpha_{\mathrm{o}}  k_2^2 - \beta } f_2 n_2 +  \valpha_{\mathrm{o}} {w} } {b} \bigg) & . 
	\end{align*} 
By the change $\gamma = \valpha_{\mathrm{o}}  k_2^2 - \beta$, the second line reads
\begin{align*}
	\mathop{\sumx \sumx}_{\beta, \gamma (\mathrm{mod}\, b)} e \bigg( \frac { \widebar{\beta} d {} f_1 n + \widebar{\gamma}  f_2 n_2 + (\beta + \gamma) \widebar{k}_2^2 w } { b } \bigg). 
\end{align*}
Consequently,  
\begin{align*}
K (n; w ; b) =	S (\widebar{b}^2  d {}  f_1 n,    w;    k_2) S (\widebar{b}^2  f_2 n_2,    w;   k_1) S(d {}  f_1 n , \widebar{k}_2^2 w;    b)  S(f_2 n_2, \widebar{k}_2^2 w;   b). 
\end{align*}
Finally  \eqref{7eq: Kloosterman} follows from the Chinese remainder theorem (or the multiplicativity of Kloosterman sums). 
\end{proof}

\subsubsection{Application of Vorono\"i Summation} 

Set 
\begin{align}\label{7eq: dx, bx}
b_{\mathrm{o}} = (b k_2, d{} f_1), \qquad b_{\star} = \frac {b k_2} { b_{\mathrm{o}} }, \quad 	f_{\star} = \frac {d {} f_1} {b_{\mathrm{o}}} . 
\end{align}Since $(m_1m_2 h_1 h_2, q) = 1$,  we have $(b_{\star}, q) = (b, q)$ in view of \eqref{7eq: k1, k2}. Write
\begin{align}\label{7eq: qb}
	q_b = \frac {q} {(b, q)}. 
\end{align} By opening the Kloosterman sum $ S (d {}  f_1 n,    w;   b k_2) $ in \eqref{7eq: Kloosterman} and applying to the $n$-sum in \eqref{7eq: O2, 2} the Vorono\"i summation formula of modulus $b_{\star} $ as in Lemma \ref{lem: Voronoi},  we infer that
\begin{align}
	\SO^{\natural}_{2} (r ; b) =  N   \sum_{\, d | d_1} \frac  {\mu (d q_b)  a (d_1 q_b /d   )} {b b_{\star} k_1 k_2}  	\sum_{w = -\infty}^{\infty} \sum_{ n = 1}^{\infty}  a ( n ) R (n; w ; b)  \check{\hat{\varww} }\bigg( r;  \frac {N n} {b_{\star}^2 q_b};  \frac {C w   } {b^2 k_1k_2} \bigg), 
\end{align}
with 
\begin{align}\label{7eq: R(n;w;b)}
	R (n; w; b) = S   (w - \widebar{f}_{\star} \widebar{q}_b b_{\mathrm{o}}  n , 0; b k_2  ) S (f_2 n_2,   w; b  k_1),  
\end{align}
\begin{align}\label{7eq: Fourier-Hankel}
	\check{\hat {\varww}} (r; u; v) = 2\pi i^{2k}	\iint \varww   (r;  x  ;  y  ) 
	e   (-v y) J_{2k-1} (4\pi \sqrt{u x} )  \frac {\nd x \nd y } {y} . 
\end{align}
By an application of the formula \eqref{3eq: Ramanujan} to the Ramanujan sum $ S (w - \widebar{f}_{\star} \widebar{q}_b b_{\mathrm{o}}  n , 0; b k_2  )  $ as in \eqref{7eq: R(n;w;b)}, we rewrite 
\begin{align}\label{7eq: O2, 3}
\begin{aligned}
		\SO^{\natural}_{2} (r ; b) = N &  \sum_{\ d | d_1} \frac  {   \mu (d q_b)    a (d_1 q_b /d   )} {b b_{\star} k_1 k_2}   \sum_{\ g | bk_2  }   g \mu (bk_2/ g)	\\
		& \cdot  \sum_{ n  }   a ( n ) \sum_{w \equiv \widebar{f}_{\star} \widebar{q}_b b_{\mathrm{o}}  n (\mathrm{mod}\, g ) }  S (f_2 n_2,   w; b  k_1)  \check{\hat{\varww} }\bigg( r;  \frac {N n} {b_{\star}^2 q_b};  \frac {C w   } {b^2 k_1k_2} \bigg) .  
\end{aligned}
\end{align}

\vspace{5pt}

\subsubsection{Split of Integration}  \label{sec: split integral}


Let us truncate the $r$-integral at $|r| = T^{\vepsilon}/ T$ and apply a smooth dyadic partition on the larger part of the integral domain.  It follows from the identities \eqref{7eq: O2++} and \eqref{7eq: O2, 3}, the Weil bound \eqref{3eq: Weil}, and the Deligne bound \eqref{3eq: Deligne}  that 
\begin{align}\label{7eq: O2}
\begin{aligned}
	\SO^{\natural}_{2}	\Lt   T^{ \vepsilon} \sum_{n_2 \sim N_2'  }  \sideset{}{^{{\natural}}} \sum_{b \shskip \Lt C} \sum_{\ d | d_1}  \sum_{  g | bk_2  } 
	 \frac {\varPi T N  g \sqrt{bk_1}} {b^2  b_{\star} k_1 k_2}  \Bigg\{  & \int_{- T^{\vepsilon}/ T}^{T^{\vepsilon}/ T}     \mathrm{S} (b, g; d; n_2; r)    \nd r \\
	 &  \ \    + \max_{ T^{\vepsilon}/ T \Lt \tau \shskip \Lt \varPi^{\vepsilon}/ \varPi } 	 \mathrm{S}_{  \tau } (b, g; d; n_2 )  \Bigg \}     , 
\end{aligned}
\end{align}
for $\tau$ dyadic,  where, if  $d$ and $n_2$ are omitted, 
\begin{align}\label{7eq: S(b,g;r)}
	\mathrm{S} (b, g; r) = \mathop{\mathop{\sum \sum}_{n, \, w} }_{w \equiv \widebar{f}_{\star} \widebar{q}_b b_{\mathrm{o}} n (\mathrm{mod}\, g ) }  \sqrt{(w, bk_1)}   \bigg|\check{\hat{\varww} }\bigg( r;  \frac {N n} {b_{\star}^2 q_b};  \frac {C w   } {b^2 k_1k_2} \bigg)\bigg|, 
\end{align}
\begin{align}
\mathrm{S}_{\tau} (b, g) =  \sum_{\pm}   \mathop{\mathop{\sum \sum}_{n, \, w} }_{w \equiv \widebar{f}_{\star} \widebar{q}_b b_{\mathrm{o}} n (\mathrm{mod}\, g ) }  \sqrt{(w, bk_1)}  \bigg| \check{\hat{\varww} }_{\tau}^{\scriptscriptstyle \pm} \bigg(  \frac {N n} {b_{\star}^2 q_b};  \frac {C w   } {b^2 k_1k_2} \bigg)\bigg|, 
\end{align}
with 
\begin{align}\label{7eq: w(u;v)}
	 \check{\hat{\varww} }_{\tau}^{\scriptscriptstyle \pm} (u; v) = 2\pi i^{2k}	\iiint e (Tr/\pi) \varvv_{\scriptscriptstyle \pm} (r / \tau) \varww   (r;  x  ;  y  )  
	 J_{2k-1} (4\pi \sqrt{u x} ) e   (-v y)   \frac {\nd r \nd x \nd y } {y}, 
\end{align}
for a certain inert weight $\varvv_{\scriptscriptstyle \pm} \in C_c^{\infty} (\pm [  1,   2])$.

Note that neither  $d$ nor $n_2$ will play a role in our later analysis---indeed  $d$ was hidden in $N$, $b_{\mathrm{o}}$, $b_{\star}$, and $f_{\star}$ (see \eqref{8eq: parameters} and \eqref{7eq: dx, bx}), while $n_2$ was in the $T^{\vepsilon}$-inert weight $\varww_{\mathrm{o}} ( r; x; y) $ in \eqref{8eq: w (x;y)} and the Kloosterman sum $S (f_2 n_2,   w; b  k_1)$ in \eqref{7eq: O2, 3} which has just been estimated by the Weil bound. 

\vspace{5pt}

\subsubsection{Bound for $ \mathrm{S} (b, g; r)  $}

We have a simple lemma for the Fourier--Hankel transform $\check{\hat {\varww}} (r; u; v)$ defined as in  \eqref{7eq: Fourier-Hankel}. 

\begin{lem}\label{lem: w (r;u;v)}
	For any $|r| \leqslant T^{\vepsilon} / T$, we have $ \check{\hat {\varww}} (r; u; v) = O (T^{-A})$ for any $A \geqslant 0$ unless $u,  |v| \leqslant T^{\vepsilon} $. 
\end{lem}

\begin{proof}
As noted below \eqref{8eq: condition, mN}, the weight $\varww   (r;  x  ;  y  )$ is $T^{\vepsilon}$-inert in $x$ and $y$ for any $|r| \leqslant T^{\vepsilon}/ T$. 	Thus, in view of \eqref{2eq: Bessel, Whittaker} and \eqref{2eq: bounds for Whittaker functions},  Lemma \ref{lem: integral, 0} implies that $\check{\hat {\varww}} (r; u; v) $ is negligibly small if either $ u > T^{\vepsilon}$ or $ |v| > T^{\vepsilon} $. 
\end{proof}

By Lemma \ref{lem: w (r;u;v)},  the $n$- and $w$-sums in \eqref{7eq: S(b,g;r)} may be truncated effectively, and, as trivially $\check{\hat {\varww}} (r; u; v) = O (1)$,  it follows  that 
\begin{align}\label{7eq: bound for S(b,g;r)}
	 \mathrm{S} (b, g; r) \Lt  \mathop{ \sum_{n \Lt T^{\vepsilon} b_{\star}^2  / N} \sum_{|w| \Lt T^{\vepsilon} b^2 k_1 k_2/ C }   }_{w \equiv \widebar{f}_{\star} \widebar{q}_b b_{\mathrm{o}} n (\mathrm{mod}\, g ) }  \sqrt{(w, bk_1)}  \Lt T^{\vepsilon} \frac {b_{\star}^2  } {N} \bigg(1 + \frac {b^2 k_1 k_2} {C g} \bigg). 
\end{align}

\vspace{5pt}

\subsubsection{Bound for $ \mathrm{S}_{\tau} (b, g)  $} 
By \eqref{8eq: w (r;x;y)}, \eqref{8eq: w (x;y)}, and \eqref{7eq: w(u;v)}, 
we write  
\begin{align}\label{7eq: w(u;v), 2}
	\check{\hat{\varww} }_{\tau}^{\scriptscriptstyle \pm} (u; v) = \int  \varvv (y) \check{\varww} _{\tau}^{\scriptscriptstyle \pm} (u; y ) e   (-v y)  \frac {\nd y} {y} ,
\end{align}
with 
\begin{align}
	 \check{\varww} _{\tau}^{\scriptscriptstyle \pm} (u; y  )= 2\pi i^{2k}	\iint    \varww_{\tau}^{\scriptscriptstyle \pm} (r; x; y) 
	e 	 ( Tr/\pi + T x \uprho_1 (r)/ y  ) J_{2k-1} (4\pi \sqrt{u x} )      {\nd r \nd x  }  ,
\end{align} 
where $\varww_{\tau}^{\scriptscriptstyle \pm} (r; x; y) $ is  $T^{\vepsilon}$-inert for all $r$, $x$, and $y$ on $\pm [\tau, 2\tau]$, $[1, 2]$, and $[\pi/\sqrt{2}, 2\sqrt{2} \pi]$; more explicitly, 
\begin{align*} 
	\varww_{\tau}^{\scriptscriptstyle \pm} (r; x; y) = \varvv_{\scriptscriptstyle \pm } (r / \tau)  \varww (x) e 	 (  S  \uprho_2 (r)/ y   )  .
\end{align*} 
For  $e 	 (  S  \uprho_2 (r)/ y   )$ see the remarks below \eqref{8eq: condition, mN}. 

\begin{lem}	\label{lem: w (u;v)}
	Assume $T^{\vepsilon} / T \Lt \tau \Lt  T^{\vepsilon} / \sqrt{T}$. 
	Then 
	
	{\rm\,(i)}  $ \check{\varww} _{\tau}^{\scriptscriptstyle +} (u; y  )$ is always  negligibly small,  
	
	{\rm(ii)} $ \check{\varww} _{\tau}^{\scriptscriptstyle -} (u; y  )$ is negligibly small  unless   $ \sqrt{u} \asymp T \tau$, in which case  
	\begin{align}\label{7eq: w(u;y)}
		 \check{\varww} _{\tau}^{\scriptscriptstyle -} (u; y  ) =  \frac { \varvv_{\flat} (u; y) } {T \sqrt{T \tau} } e (- 2 \sqrt{uy/ \pi } ) ,
	\end{align}
for a certain $T^{\vepsilon}$-inert function $  \varvv_{\flat} (u; y) $. 
\end{lem}

\begin{proof}
	This lemma follows by  applying twice Lemma  \ref{lem: integral, 0}. For notational simplicity, we shall only display the  integration variables in those inert weights. 
	
	For $\pm r \sim  \tau$, let us first consider the $x$-integral:
	 \begin{align*}
	 \check{\varww} (r; u; y ) = 2\pi i^{2k}	\int_1^2 \varww (x)  e 	 (  T x \uprho_1 (r)/ y  ) J_{2k-1} (4\pi \sqrt{u x} ) \nd x. 
	 \end{align*}  
 Since $ \pm T   \uprho_1 (r)/ y \asymp T \tau \Gt T^{\vepsilon} $, it follows from \eqref{3eq: bounds for Bessel, x<1} and  Lemma  \ref{lem: analysis of integral} that the integral is negligibly small if $ u \Lt 1 $. 
 By  \eqref{2eq: Bessel, Whittaker} and \eqref{2eq: bounds for Whittaker functions}, in the case $u \Gt 1$, we split the integral into  
 \begin{align*}
 \check{\varww} (r; u; y ) =\sum_{\pm}   \frac 1 { \sqrt[4]{u} }	\int _1^2  \varww_{\scriptscriptstyle \pm }  (x) e 	 \bigg(   \frac{T  \uprho_1 (r)} {y} x  \pm 2 \sqrt{u x} \bigg)  \frac { \nd x } {\sqrt[4]{x}} ,
 \end{align*}
for certain weights $\varww_{\scriptscriptstyle \pm } (x) = \varww_{\scriptscriptstyle \pm } (x; u) $ which are still $ \log T$-inert. Set  \begin{align}
	\sqrt{\rho_{\natural}}  = \pm  \frac {T  \uprho_1(r) }  {\sqrt{u} y} , \qquad \lambda_{\natural} =  \pm \frac {uy} {T \uprho_1 (r)} , 
\end{align}  and make the change   $ {x} \ra  x  / \rho_{\natural}    $, it follows that 
\begin{align*}
 \check{\varww} (r; u; y ) =	\sum_{\pm}   \frac {{\textstyle \lambda_{\natural} \sqrt{\lambda_{\natural} }} } { {u} }	\int _{  \rho_{\natural} }^{2 \rho_{\natural} }  \varww_{\scriptscriptstyle \pm }(x/ \rho_{\natural}  ) e 	  ( \lambda_{\natural}  (  x  \pm   2 \sqrt{  x})  )  \frac { \nd x } {\sqrt[4]{x}} . 
\end{align*}
By Lemma \ref{lem: analysis of integral}, we infer that $  \check{\varww}  (r; u ; y ) $ is negligibly small unless 
\begin{align}\label{7eq: range of u}
	  \sqrt{u} \asymp T \tau ,   
\end{align}
in which case $$  \check{\varww} (u; r; y )  = \frac {\lambda_{\natural}  \varvv_{\natural} (\lambda_{\natural} )  } {u} e (- \lambda_{\natural} )  , $$
for $ \varvv_{\natural} (\lambda )  = \varvv_{\natural} (\lambda; u ) $ again $\log T$-inert.
Consequently, 
\begin{align*}
 \check{\varww} _{\tau}^{\scriptscriptstyle \pm} (u; y  )  =   \frac {y} {T  } \int_{\pm \tau}^{\pm 2\tau}  
  \varww_{\natural}^{\scriptscriptstyle \pm} \Big( \frac {r} {\tau} \Big)   e \bigg(\frac {T r} {\pi} \mp \frac {uy} {T \uprho_1 (r) }  \bigg) \frac {\nd r} {r} , 
\end{align*}
where $\varww_{\natural}^{\scriptscriptstyle \pm}(r/\tau) $ is   $T^{\vepsilon}$-inert with support on $\pm [\tau, 2\tau ]$;  
more explicitly, 
\begin{align*}
\varww_{\natural}^{\scriptscriptstyle \pm} \Big( \frac {r} {\tau} \Big) = \varvv_{\scriptscriptstyle \pm}  \Big( \frac {r} {\tau} \Big)  e   \bigg(\frac {S  \uprho_2 (r)}  {y}  \bigg)	\cdot   \frac {r } {\uprho_1 (r) }    \varvv_{\tau}^{\natural} \bigg( \hspace{-2pt} \pm \frac {uy} {T \uprho_1 (r)} \bigg). 
\end{align*}
Note that 
\begin{align*}
	\frac 1 {\uprho_1 (r)} = \frac 1 {\exp (r) - 1}  = \frac 1 {r} + O(1), 
\end{align*}
so, in view of \eqref{7eq: range of u}, it is easy to verify that   $ e  \big(   {u y} / {T } \cdot (1/\uprho_1 (r) - 1/r)    \big)$
is $T^{\vepsilon}$-inert provided that $ \tau \Lt T^{\vepsilon}/ \sqrt{T} $. Let us absorb it into $\varww_{\natural}^{\scriptscriptstyle \pm}  (r / \tau )   $ so that  
\begin{align*}
	\check{\varww} _{\tau}^{\scriptscriptstyle \pm} (u; y  )  =   \frac {y} {T  }    \int_{\pm \tau}^{\pm 2\tau}    \varww_{\natural}^{\scriptscriptstyle \pm}  \Big( \frac {r} {\tau} \Big)   e \bigg(\frac {T r} {\pi} \mp \frac {uy} {T r }  \bigg) \frac {\nd r} {r} . 
\end{align*}
Set
\begin{align}\label{7eq: rho, lambda, flat}
	\rho_{\flat} = \frac {T} {\sqrt{\pi u y}}, \qquad \lambda_{\flat} = \frac {\sqrt{uy}} {\sqrt{\pi}}. 
\end{align}
By the change $r \ra r / \rho_{\flat}$, we obtain  
\begin{align*}
	\check{\varww} _{\tau}^{\scriptscriptstyle \pm} (u; y  )  =   \frac {y} {T} \int_{\pm \tau \rho_{\flat}}^{\pm 2\tau \rho_{\flat} }   
	\varww_{\natural}^{\scriptscriptstyle \pm}  \bigg(  \frac r    {\tau \rho_{\flat}}  \bigg)      e  \bigg(\lambda_{\flat} \bigg(r \mp \frac 1 r \bigg) \bigg) \frac {\nd r} {r} . 
\end{align*}
Finally, the statements for $\check{\varww} _{\tau}^{\scriptscriptstyle \pm} (u; y  ) $ follow directly from Lemma   \ref{lem: analysis of integral}. Note here that $\tau \rho_{\flat} \asymp 1$ and  $ \sqrt{\lambda_{\flat}} \asymp \sqrt{T \tau} $ by \eqref{7eq: range of u} and \eqref{7eq: rho, lambda, flat}. 
\end{proof}

For $ \sqrt{u} \asymp T \tau$,  on inserting \eqref{7eq: w(u;y)} into \eqref{7eq: w(u;v), 2}, we obtain 
\begin{align}
	\check{\hat{\varww} }_{\tau}^{\scriptscriptstyle -} (u; v) = \frac { 1 } {T \sqrt{T \tau} } \int  \varvv (y) \varvv_{\flat} (u; y) e    ( - 2 {  \sqrt{uy/ \pi }} -v y  )  \frac {\nd y} {y} . 
\end{align}

\begin{lem}\label{lem: w(u;v)}
	 Assume  $T^{\vepsilon} / T \Lt \tau \Lt  T^{\vepsilon} / \sqrt{T}$ as in Lemma \ref{lem: w (u;v)}. Then $\check{\hat{\varww} }_{\tau}^{\scriptscriptstyle +} (u; v)$ is negligibly small unless $ \sqrt{u} \asymp |v| \asymp T \tau $, in which case 
	 \begin{align} 
	 	 \check{\hat{\varww} }_{\tau}^{\scriptscriptstyle +} (u; v) \Lt \frac {T^{\vepsilon} } { T^2 \tau }. 
	 \end{align} 
\end{lem}

\begin{proof}
	This lemma follows easily from Lemmas \ref{lem: staionary phase} and \ref{lem: 2nd derivative}. 
\end{proof}

For $T^{\vepsilon} / T \Lt \tau \Lt  T^{\vepsilon} / \sqrt{T}$, it follows from Lemma \ref{lem: w(u;v)} that 
\begin{align}\label{7eq: bound for S(b,g)}
	\mathrm{S}_{\tau} (b, g ) \Lt \frac {T^{\vepsilon} } { T^2 \tau }  \mathop{ \sum_{n \asymp T^2 \tau^2 b_{\star}^2  / N} \sum_{|w| \asymp T \tau b^2 k_1 k_2/ C }   }_{w \equiv \widebar{f}_{\star} \widebar{q}_b b_{\mathrm{o}} n (\mathrm{mod}\, g ) }   \hspace{-2pt}  \sqrt{(w, bk_1)}  \Lt T^{\vepsilon} \frac {\tau b_{\star}^2  } {  N} \bigg(1 + \frac {T\tau b^2 k_1 k_2} {C g} \bigg). 
\end{align}

\vspace{5pt} 

\subsubsection{Final Estimations}\label{sec: final}  Assume $ T^{1/2} \leqslant M \leqslant T^{1-\vepsilon} $.  It follows from \eqref{7eq: O2}, \eqref{7eq: bound for S(b,g;r)}, and \eqref{7eq: bound for S(b,g)}, along with \eqref{7eq: k1, k2}, \eqref{8eq: parameters}, and \eqref{7eq: dx, bx} (trivially,  $b_{\star} \leqslant b k_2$,  $ k_1 \leqslant d_1 h_1 $, and $k_2 \leqslant d_2 h_2$),  that
\begin{align*}
	\SO^{\natural}_{2} &	\Lt   T^{ \vepsilon} \sum_{n_2 \sim N_2'  }  \sideset{}{^{{\natural}}} \sum_{b \shskip \Lt C}  \sum_{\ d | d_1}  \sum_{  g | bk_2  }
	\frac {\varPi T N  g \sqrt{bk_1}} {b^2  b_{\star} k_1 k_2} \cdot   \frac {  b_{\star}^2  } { \varPi N} \bigg(1 + \frac {T  b^2 k_1 k_2} {\varPi C    g} \bigg) \\
	& \Lt T^{\vepsilon} N_2'  \sum_{b \shskip \Lt C} T b k_2 \bigg(  \frac {1} {\sqrt{b k_1}} + \frac { T \sqrt{bk_1}  } {\varPi C} \bigg) \\
	& \Lt T^{\vepsilon}   \frac { T^2 C^{3/2} N_2'  k_1^{1/2} k_2}  {\varPi} \\
	& \Lt T^{\vepsilon} \frac {T^{1/2} m_1^{3/2} N_1^{3/2} N_2  } {\varPi d_1 d_2^{3/2} h_1 h_2^{5/2} } . 
\end{align*}
Note that we have $N_1 \Lt T^{2+\vepsilon}$ and hence the next lemma.   

\begin{lem}
	 Assume  $ T^{1/2} \leqslant \varPi \leqslant T^{1-\vepsilon} $. Then for $\langle \boldsymbol{m N} \rangle > T/ \varPi $, we have 
	 \begin{align}
	 	\SO_2 (\boldsymbol{d}, \boldsymbol{f},    \boldsymbol{h}; \boldsymbol{N} ) \Lt   T^{\vepsilon} \frac {T^{3/2} m_1^{3/2}  \|\boldsymbol{N} \|  } {\varPi \|\boldsymbol{dh} \| }. 
	 \end{align}
\end{lem}

\begin{coro}\label{cor: O, 2}
		The contribution  to $\SO_2 (\boldsymbol{m} )$ {\rm(}see {\rm\eqref{7eq: O2(m)}}{\rm)}  from those $\SO_2 (\boldsymbol{d}, \boldsymbol{f},    \boldsymbol{h}; \boldsymbol{N} )$ such that $$ \max \big\{\langle  \boldsymbol{  m N } \rangle, \langle  \widetilde{\boldsymbol{  m N }} \rangle \big\} > {T  / \varPi}$$ has  bound  $O \big(   {T^{5/2+\vepsilon} \max \{ \boldsymbol{m} \}^{3/2} / \varPi} \big)$. 
\end{coro}

\subsection{Conclusion} 

Recall that ${\SC}^{\natural }_{2}  (\boldsymbol{m} ) $  differs from  ${\SC}_{2}  (\boldsymbol{m}  ) $ by  an exponentially small error. Thus, in view of  \eqref{7eq: C+E=D+O}, Theorem \ref{thm: C2(m)} is a direct consequence of Lemmas \ref{lem: Eis E1,2}, \ref{lem: main term}  and Corollaries \ref{cor: O, 1}, \ref{cor: O, 2}.

\section{Mean Lindel\"of Bound for the Untwisted Second Moment} 

\delete{The purpose of this and the next section is to study the untwisted spectral second moment  $\SC_2 = \SC_2 (1, 1)$: 
\begin{align}
	\SC_2 = \sum_{j = 1}^{\infty}   \omega_j     |L (s_j,  \vQ \otimes u_j )|^2   \upphi (t_j), \qquad \upphi (t) =   \exp \left(  - \frac {(t - T)^2}  {\varPi^2} \right) . 
\end{align}   
The readers who are interested in the proof of Theorems   \ref{thm: non-vanishing} and \ref{thm: non-vanishing, short} may jump to \S \ref{sec: mollified}. }

In this section, we verify the mean Lindel\"of bound in Theorem \ref{thm: Lindelof}: 
\begin{align}
	\sum_{j = 1}^{\infty}   \omega_j     |L (s_j,  \vQ \otimes u_j )|^2  \exp \left(  -   {(t_j - T)^2} / {\varPi^2} \right)  = O_{\vQ, \vepsilon} (\varPi T^{1+\vepsilon})     . 
\end{align} 
To this end, it suffices to prove the next theorem  as one may  divide  $ ( T - \varPi \log T, T + \varPi \log T] $ into intervals of unity length. 

\begin{theorem}\label{thm: Lindelof, 2}
	We have 
	\begin{align}\label{10eq: Lindelof}
		\sum_{T < t_j \leqslant T+1}   |L (s_j,  \vQ \otimes u_j )|^2 \Lt_{\vQ, \vepsilon} T^{1+\vepsilon} ,
	\end{align}
for any $\vepsilon > 0$ and $T \Gt 1$. 
\end{theorem}

Note that the harmonic weight $\omega_j $ does not play a role in view of  $t_j^{- \vepsilon} \Lt \omega_j \Lt t_j^{\vepsilon}$ (see \eqref{3eq: omegaj}). 

\subsection{The Large Sieve of Luo}

It was stated without proof by
Iwaniec \cite{Iwaniec-Spectral-Weyl} and proven independently by Luo \cite{Luo-LS} and Jutila \cite{Jutila-LS} that 
\begin{align}\label{10eq: LS, M=1}
	\sum_{T < t_j \leqslant T+1} \omega_j  \bigg| \sum_{ n \leqslant N}  a_{n} \lambda_{j} (n)  \bigg|^2 \Lt_{\vepsilon}  (T +  N  ) (TN)^{\vepsilon} \sum_{ n \leqslant N}  |a_{n}|^2 ,
\end{align}
while Luo observed that, by partial summation, \eqref{10eq: LS, M=1} is equivalent to its twisted variant:
\begin{align}\label{10eq: LS, M=1, twisted}
	\sum_{T < t_j \leqslant T+1}  \omega_j \bigg| \sum_{ n \leqslant N}  a_{n} \lambda_{j} (n) n^{it_j}  \bigg|^2 \Lt_{\vepsilon}  (T +  N  ) (TN)^{\vepsilon} \sum_{ n \leqslant N}  |a_{n}|^2 ;  
\end{align}
the twist  $n^{it_j}$ does not play a role because $t_j$ is restricted in a segment of unity length. 

\subsection{Proof of Theorem \ref{thm: Lindelof, 2}} 
As a consequence of his large sieve and   approximate functional equation, Luo's Theorem 1 in \cite{Luo-LS} states that 
\begin{align*}
	\sum_{T < t_j \leqslant T+1}   |L (s_j,    u_j )|^4 \Lt_{  \vepsilon} T^{1+\vepsilon} . 
\end{align*}
Of course \eqref{10eq: Lindelof} is an analogue and follows by adapting Luo's arguments. Nevertheless,  we provide here an alternative proof by the approximate function equation of $L (s_j, \vQ \otimes u_j)$ as described in \S \ref{sec: AFE}. 

Let $T < t_j \leqslant T+1$.  It follows from \eqref{3eq: AFE} (with $X = 1$), \eqref{3eq: V(y;t), 1}, and \eqref{3eq: V(y;t), 2} that, up to a negligibly small error, $ L  (s_j, \vQ \otimes u_j)$ is equal to
\begin{align*}
\int_{\vepsilon -i \log T}^{\vepsilon + i \log T}   \bigg\{ \sum_{\pm}  \upepsilon (t_j)^{(1\mp 1) /2}  \upgamma_{1} (v, \pm t_j )  \sum_{ h \leqslant T^{1/2+\vepsilon}} \frac {\varepsilon_q (h)} {h^{1\pm 2it_j + 2v }} 	\sum_{n \leqslant T^{1+\vepsilon}/ h^2 }   \frac {a(n) \lambda_j (n) } {n^{1/2\pm it_j + v} }  \bigg\}  \frac {\nd v} {\theta^v v}  . 
\end{align*} 
Note that $ \upepsilon (t_j)$ has unity norm, while, by \eqref{3eq: bound for G, 1}, $\upgamma_{1} (v,  \pm t_j )$ has uniform bound  (on the integral contour):
\begin{align*}
	\upgamma_{1} (v,  \pm t_j )   \Lt_{\vepsilon, k} T^{\vepsilon} .
\end{align*} 
Then Cauchy--Schwarz yields 
\begin{align*}
	\sum_{T < t_j \leqslant T+1}  \omega_j |L (s_j,  \vQ \otimes u_j )|^2 \Lt  T^{ \vepsilon} &   \int_{\vepsilon -i \log T}^{\vepsilon + i \log T} \sum_{\pm} \sum_{ h \leqslant T^{1/2+ \vepsilon}} \frac 1 {h } \\
	& \cdot \sum_{T < t_j \leqslant T+1} \omega_j \bigg| \sum_{n \leqslant T^{1+\vepsilon}/ h^2 }   \frac {a(n) \lambda_j (n) } {n^{1/2\pm it_j + v} }  \bigg|^2   \left|  {\nd v}   \right|   , 
\end{align*}
up to a negligible error. Therefore  \eqref{10eq: Lindelof}  follows by an application of Luo's large sieve \eqref{10eq: LS, M=1, twisted}, along with Deligne's bound \eqref{3eq: Deligne}.

\section{A Useful Unsmoothing Lemma}

As in \eqref{5eq: phi}, for large parameters $T, \varPi$ with $T^{\vepsilon} \leqslant \varPi \leqslant T^{1-\vepsilon}$,  define
\begin{align}\label{8eq: phi}
\upphi (t) =	\upphi_{T, \varPi} (t) = \exp \left( - \frac {(t-T)^2}  {\varPi^2} \right). 
\end{align}
In this section, we introduce a useful unsmoothing process as in \cite[\S 3]{Ivic-Jutila-Moments} by an average of the weight $\upphi_{T, \varPi} $ in the $T$-parameter.


\begin{defn}\label{defn: unsmooth}
	Let $   H \leqslant T / 3 $ and $T^{\vepsilon} \leqslant \varPi \leqslant H^{1-\vepsilon}  $. Define  
	\begin{align}\label{5eq: defn w(nu)}
		\uppsi (t) = \uppsi_{T, \varPi, H} (t) = \frac 1 { \sqrt{\pi}  \varPi} \int_{\, T- H}^{T + H} \upphi_{K, \varPi} (t) \nd K  .
	\end{align}
\end{defn}

The next lemma is essentially   Lemma 5.3 in \cite{Qi-Liu-Moments}, adapted from the arguments in \cite[\S 3]{Ivic-Jutila-Moments}. 

\begin{lem}\label{lem: unsmooth, 1}
	Let  $\lambda $  be a real constant.   	Suppose that $a_j $ is a sequence such that  
	\begin{align}\label{10eq: assumption}
		\sum  \omega_j \upphi   (t_j) |a_j|  = O_{\lambda, \vepsilon}  \big( \varPi T^{\lambda}  \big) 
	\end{align}
	for any $\varPi$ with $ T^{\vepsilon} \leqslant \varPi \leqslant T^{1-\vepsilon}  $. Then for  $ \varPi^{1+\vepsilon} \leqslant   H \leqslant T / 3 $ we have
	\begin{align}
		\mathop{\sum  }_{  |t_j - T| \shskip \leqslant H}  \omega_j  a_j =  \sum \omega_j \uppsi  (t_j) a_j + O_{\lambda, \vepsilon}  \big( \varPi T^{\lambda}  \big) . 
	\end{align}
\end{lem}

Later, we shall choose $a_j$ to be $\delta (L (s_j, \vQ \otimes u_j) \neq 0)$ (the Kronecker $\delta$ that detects $L (s_j, \vQ \otimes u_j) \neq 0$), $L (s_j, \vQ \otimes u_j) $, or $|L (s_j, \vQ \otimes u_j)|^2$ (see \S \S \ref{sec: conclusion, non-vanishing} and \ref{sec: unsmooth}). Note that in all these cases  we  may let  $\lambda = 1$ or $ 1+\vepsilon$ in \eqref{10eq: assumption} in light of the (harmonic weighted) Weyl law or the mean Lindel\"of bounds.

\section{The Mollified Moments} \label{sec: mollified}

Historically, the mollifier method was first used  on zeros of the Riemann zeta  function by Bohr, Landau, and Selberg \cite{Bohr-Landau,Selberg-Mollifier}. In this section, we shall mainly follow the ideas and maneuvers in \cite{IS-Siegel,KM-Analytic-Rank-2}.

\subsection{Setup} 

We introduce the mollifier  
\begin{align}\label{9eq: defn Mj}
	M_j	= \sum_{m \leqslant M}  \frac {x_m   a (m) \lambda_j  (m)} {m^{1/2+it_j}}  . 
\end{align} 
Assume that the  coefficients $ x_m  $ 
are real-valued, normalized so that 
\begin{align}
	\label{9eq: x1=1}
	x_1 = 1, 
\end{align}
supported on square-free $m $ with $(m, q) = 1$ and $m \leqslant M$, and not too large (see \eqref{9eq: xi(m)} and \eqref{9eq: xm}), say
\begin{align}\label{9eq: bound for xm}
	x_m = O (\tau (m)). 
\end{align} 
Note that
\begin{align}\label{9eq: |Mj|2}
	|M_j|^2 = \mathrm{Re} \mathop{\sum \sum}_{m_1, m_2 \leqslant M}  \frac {x_{m_1} x_{m_2}   a (m_1) a(m_2)  \lambda_j (m_1) \lambda_j {(m_2)}  } {\sqrt{m_1 m_2} (m_1/m_2)^{it_j} } . 
\end{align}
For $\upphi (t)$ given in \eqref{5eq: phi}, consider the mollified moments
\begin{align}\label{9eq: defn M1}
	\SM_1 =  \sum_{j = 1}^{\infty}   \omega_j  M_j   L (s_j,  \vQ \otimes u_j )  \upphi (t_j), 
\end{align} 
\begin{align}\label{9eq: defn M2}
	\SM_2 =  \sum_{j = 1}^{\infty}   \omega_j  |M_j|^2     |L (s_j,  \vQ \otimes u_j )|^2    \upphi (t_j). 
\end{align} 
By the Cauchy inequality, 
\begin{align}\label{9eq: Cauchy} 
		\sum_{L (s_j, \vQ \shskip \otimes \shskip u_j) \neq 0} \omega_j \upphi (t_j)   \geqslant \frac { \lp \,  {\displaystyle \sum } \, \omega_j  M_j L (s_j, \vQ \otimes u_j)  \upphi (t_j)   \rp^2  }  { {\displaystyle \sum } \, \omega_j \left|M_j  L (s_j, \vQ \otimes u_j) \right|^2 \upphi (t_j)  }  .  
\end{align}
Thus our task is to calculate the mollified moments $  \SM_1$ and $\SM_2 $ and then carefully choose the coefficients $x_m$ to maximize the ratio $\SM_1^2 / \SM_2$.

\subsection{The Mollified First Moment}

By the definitions in \eqref{2eq: moments M1}, \eqref{9eq: defn Mj}, and \eqref{9eq: defn M1}, we have
\begin{align*}
	\SM_1 = \sum_{m \leqslant M} \frac {x_m a(m)} {\sqrt{m}} \SC_1 (m),  
\end{align*}
and hence by the bounds for $a (m)$ and $x_m$ in \eqref{3eq: Deligne} and \eqref{9eq: bound for xm}, the following asymptotic formula for $\SM_1 $ is a direct consequence of Theorem \ref{thm: C1(m)}. 

\begin{lem}\label{lem: M1}
Let $\varPi = T^{ \vnu }$    and  $M = T^{\varDelta}$. Then for 
$\varDelta <    \min \{2\vnu, 6\vnu-2  \}/ 7$ there exists $\delta = \delta (\vnu, \varDelta) > 0$ such that 
\begin{align}\label{9eq: M1}
	\SM_1 = \frac {2\varPi T} {\pi \sqrt{\pi}} + O \big(\varPi T^{1-\delta} \big). 
\end{align}
\end{lem}

\subsection{The Mollified Second Moment}

In view of the definitions in \eqref{2eq: moments M2}, \eqref{9eq: |Mj|2}, and \eqref{9eq: defn M2}, we have
\begin{align*}
	\SM_2 =  \mathop{\sum \sum}_{m_1, m_2 \leqslant M}  \frac {x_{m_1} x_{m_2}   a (m_1) a(m_2)    } {\sqrt{m_1 m_2}   }  \SC_2 (m_1, m_2) .
\end{align*}
Now we insert the asymptotic formula in Theorem \ref{thm: C2(m)}, use \eqref{3eq: Deligne} and \eqref{9eq: bound for xm} to bound the error term by 
\begin{align*}
 &   	\frac {   \varPi^3  } {T^{1-\vepsilon}   } \mathop{\sum \sum}_{m_1, m_2 \leqslant M}  \frac {\tau (m_1)^2 \tau(m_2)^2  (m_1, m_2)  } { {m_1 m_2}   } + {\frac {  T^{5/2+\vepsilon} } {  {\varPi} } } { \mathop{\sum \sum}_{m_1, m_2 \leqslant M}  \frac {\tau (m_1)^2 \tau(m_2)^2    } {\sqrt{m_1 m_2}   }  \textstyle \sqrt{m_1^3+m_2^3}}	  \\
  & \Lt    \frac { \varPi^3 M^{\vepsilon}  } {T^{1-\vepsilon}} + {\frac {  T^{5/2+\vepsilon} M^{5/2+\vepsilon}  } {  {\varPi} } } ,  
\end{align*} 
and then arrive at the asymptotic formula for $\SM_2 $ in the next lemma. 

\begin{lem}\label{lem: M2}
	Let $\varPi = T^{ \vnu }$    and  $M = T^{\varDelta}$.  Then for  $\varDelta < (4 \vnu - 3) /5  $ there exists $\delta = \delta(\vnu, \varDelta ) > 0$ such that 
	\begin{align}\label{9eq: M2}
		\SM_2 = \frac {4 \gamma_1 \varPi T} {\pi \sqrt{\pi}}  \big( {\SM}_{20}  + \SM_2'\big) +  O \big(\varPi T^{1-\delta} \big), 
	\end{align}
	with 
	\begin{align}\label{9eq: M20}
		{\SM}_{20}  = \sum_{m} \mathop{\sum \sum}_{(m_1^{\star}, m_2^{\star}) = 1} \frac {a (m_1^{\star})^2 a (m_2^{\star})^2 a (m)^2 \SB (m) } { m_1^{\star} m_2^{\star} m } x_{m_1^{\star} m} x_{m_2^{\star} m} \log \frac {\hat{T} } {\sqrt {m_1^{\star} m_2^{\star}}  }, 
	\end{align}
	\begin{align}
		{\SM}_{2}'   =   \sum_{m} \mathop{\sum \sum}_{(m_1^{\star}, m_2^{\star}) = 1} \frac {a (m_1^{\star})^2 a (m_2^{\star})^2 a (m)^2 \SB'  (m) } { m_1^{\star} m_2^{\star} m } x_{m_1^{\star} m} x_{m_2^{\star} m} , 
	\end{align}
	where $\hat{T}$ is a multiple of $T$ defined by 
	\begin{align}
		\log \hat{T} =   {\gamma_0'} / {\gamma_1} + \log T. 
	\end{align} 
\end{lem}

By relaxing the co-primality condition $(m_1^{\star}, m_2^{\star}) = 1$ in \eqref{9eq: M20} by the M\"obius function, we have 
\begin{align*}
	 {\SM}_{20} =    \sum_{m} \sum_{d}  \mathop{\sum \sum}_{n_1, \, n_2}  \frac {\mu (d) a(d)^4  a (m)^2 \SB (m) a (n_1)^2 a (n_2)^2 } { d^2 m n_1 n_2 } x_{dmn_1} x_{dm n_2}  \log \frac {\hat{T} } {d \sqrt {n_1 n_2}  } . 
\end{align*}
Hence we may write
\begin{align}\label{9eq: M20=MMM}
	{\SM}_{20} = \log \hat{T} \cdot \hat{\SM\ }\hspace{-3.5pt}_{20} - \breve{\SM\ }\hspace{-3.5pt}_{20}  -  {\SM}_{20}', 
\end{align}
where $ \hat{\SM\ }\hspace{-3.5pt}_{20}$, $ \breve{\SM\ }\hspace{-3.5pt}_{20}$, and $  {\SM}_{20}'$ are in the diagonal form: 
\begin{align}\label{9eq: M2... 1}
	\hat{\SM\ }\hspace{-3.5pt}_{20} = \sum_{m} \frac {a(m)^2 \xi (m)} {m}   y_m^2, \qquad \breve{\SM\ }\hspace{-3.5pt}_{20} = \sum_{m} \frac {a(m)^2 \xi (m)} {m} y_m \breve{y}_m,  
\end{align}
\begin{align}\label{9eq: M2... 2}
 {\SM}_{20}' = \sum_{m} \frac {a(m)^2 \breve{\xi} (m)} {m}  y_m^2, 
\end{align}
with
\begin{align}\label{9eq: ym}
	y_{m} = \sum_{n}  \frac {a(n)^2} {n} x_{m n} , \qquad 	\breve{y}_m = \sum_{n}  \frac {a(n)^2 \log n} {n} x_{m n}, 
\end{align}
\begin{align}\label{9eq: nu(m)} 
	\xi (m) = \sum_{d | m} \frac{ a(d)^2 \mu(d) \SB (m/d)}{d}, \qquad \breve{\xi} (m) =   \sum_{d | m} \frac{ a(d)^2 \mu(d)    \SB (m/d)}{d} \log d . 
\end{align}
Similar to \eqref{9eq: M2... 1}, \eqref{9eq: M2... 2}, and \eqref{9eq: nu(m)}, we have 
\begin{align}\label{9eq: M2', xi'}
	\SM_{2}' = \sum_{m} \frac {a(m)^2 \xi'  (m)} {m}  y_m^2, \qquad  \xi'  (m) =   \sum_{d | m} \frac{ a(d)^2 \mu(d)    \SB'  (m/d)}{d}. 
\end{align}

By M\"obius inversion, it follows from \eqref{9eq: ym} that 
\begin{align}\label{9eq: xn=y}
	 x_n = \sum_{m} \frac {a(m)^2 \mu (m)} {m} y_{m n},   \qquad 
	\breve{y}_m = \sum_{n} \frac {a(n)^2 \varLambda(n)} {n} y_{m n}, 
\end{align}
where  $\varLambda (n) $ is the von Mangoldt function. Note that the assumption $x_1 = 1$ in \eqref{9eq: x1=1} now becomes the linear constraint
\begin{align}\label{9eq: x1=1, 2}
	\sum_{m} \frac {a(m)^2 \mu (m)} {m} y_{m} = 1. 
\end{align}
By Cauchy, the optimal choice of $ y_{m}  $ ($m$ square-free, $(m, q)= 1$, and $m \leqslant M$) which optimizes the quadratic form $ \hat{\SM\ }\hspace{-3.5pt}_{20} $ in \eqref{9eq: M2... 1} with respect to \eqref{9eq: x1=1, 2} is given by 
\begin{align}\label{9eq: choice ym}
	y_m =  \frac {1} {\Xi (M)  } \cdot \frac {\mu (m)} {\xi (m)}    , \qquad \Xi (M) = \sideset{}{^{_{  {\scriptstyle \prime}}}}  \sum_{m \shskip \leqslant M} \frac {  a(m)^2 \mu (m)^2 } {m \xi (m)} ,
\end{align}
where the prime on the $m$-sum means  $(m, q) = 1$.     Define the associated Dirichlet series 
\begin{align}
	D (s, \Xi ) =  \sideset{}{^{_{  {\scriptstyle \prime}}}}  \sum_{m =1 }^{\infty} \frac {  a(m)^2 \mu (m)^2  } {  \xi (m)  m^s } =  \prod_{p \nmid q}   \bigg( 1 + \frac {a(p)^2} {\xi (p) p^s}  \bigg) . 
\end{align}
It follows from the product form of  $\SB (m)$ as in \eqref{7eq: B(m)} that 
\begin{align}\label{9eq: xi(m)}
\xi (m) = \prod_{p | m } \bigg(1 - \frac {a(p)^2} {p(p+1) } \bigg),
\end{align}
and hence $ D (s, \Xi ) $ 
has analytic continuation to $\mathrm{Re}(s) > 1/2$ (to see this, compare its logarithm with that of (the Euler product of) $L (s, \vQ \otimes \vQ)$ as in \eqref{3eq: L(s, QQ)}) with a simple pole at $s = 1$. 

\begin{defn}\label{defn: gammaf}
	Define $\gamma_{\flat} = \gamma_{\flat} (\vQ)  $ to be the residue of $	D (s, \Xi ) $ at $s = 1$; more explicitly,  
	\begin{align}
		\gamma_{\flat} =  \prod_{p \nmid q}  \bigg(1 - \frac {a(p)^2} {p(p+1) } \bigg)^{-1} \bigg(1 + \frac {a(p)^2} {p+1 } \bigg) \bigg( 1 - \frac {1} { p } \bigg) \cdot  \prod_{p | q} \bigg( 1 - \frac {1} { p } \bigg) . 
	\end{align}
\end{defn}

It follows that 
\begin{align}\label{9eq: Xi, asymp}
\Xi (M) 
= \gamma_{\flat}   \log M 	+ O(1). 
\end{align}

Moreover, by  \eqref{9eq: xn=y} and \eqref{9eq: choice ym}, our $x_{m}$ has bound  
\begin{align}\label{9eq: xm}
	 x_m = O \lp \frac { 1} {\xi (m)} \rp , 
\end{align}  so it is easy to check by \eqref{3eq: Deligne} and \eqref{9eq: xi(m)} the growth condition for $x_m$ as in \eqref{9eq: bound for xm}. 

\vspace{7.5pt} 

\subsubsection{Evaluations of $	\hat{{\protect\SM}\ }\hspace{-3.5pt}_{20} $ and  $	\breve{{\protect\SM}\ }\hspace{-3.5pt}_{20} $} 

By definition, 
\begin{align}\label{9eq: hM20}
	\hat{\SM\ }\hspace{-3.5pt}_{20} = \frac 1 {\Xi (M) } . 
\end{align}
By the Prime Number Theorem for $L (s, \vQ \otimes \vQ)$  (see \cite[Theorem 5.13]{IK}), 
\begin{align}\label{9eq: PNT}
 \sum_{ p \shskip \leqslant X} \frac { a(p)^2 \log p } {p} = \log X + O (1). 
\end{align} 
For $m$ square-free,  
\begin{align}\label{9eq: PNT, 2}
	\sum_{p | m} \frac {\log p} {p} = O    (\log \log (3m)).
\end{align}
Note that  
$	  1 / {\xi (p)} = 1 + O (1/p^2)  $, so it follows from   \eqref{9eq: xn=y}, \eqref{9eq: choice ym}, \eqref{9eq: PNT}, and \eqref{9eq: PNT, 2} that 
\begin{align*}
	  {\breve{y}_m}  /{y_m} = - 
	\mathop{\sum_{p  \nmid qm  }}_{  p \shskip \leqslant M/m  } \frac {a(p)^2 \log p} {\xi (p) p }  
	  =   -      \log  ( {M} /{m}) + O (\log \log (3 m))  . 
\end{align*}
Consequently,
\begin{align*}
	\begin{aligned}
		\breve{\SM\ }\hspace{-3.5pt}_{20} & = \sum_{m \leqslant M} \frac {a(m)^2 \xi (m) } {m}   y_m^2 \cdot \log m - (\log M - O (\log \log M) ) \sum_{m \leqslant M} \frac {a(m)^2 \xi (m)} {m}   y_m^2    , 
	\end{aligned}
\end{align*}
and a partial summation using \eqref{9eq: Xi, asymp} yields
\begin{align} \label{9eq: bM20}
		\breve{\SM\ }\hspace{-3.5pt}_{20}   = -  \frac {\log M } {2 \, \Xi (M) } + O  \bigg( \frac {\log \log M} {\Xi (M) } \bigg) . 
\end{align}

\subsubsection{Estimations of ${\protect\SM}_{2}'  $ and  $	{\protect\SM}_{20}' $} 
Recall the definitions of $	\xi  (m)$, $\breve{\xi} (m) $, and $	\xi'  (m)$ in \eqref{9eq: nu(m)} and \eqref{9eq: M2', xi'}.  We have 
\begin{align*}
	\breve{\xi} (m) = - \sum_{p | m} \frac {a(p)^2} {p} \log p \cdot \xi (m/ p) , 
\end{align*} 
and by \eqref{7eq: B'(m)} 
\begin{align*}
	\xi'  (m) =  \sum_{p | m } B(p)   \log p \cdot \xi (m/p) , \qquad B(p) = O \lp \frac 1 {p} \rp . 
\end{align*} 
Thus 
\begin{align*}
	   {\breve{\xi} (m)} / {\xi (m)} , \   {	\xi'  (m)} / {\xi (m)} = O (\log \log (3m)), 
\end{align*}
due to \eqref{9eq: PNT, 2}. 
 By trivial estimation,  
\begin{align}\label{9eq: M2' M20'}
	{ \SM}_{2}' , \, { \SM}_{20}' = O \lp \frac {\log \log M} {\Xi (M) } \rp. 
\end{align}

\subsubsection{Conclusion} 

By \eqref{9eq: M2}, \eqref{9eq: M20=MMM}, \eqref{9eq: Xi, asymp}, \eqref{9eq: hM20}, \eqref{9eq: bM20}, and  \eqref{9eq: M2' M20'}, we conclude with the following asymptotic formula for $\SM_2$. 

\begin{lem}\label{lem: M2, 2}
	 Let $\varPi = T^{ \vnu  }$    and  $M = T^{\varDelta}$.  Then for  $  \varDelta < (4 \vnu - 3) /5  $  we have
	 \begin{align}
	 	\SM_2 = \frac {2 \varPi T} {\pi \sqrt{\pi}}  \cdot  \frac {2+\varDelta}  {\gamma \varDelta}  \bigg( 1 + O  \bigg( \frac {\log\log T} {\log T}   \bigg) \bigg), \qquad \gamma = \frac {\gamma_{\flat} } {\gamma_1}  . 
	 \end{align}
\end{lem} 

Note that the Euler product form of  $\gamma = {\gamma_{\flat}} / {\gamma_{1} }$ in \eqref{1eq: beta(Q)} follows from those of $\gamma_1$ and $\gamma_{\flat}$ in Definitions \ref{defn: constants, 1}  and \ref{defn: gammaf}.

\subsection{Conclusion} \label{sec: conclusion, non-vanishing}
For  $3/4 < \vnu < 1$  let   $\varPi = T^{ \vnu  }$.   
By Lemmas \ref{lem: M1} and \ref{lem: M2, 2}, it follows from Cauchy   in \eqref{9eq: Cauchy}  that
\begin{align}
	\sum_{L (s_j, \vQ \shskip \otimes \shskip u_j) \neq 0} \omega_j \upphi (t_j)   \geqslant \frac {2 \varPi T} {\pi \sqrt{\pi}} \gamma \bigg( \frac  { 4\vnu - 3} {4\vnu + 7} - \vepsilon  \bigg) , 
\end{align} 
 as $T \ra \infty$  for any $\vepsilon > 0$.  
Next, we apply Lemma \ref{lem: unsmooth, 1} with $a_n =   \delta (L (s_j, \vQ \otimes u_j) \neq 0)$ (namely, the Kronecker $\delta$ symbol) and $H = \varPi^{1+\vepsilon}$, obtaining 
\begin{align}\label{9eq: unsmoothed}
		\mathop{\sum_{ |t_j - T| \shskip \leqslant H}}_{L (s_j, \vQ \shskip \otimes \shskip u_j) \neq 0}  \omega_j     \geqslant \frac {4 H T} {\pi^2 } \gamma \bigg( \frac  { 4 \mu - 3} {4 \mu + 7} - \vepsilon  \bigg) ,\qquad \mu = \frac {\log H} {\log T}, 
\end{align} 
for any $ T^{3/4+\vepsilon} \leqslant  H \leqslant T/3$. 
 By the harmonic weighted Weyl law as in \cite{XLi-Weyl}:
\begin{align}\label{9eq: Weyl, 1}
	\sum_{t_j \leqslant T} \omega_j = \frac {T^2} {\pi^2} + O (T), 
\end{align} 
we have
\begin{align}\label{9eq: Weyl, 2}
	 {\sum_{ |t_j - T| \shskip \leqslant H}}  \omega_j     = \frac {4 H T} {\pi^2 } + O (T) . 
\end{align} 

\begin{theorem}
Let $\vepsilon > 0$ be arbitrarily small. As $T \ra \infty$, we have  
	\begin{align}\label{9eq: case t<T}
		\mathop{\sum_{ t_j \leqslant T}}_{L (s_j, \vQ \shskip \otimes \shskip u_j) \neq 0}  \omega_j  \geqslant \gamma \bigg(\frac 1 {11} - \vepsilon \bigg) \sum_{ t_j \leqslant T} \omega_j . 
	\end{align}
	and, for $3/4 < \mu < 1$,  
	\begin{align}\label{9eq: case t-T<H}
		\mathop{\sum_{ |t_j - T| \shskip \leqslant T^{\mu}}}_{L (s_j, \vQ \shskip \otimes \shskip u_j) \neq 0}  \omega_j  \geqslant \gamma \bigg( \frac  { 4 \mu - 3} {4 \mu + 7} - \vepsilon  \bigg)   \sum_{ |t_j - T| \shskip \leqslant T^{\mu}} \omega_j . 
	\end{align}
\end{theorem}

Note that \eqref{9eq: case t<T} follows from \eqref{9eq: unsmoothed} and \eqref{9eq: Weyl, 1} with $ H = T/3 $ along with a dyadic summation, while \eqref{9eq: case t-T<H} follows directly from \eqref{9eq: unsmoothed} and \eqref{9eq: Weyl, 2}.    

Since the harmonic weight 
\begin{align*}
	\omega_j = \frac {2} {  L(1, \mathrm{Sym}^2 u_j)}, 
\end{align*}
one may remove $\omega_j$ without compromising the strength of results by the method of Kowalski and Michel \cite{KM-Analytic-Rank}, adapted for the Maass-form case  in \cite{BHS-Maass} (see also the paragraph below (2.9) in \cite{IS-Siegel}).  Thus after the removal of  $\omega_j$  we obtain the non-vanishing results in Theorems \ref{thm: non-vanishing} and \ref{thm: non-vanishing, short}.

\section{Moments with no Twist and Smooth Weight}\label{sec: unsmooth}

\subsection{Preliminary Results} \label{sec: prelim}
First, we show that if Lemma  \ref{lem: unsmooth, 1} is applied directly to Corollaries  \ref{cor: 1st moment} and \ref{cor: 2nd moment}, then we may remove the smooth weight $\exp (-t_j/T)$ in Luo's asymptotic formulae \eqref{2eq: Luo 1st} and \eqref{2eq: Luo 2nd} as follows (with essentially the same error terms): 
\begin{align}\label{11eq: 1st}
	\sum_{t_j \leqslant T} \omega_j L (s_j,  \vQ \otimes u_j ) = \frac {1} {\pi^2} T^{2} + O (T^{7/4+\vepsilon}), 
\end{align} 
\begin{align}\label{11eq: 2nd} 
	\sum_{t_j \leqslant T} \omega_j |L (s_j,  \vQ \otimes u_j )|^2 = \frac {2\gamma_1} {\pi^2}  T^2 \log T + \frac {2\gamma_0' - \gamma_1 } {\pi^2 } T^2  + O (T^{11/6+\vepsilon}). 
\end{align}
Later, we shall use some ideas from \cite{Qi-Ivic} to further improve the error terms (see Theorem \ref{cor: unsmoothed}). 

\begin{coro}\label{cor: 1st}
	Let $ T^{1/3+ \vepsilon} \leqslant H \leqslant T/3$. Then 
	\begin{align}\label{11eq: 1st moment}
		\sum_{ |t_j - T| \leqslant H }   \omega_j L (s_j,  \vQ \otimes u_j ) =  \frac {4  } {\pi^2}  H T + O \big(T^{\vepsilon} (H^{1/2} T^{7/6} + H^{3/4} T  ) \big). 
	\end{align}
\end{coro}

\begin{proof}
	By applying  Lemma \ref{lem: unsmooth, 1} to the asymptotic formula in Corollary \ref{cor: 1st moment}, we have
	\begin{align*}
		\sum_{ |t_j - T| \leqslant H }   \omega_j L (s_j,  \vQ \otimes u_j ) =  \frac {4  } {\pi^2}  H T + O  \bigg(T^{\vepsilon} \bigg(\frac {H T^{4/3} } {\varPi } + \frac {H T} {\varPi^{1/3}} + \varPi T \bigg)\bigg),
	\end{align*}
	for any $\varPi^{1+\vepsilon} \leqslant H \leqslant T/3$, and  we obtain \eqref{11eq: 1st moment} on choosing $\varPi =  H^{1/2} T^{1/6} + H^{3/4}$. 
\end{proof}


\begin{coro} \label{cor: 2nd}
	Let $ T^{3/4+\vepsilon} \leqslant H \leqslant T/3$. Then
	\begin{align}\label{11eq: 2nd moment}
		\sum_{ |t_j - T| \leqslant H }   \omega_j |L (s_j,  \vQ \otimes u_j )|^2  = \frac {4} {\pi^2} \int_{T-H}^{T+H} K  (\gamma_1 \log K + \gamma_0') \nd K + O  ( H^{1/3} T^{3/2+\vepsilon}). 
	\end{align}
\end{coro}

\begin{proof}
	By applying Lemma \ref{lem: unsmooth, 1}  to the asymptotic formula in Corollary \ref{cor: 2nd moment}, we infer that the spectral second moment on the left of  \eqref{11eq: 2nd moment} is 
	\begin{align*}
		\frac {4} {\pi^2} \int_{T-H}^{T+H} K^2 (\gamma_1 \log K + \gamma_0') \nd K + O \bigg(T^{\vepsilon} \bigg(\frac {H T^{5/2}} {\varPi^2 } + \varPi T\bigg)\bigg), 
	\end{align*}
	for any $\varPi^{1+\vepsilon} \leqslant H \leqslant T/3$, and we obtain \eqref{11eq: 2nd moment} on choosing $\varPi = H^{1/3} T^{1/2}$. 
\end{proof}

Thus \eqref{11eq: 1st} and \eqref{11eq: 2nd} follow from the above asymptotic formulae for $H = T/3$  along with a dyadic summation.

	\subsection{Refined Analysis for the Bessel Integral {\large$H(x, y)$}}\label{sec: Bessel, refined}
	
	Let the notation be as in \S\S \ref{sec: AFE} and \ref{sec: Bessel}.
 First of all,  by Stirling's formula (see \eqref{app1: refined quot} and \eqref{app1: Delta}) we have the following decomposition of  $ \upgamma_{\upsigma } (v, t) $ (see \eqref{3eq: defn G+-} and \eqref{3eq: defn G2}).  This is a refinement of Lemma \ref{lem: G(v,t)}. 
	
	\begin{lem}
Let $\mathrm{Re}(v) > 0$ be fixed. Let $t$ be real. Write
\begin{align}\label{12eq: gamma}
	\upgamma_{\upsigma } (v, t) =  \frac {\Gamma  ( k + v )^{ {\upsigma}}    } {\Gamma (k)^{ {\upsigma}}    } (2i^{\upsigma} t)^{\upsigma v}  \big(\exp (v^2)  + \updelta_{\upsigma} (v, t) \big), \qquad \upsigma = 1, 2. 
\end{align} Then 
\begin{align}\label{12eq: delta, 2}
	\begin{aligned}
		\frac {\partial^{n } \updelta_{\upsigma } (v, t) } {\partial t^n  } \Lt_{n,  \mathrm{Re} (v) ,   k }  \frac {\log^2 (|t|+2)} {|t|+1}  \cdot  \lp \frac { \log (|t|+2) } {|t|+1 } \rp^{n}   ,  
	\end{aligned}
\end{align} 
for any $n \geqslant 0$.  
 		 
	\end{lem}
	
Next, we refine the proof of Lemma \ref{lem: Bessel}, with \eqref{12eq: gamma}, \eqref{12eq: delta, 2} applied instead of  \eqref{4eq: bound for u, 2} (or \eqref{6eq: bound for G, 2}) to obtain a decomposition of the Bessel integral $H (x, y)$. Let us write $\vkappa = \upsigma v$ so as to avoid conflict of notation with $ v = xy/2 $ in \eqref{4eq: f(r,v,w)} (although $\vkappa$ was practically $ \upsigma \vepsilon$ in Definition \ref{defn: space T(r)}).   Note here that  after the change $t^{\vkappa} \ra (T + \varPi t)^{\vkappa}$ we also need to decompose  $$(T+\varPi t)^{\vkappa} = T^{\vkappa} + T^{\vkappa} \big( (1 +   {\varPi t} / {T}   )^{\vkappa} - 1 \big), $$ for $|\mathrm{Im} (\vkappa)| \leqslant \upsigma \log T$.   

\begin{lem}\label{lem: Bessel, refined} 
 We have
	 \begin{align}\label{12eq: H(x,y)}
	 	H (x, y) =   H^{\mathrm{o}} (x, y) +   \frac {\varPi \log T } {T}  H^{\flat} (x, y) + O (T^{-A}), 
	 \end{align}
 where  $H^{\mathrm{o}} (x, y)$ and $ H^{\flat} (x, y) $ are defined as the integral in {\rm \eqref{4eq: H(x,y) = I(v,w)}}{\rm;} more explicitly, 
 \begin{align}\label{12eq: Ho}
 	H^{\mathrm{o}} (x, y) = c (\vkappa) \cdot  \varPi T^{1+\vkappa}  \int_{-\varPi^{\vepsilon}/\varPi}^{\varPi^{\vepsilon}/ \varPi} 
   \exp (  2i  T r - \varPi^2 r^2) \cos  ( f (r; v, w) )   \nd r, 
 \end{align} 
for some $c (\vkappa) = O(1)$, and $  H^{\flat} (x, y) $ is of the same form with a certain Schwartz $g^{\flat} (\varPi r)$ instead of $\exp (- \varPi^2 r^2)$. 
\end{lem}

\subsection{The Method}\label{sec: method}
Recall from Definition \ref{defn: unsmooth} and Lemma \ref{lem: unsmooth, 1} that we need to change $T$ into $K$ and average over $K$ from $T-H$ to $T+H$.  
The key point is that for $ H^{\mathrm{o}} (x, y) $ as in \eqref{12eq: Ho}, we may integrate $K^{1+\vkappa} \exp (2i K r)$ at an early stage by 
\begin{align}\label{12eq: partial}
	\int_{T-H}^{T+H}  K^{1+\vkappa} \exp (2i K r) \nd K = \frac {K^{1+\vkappa} \exp (2i K r)} {2i r} \bigg|_{T-H}^{T+H} - \frac {1+\vkappa} {2i r} \int_{T-H}^{T+H}  K^{\vkappa} \exp (2iKr) \nd K. 
\end{align}
As $r$ occurs in the denominators, this is effective if $r$ is not extremely small, so it will be convenient to split the $r$-integral as in \S \ref{sec: split integral}: 
\begin{align}\label{12eq: split integral}
	\int_{-\varPi^{\vepsilon}/\varPi}^{\varPi^{\vepsilon}/ \varPi}  =  \int_{-T^{\vepsilon}/T}^{T^{\vepsilon}/ T} + \sum_{ T^{\vepsilon}/ T \Lt \tau \shskip \Lt \varPi^{\vepsilon}/ \varPi } \int_{\tau}^{2\tau} +  \int_{-2\tau}^{-\tau}   
\end{align}
for $\tau$ dyadic, and \eqref{12eq: partial} will only be used for those integrals on $\pm [\tau, 2\tau]$. After these manipulations, we may proceed to treat the off-diagonal terms as in \S \S \ref{sec: off, 1st} and \ref{sec: off}. However, the method does not  apply well to the Bessel integral $H_-(x, y)$.

\subsection{Further Improvement}\label{sec: unsmooth, 2}

Let us denote $ \SO_{1}^{\pm} (T, \varPi)  $ and $\SO_2 (T, \varPi)$ respectively the off-diagonal sums in \eqref{6eq: O(m)} and \eqref{7eq: O(m)} for $m = 1$ and $\boldsymbol{m} = \boldsymbol{1}$ (the Eisenstein contributions are relatively small in view of Lemma \ref{lem: Eis E1,2}).  

According to \eqref{12eq: H(x,y)}, after inserting the expression  of $ V_{\upsigma}     (  y ; t  ) $ as in \eqref{3eq: V(y;t), 2}, we split
\begin{align}\label{12eq: O1+}
	\SO_{1}^{+} (T, \varPi) & = \SO_{1}^{\mathrm{o}} (T, \varPi) + \frac {\varPi \log T } {T} \SO_{1}^{\flat} (T, \varPi) + O (T^{-A}), \\
\label{12eq: O2}	\SO_{2}  (T, \varPi) & = \SO_{2}^{\mathrm{o}} (T, \varPi) + \frac {\varPi \log T } {T} \SO_{2}^{\flat} (T, \varPi) + O (T^{-A}). 
\end{align}
Trivially, by Lemma \ref{lem: O1(m)} and Corollaries \ref{cor: O, 1}  and \ref{cor: O, 2}, 
\begin{align}\label{12eq: Of1}
	\frac 1 {\sqrt{\pi} \varPi} \int_{\, T- H}^{T + H}   \big|\SO_{1}^{\flat} (K, \varPi) \big| \nd K \Lt   \frac {H X T^{1+\vepsilon}} {\varPi} , 
\end{align}
\begin{align}\label{12eq: Of2}
	\frac 1 {\sqrt{\pi} \varPi} \int_{\, T- H}^{T + H}   \big|\SO_{2}^{\flat} (K, \varPi)\big| \nd K \Lt   \frac {H T^{5/2+\vepsilon}} {\varPi^2}, 
\end{align}
and, by Lemma \ref{lem: O1(m)-},   
\begin{align}\label{12eq: O-1}
	\frac 1 {\sqrt{\pi} \varPi} \int_{\, T- H}^{T + H} \big|\SO_{1}^{-} (K, \varPi)\big| \nd K  \Lt T^{\vepsilon} \frac {H T (\sqrt{T} + \varPi) } {\varPi \sqrt{X}}. 
\end{align}
By the arguments in \S \S \ref{sec: O+, 1}, \ref{sec: O+, 2}, \ref{sec: split integral}--\ref{sec: final} and the discussions in \S \ref{sec: method}, in particular \eqref{12eq: partial} and \eqref{12eq: split integral},  we infer that 
\begin{align}\label{12eq: Oo1}
	\frac 1 {\sqrt{\pi} \varPi} \int_{\, T- H}^{T + H} \SO_{1}^{\mathrm{o}} (T, \varPi) \nd K \Lt   X T^{1+\vepsilon}    ,
\end{align} 
\begin{align}\label{12eq: Oo2}
	\frac 1 {\sqrt{\pi} \varPi} \int_{\, T- H}^{T + H} \SO_{2}^{\mathrm{o}} (T, \varPi) \nd K \Lt   \frac  {T^{5/2+\vepsilon}}  {\varPi}   .    
\end{align} 
It is important that $\tau$ arises in the numerator of the bound  for  $ \mathrm{S}_{\tau} (b, g )$ as in \eqref{7eq: bound for S(b,g)}. 

At any rate, the rule is simple: $ \SO_{1}^{\mathrm{o}} (T, \varPi) $ and $ \SO_{2}^{\mathrm{o}} (T, \varPi) $ still have the same bounds as  $ \SO_{1}^{+} (T, \varPi) $ and $ \SO_{2}  (T, \varPi) $, respectively, even after the average procedure.

Thus, it follows from \eqref{12eq: O1+}, \eqref{12eq: Of1}, \eqref{12eq: O-1}, and \eqref{12eq: Oo1} that
\begin{align*}
	\frac 1 {\sqrt{\pi} \varPi} \int_{\, T- H}^{T + H} (\SO_{1}^{+} (K, \varPi) + \SO_{1}^{-} (K, \varPi)) \nd K \Lt T^{\vepsilon} \bigg(  T X + \frac {H T (\sqrt{T} + \varPi) } {\varPi \sqrt{X}} \bigg), 
\end{align*}
hence, on the choice $X = H^{2/3} (T^{1/3}/ \varPi^{2/3} + 1)$, 
\begin{align}
		\frac 1 {\sqrt{\pi} \varPi} \int_{\, T- H}^{T + H} (\SO_{1}^{+} (K, \varPi) + \SO_{1}^{-} (K, \varPi)) \nd K \Lt T^{\vepsilon} \frac {H^{2/3} T (T^{1/3} + \varPi^{2/3})} {\varPi^{2/3}} ,  
\end{align}
and it follows from \eqref{12eq: O2}, \eqref{12eq: Of2}, and \eqref{12eq: Oo2} that 
 \begin{align}
 	\frac 1 {\sqrt{\pi} \varPi} \int_{\, T- H}^{T + H}  \SO_{2} (K, \varPi)  \nd K \Lt  \frac  {T^{5/2+\vepsilon}}  {\varPi}  . 
 \end{align}

Similar to the arguments in \S \ref{sec: prelim}, if we apply Lemma   \ref{lem: unsmooth, 1} and  choose $\varPi = H^{2/5} T^{1/5} + H^{2/3}$ or $T^{3/4}$, then Corollaries \ref{cor: 1st} and \ref{cor: 2nd} are improved in the following way, so that Theorem \ref{cor: unsmoothed} is a  direct consequence.   

\begin{theorem} 
	Let $ T^{1/3+ \vepsilon} \leqslant H \leqslant T/3$. Then 
	\begin{align}\label{12eq: 1st moment}
		\sum_{ |t_j - T| \leqslant H }   \omega_j L (s_j,  \vQ \otimes u_j ) =  \frac {4  } {\pi^2}  H T + O \big(T^{\vepsilon} (H^{2/5} T^{6/5} + H^{2/3} T  ) \big). 
	\end{align}
\end{theorem}

\begin{theorem} 
	Let $ T^{3/4+\vepsilon} \leqslant H \leqslant T/3$. Then
	\begin{align}\label{12eq: 2nd moment}
		\sum_{ |t_j - T| \leqslant H }   \omega_j |L (s_j,  \vQ \otimes u_j )|^2  = \frac {4} {\pi^2} \int_{T-H}^{T+H} K  (\gamma_1 \log K + \gamma_0') \nd K + O  (   T^{7/4+\vepsilon}). 
	\end{align}
\end{theorem}

\appendix

\section{Stirling's Formulae} \label{sec: Stirling}

According to \cite[\S \S 1.1, 1.2]{MO-Formulas}, for   $ |\arg (s) | < \pi$, as $|s| \ra \infty$, we have 
\begin{align}\label{app1: Stirling, 1}
	\log \Gamma (s) = \lp s+\frac 1 2 \rp \log s - s + \frac 1 2 \log (2\pi) + O \lp \frac 1 {|s|} \rp, 
\end{align}
\begin{align}\label{app1: Stirling, 2} 
	  \psi (s) & = \log s - \frac 1{2s} + O \lp \frac 1 {|s|^2} \rp, \\
\label{app1: Stirling, 2.2}     \psi^{(n)} (s) & = (-1)^{n-1} \bigg(\frac { (n-1)!} {s^{n}} + \frac {n!} {2s^{n+1}} \bigg) + O \lp \frac 1 {|s|^{n+2}} \rp, 
\end{align}
where $\psi (s)  = \Gamma' (s) / \Gamma (s) $ is the logarithm derivative of $\Gamma (s)$. 

Let $ |\arg (s) | < \pi$ and $|s|$ be large.  A  useful consequence of \eqref{app1: Stirling, 1} is the following bound: 
\begin{align}\label{app1: Stirling, 3}
	\frac {\Gamma (s+ \valpha)} {\Gamma ({s} )} = O_{\mathrm{Re} (\valpha)}  \big( \big|s^{\valpha }  \big|  \big), 
\end{align}
provided  that  $|\valpha| \leqslant \sqrt{|s|}$, and, more generally, it follows from \eqref{app1: Stirling, 2}--\eqref{app1: Stirling, 3} that 
\begin{align}\label{app1: Stirling, 4}
	\frac {\nd^n} {\nd s^n }	\lp \frac {\Gamma (s+ \valpha)} {\Gamma ({s} )} \rp = O_{n, \mathrm{Re} (\valpha)}   \bigg(  \big|s^{\valpha }  \big|  \left(\frac {1 + |\valpha | } {|s|} \right)^n  \bigg)
\end{align}
holds for such $\valpha$. Further, if we write  
\begin{align}\label{app1: refined quot}
	\frac {\Gamma (s+ \valpha)} {\Gamma (s ) } = s^{\valpha} (1 + \Delta (s; \valpha) ) ,
\end{align}
then 
\begin{align}\label{app1: Delta}
	 \Delta^{(n)} (s; \valpha) = O_{n, \mathrm{Re} (\valpha)}   \bigg(      \frac {1 + |\valpha |^2 } {|s|^{n+1}}  \bigg),
\end{align}
provided  that  $|\valpha| \leqslant \sqrt{|s|}$.

	\section{Stationary Phase}

We record here  \cite[Lemma A.1]{AHLQ-Bessel}, a slightly improved version of  \cite[Lemma {\rm 8.1}]{BKY-Mass}.

\begin{lem}\label{lem: staionary phase}
	Let $\varww   \in C_c^{\infty} [a, b]$. Let  $f  \in C^{\infty} [a, b]$ be real-valued.  Suppose that there
	are   parameters $P, Q, R, S, Z  > 0$ such that
	\begin{align*}
		f^{(j)} (x) \Lt_{  j} Z / Q^{j}, \qquad \varww^{(n)} (x) \Lt_{  n} S / P^{n},
	\end{align*}
	for  $j \geqslant 2$ and $n \geqslant 0$, and
	\begin{align*}
		| f' (x) | \Gt R. 
	\end{align*}
	Then 
	\begin{align*}
		\int_a^b  e (f(x)) \varww (x)  \nd x \Lt_{ A} (b - a) S \bigg( \frac {Z} {R^2Q^2} + \frac 1 {R Q} + \frac 1 {R P} \bigg)^A  
	\end{align*} 
	for any  $A \geqslant 0$.
\end{lem}

According to \cite{KPY-Stationary-Phase}, let us introduce the notion of inert functions in a simplified setting.  

\begin{defn}\label{defn: inert}
	Let $\boldsymbol{I}  \subset \BR_+^{d}$ be a product of intervals {\rm(}not necessarily finite{\rm)}.  For $X \geqslant 1$, we say a smooth function $\varww \in C^{\infty} (\boldsymbol{I})$ is $X$-inert if 
	\begin{align*}
		\boldsymbol{x}^{\boldsymbol{n}} 	\varww^{(\boldsymbol{n})} (\boldsymbol{x})  \Lt_{\boldsymbol{n}} X^{| \boldsymbol{n} |}  , \qquad \text{($\boldsymbol{x} \in \boldsymbol{I}$),} 
	\end{align*} for every $\boldsymbol{n} \in \mathbf{N}_0^{d}$, where in the multi-variable notation $\boldsymbol{x}^{\boldsymbol{n}} = x_1^{n_1} \cdots x_d^{n_d}$, $ \varww^{(\boldsymbol{n})} (\boldsymbol{x}) = \varww^{(n_1, \cdots, n_d)} (x_1, \cdots, x_d) $, and $| \boldsymbol{n} | = n_1+ \cdots + n_{d}$. 
\end{defn}

Note that $X$ is relatively small  in practice: $X = \log T$ or $T^{\vepsilon}$. Also, some exponential factors could be $X$-inert if its phase is small.  For example,  $e (\lambda x^{\gamma})$ ($\gamma \neq 0$) is $X$-inert on $[\rho, 2 \rho]$ provided that $ \lambda \rho^{\gamma} \Lt X $. 

The next lemma follows easily by repeated partial integration, and it will be used to handle the Fourier or Hankel transform in certain degenerated cases. 

\begin{lem}\label{lem: integral, 0}
	Let $\gamma \neq 0$ be real. 	
	Let $\varww   \in C_c^{\infty} [\rho ,  2 \rho]$ be $X$-inert. Then the integral
	\begin{align*}
		\int_{\rho}^{2\rho} e (\lambda x^{\gamma} ) \varww (x ) \nd x 
	\end{align*}
	has bound  $O (\rho T^{-A})$ for any $A \geqslant 0$ in the case $ \lambda \rho^{\gamma} > T^{\vepsilon} X  $. 
\end{lem}

Let us record here  a generalization of 
the stationary phase estimate in \cite[Theorem 1.1.1]{Sogge}.

\begin{lem}\label{lem: stationary phase estimates}
	Let  $ \sqrt {\lambda} \geqslant X \geqslant  1$.  	Let $\varww (x, \lambda, \boldsymbol{x})  \in C^{\infty} ([a, b] \times [X^2, \infty) \times \boldsymbol{I})$ be $ X$-inert, with compact support in the first variable $x$. Let $f (x) \in C^{\infty} [a, b]$ be   real-valued.  Suppose    $f(x_0) = f'(x_0) = 0$ at a point  $  x_0 \in (a, b)$, with $ f'' (x_0) \neq 0$ and $f' (x) \neq 0$ for all $x \in [a, b] \smallsetminus \{x_0\} $. Define
	\begin{align*}
		I (\lambda, \boldsymbol{x}) = \int_a^b e (\lambda f(x)) \varww (x, \lambda, \boldsymbol{x})  \nd x,
	\end{align*}  
	then $\sqrt{\lambda} \cdot I (\lambda, \boldsymbol{x}) $ is an $X$-inert function.  
\end{lem}

\begin{proof}
	Note that if there were no variable $\boldsymbol{x}$, then this lemma is \cite[Theorem 1.1.1]{Sogge} in the case $X = 1$ and  \cite[Lemma 7.3]{Qi-GL(3)} in general.   However, since $ e ( \lambda f(x))$ does not involve $\boldsymbol{x}$, the derivatives for the added variable  $\boldsymbol{x}$ may be treated easily. 
\end{proof}

\begin{remark}
	 Of course, the main theorem in \cite{KPY-Stationary-Phase} is much more general than Sogge's {\rm\cite[Theorem 1.1.1]{Sogge}}, but the latter has a simpler proof and no error term.  
\end{remark}

 The next lemma is a simple application of Lemmas \ref{lem: staionary phase} and \ref{lem: stationary phase estimates}. 
 
 \begin{lem} \label{lem: analysis of integral}
 	Let  $ \gamma \neq 0,  1$ be real. 
 	For $ \sqrt {\lambda} \geqslant X \geqslant  1$ and $\rho > 0$, define 
 	\begin{align*}
 		I_{\gamma}^{\pm} (\lambda, \boldsymbol{x}) =   \int_{\rho}^{2\rho}  e \big(\lambda \big(x \pm \gamma   x^{1/\gamma} \big) \big) \varww (x, \lambda, \boldsymbol{x}) \nd x,  
 	\end{align*} 
 	for an $X$-inert function  $\varww (x, \lambda, \boldsymbol{x}) \in C^{\infty} ([\rho ,  2 \rho] \times [X^2, \infty) \times \boldsymbol{I} )$, with compact support in the first variable $x$. 
 	
 	{\rm \,(i)} We have
 	$$ I_{\gamma}^{\pm} (\lambda, \boldsymbol{x}) \Lt_A  \rho \cdot \bigg( \frac {  X  } {\lambda (\rho +\rho^{1/\gamma})}\bigg)^A  $$ 
 	for any value of $\rho$ in the $+$ case, or for $ \min \left\{ \rho/\sqrt{2}, \sqrt{2}/\rho \right\}   < 1 / 2 $ in the $-$ case.  
 	
 	{\rm(ii)} Define  
 	\begin{align*}
 		\varvv_{\gamma} (\lambda, \boldsymbol{x} ) =  e (  \lambda (\gamma -1) ) \cdot \sqrt{\lambda}   I_{\gamma}^{-} (\lambda, \boldsymbol{x}  ), 
 	\end{align*}  
 	then $\varvv_{\gamma} (\lambda, \boldsymbol{x} )$ is an $X$-inert function for any $1/ 2 \leqslant \rho/\sqrt{2} \leqslant 2 $. 
 \end{lem}

{Finally, we  record the second derivative tests from \cite[Lemmas A.2, A.3]{AHLQ-Bessel}, as variants of \cite[Lemma 5.1.3]{Huxley} and \cite[Lemma 4]{Munshi-Circle-III}. 

\begin{lem}\label{lem: 2nd derivative}
	Let $\varww   \in C_c^{\infty} [a, b]$ and $V$ be its total variation. 	Let $f \in C^{\infty} [a, b]$ be real-valued. If $f'' (x) \geqslant \lambda > 0$ on $[a, b]$, then 
	\begin{align*}
		\bigg|\int_a^b  e (f(x)) \varww (x)  \nd x  \bigg| \leqslant \frac {4 V} {\sqrt{\pi \lambda}} . 
	\end{align*}
\end{lem}

 \begin{lem}\label{lem: 2nd derivative test, dim 2}
	Let  $f \in C^{\infty} ([a, b] \times [c, d])$ be real-valued such that uniformly 
	\begin{align*}
		& \left|\partial^2 f / \partial x^2 \right| \Gt \lambda > 0, \hskip 15pt  \left|\partial^2 f / \partial y   ^2 \right| \Gt \rho > 0, \\
		& |\det f''|  = \left|\partial^2 f / \partial x^2 \cdot \partial^2 f / \partial y   ^2 - (  \partial^2 f / \partial x \partial y    )^2 \right| \Gt \lambda  \shskip \rho .
	\end{align*}
	Let $\varww \in C_c^{\infty} ([a, b] \times [c, d])$. Define 
	\begin{align*}
		V = \int_a^b \int_c^d \left|  \frac {\partial^2 \varww(x, y   )} {\partial x \partial y   } \right| \nd x \shskip \nd y   .
	\end{align*}
	Then
	\begin{align*}
		\int_a^b \int_c^d e (f(x, y   )) \varww (x, y   ) \nd x \shskip \nd y     \Lt \frac { V } {\sqrt { \lambda  \shskip \rho}},
	\end{align*}
	with an absolute implied constant.
\end{lem}

\delete{The next lemma is essentially \cite[Theorem 1.1.1]{Sogge} in the case $X = 1$ and  \cite[Lemma 7.3]{Qi-GL(3)} in general.

\begin{lem}\label{lem: stationary phase, main}
	Let $S >0$ and $ \sqrt {\lambda} \geqslant X \geqslant  1$.  	Let $\varww (x; \lambda)$ be  in $ C_c^{\infty} [ a, b] $ for all $\lambda$, and $f (x) \in C^{\infty} [a, b]$ be   real-valued. Suppose that  $   \lambda^{j} \partial_x^{n} \partial_\lambda ^{ j}  \varww  (x; \lambda) \Lt_{  n,   j } S X^{n + j}
	$ and that  $f(x_0) = f'(x_0) = 0$ at a point  $  x_0 \in (a, b)$, with $ f'' (x_0) \neq 0$ and $f' (x) \neq 0$ for all $x \in [a, b] \smallsetminus \{x_0\} $. Then
	\begin{align*}
		\frac {\nd ^{ j}} {\nd \lambda ^{ j}} \int_a^b e (\lambda f(x)) \varww (x; \lambda) \nd x \Lt_{  j } \frac {S X^j}  {  \lambda^{   1/2 +  j } }.
	\end{align*}
\end{lem}
}


\newcommand{\etalchar}[1]{$^{#1}$}
\def\cprime{$'$}

\end{document}